\documentclass[11pt]{article}
\usepackage{latexsym,color,amsmath,amsthm,amssymb,amscd,amsfonts,mathtools,extarrows,chemfig,mhchem,textcomp}
\usepackage{graphicx,tikz,pgfplots}\usetikzlibrary{arrows}
\usepackage[a4paper, left=1in, right=1in, top=1in, bottom=1in]{geometry}

\usepackage{float}
\usepackage{graphicx} 
\usepackage{hyperref,comment}
\usepackage{fullpage}
\usepackage[shortlabels]{enumitem}
\usepackage[mathscr]{euscript}
\usepackage{comment}
\usepackage{csquotes}

\usepackage{xypic,tikz}
\allowdisplaybreaks[4]

\newtheorem{theorem}{Theorem}
\newtheorem{lemma}[theorem]{Lemma}
\newtheorem{prop}[theorem]{Proposition}

\newtheorem{corollary}[theorem]{Corollary}
\newtheorem{defn}[theorem]{Definition}
\newtheorem{definition}[theorem]{Definition}

\newtheorem{remark}[theorem]{Remark}

\newtheorem{example}[theorem]{Example}

\newcommand{\polar}{\circ}
\newcommand{\R}{\mathbb{R}}

\newcommand{\dom}{\mathrm{dom}\,}
\newcommand{\gradientp}{\partial}
\newcommand{\al}{\alpha}
\newcommand{\bet}{\tau}
\newcommand{\acg}{\varphi}
\newcommand{\afc}{\psi}

\renewcommand{\le}{\leqslant}
\renewcommand{\ge}{\geqslant}
\renewcommand{\leq}{\leqslant}
\renewcommand{\geq}{\geqslant}
\renewcommand{\setminus}{\smallsetminus}
\renewcommand{\subset}{\subseteq}

\DeclareMathOperator*{\KA}{\scalebox{0.5}{$K(A)$}}
\DeclareMathOperator*{\RA}{\scalebox{0.5}{$R(A)$}}
\DeclareMathOperator*{\KAs}{\scalebox{0.5}{$K(A^*)$}}

\newcommand{\f}{f}
\newcommand{\g}{g}

\renewcommand{\nu}{\upnu}

\newcommand{\cvx}{\mathrm{Cvx}}
\newcommand{\cvp}{\mathrm{CV}}

\newcommand{\h}{\mathcal{H}}

\newcounter{c}
\title{
Hidden critical and Morse equivalence behind duality:\\ Theory and Applications
}
\allowdisplaybreaks[4]
\author{Dong Zhang\thanks{School of Mathematical Sciences,  Peking University,   100871 Beijing, China. {\tt dongzhang@math.pku.edu.cn} \\
Dong Zhang is supported by grants from  the  National Natural Science Foundation of China (No.\ 12401443).}\footnotemark[1]}
\date{}

\begin{document}

\maketitle 
 \allowdisplaybreaks[4]
\begin{abstract}
The aim of this paper is to establish critical duality theory for \emph{ratios of nonnegative homogeneous convex functions} (shorten for \emph{RC functions}) and \emph{differences of convex functions} (abbreviated as \emph{DC functions}) on Banach spaces, by which, we reveal how some seemingly completely different problems from critical point theory, combinatorial geometry, and spectral graph theory, can be identified with each other. These results open up new perspectives for understanding classic problems and concepts such as eigenvalue problems, Cheeger constants, zonotopes, and hypergraphs. 

Specifically‌, we establish a series of duality results on critical point theory and Morse theory for RC functions, including the homotopy type of sublevel sets, the Morse critical points and their Rothe critical groups,  Lagrange critical points and their multiplicities, Lusternik-Schnirelman min-max critical values, Poincare polynomials, as well as the structure of handlebody decompositions, all of which are proved to be preserved under polarity dual. 
Moreover, we obtain the first critical duality theory of DC functions which does not depend on the DC decomposition. 
This answers a question left open from the works of Toland on DC functions and the works of Le-Pham on DC programming. 

We apply these results to establish various equalities involving exact and coexact Cheeger constants on polyhedrons and manifolds; we provide a reformulation of the graph Cheeger constant using zonotopes; we introduce the contact data which serves as a geometric characterization of Lagrange criticality; and we show that the eigenproblems for 1-Laplacian and $\infty$-Laplacian on hypergraphs are equivalent to the contact problems of zonotopes, which indeed establishes a new characterization of zonotopes. 
Moreover, based on our duality theory, we derive the first necessary and sufficient condition for a general convex body to be a zonotope; and we propose a zonotope-representation theorem, which states that the contact data of a centrally symmetric convex polytope and a zonotope can be expressed as the contact data between two zonotopes. 
We also prove a duality equivalence for certain nonlinear eigenvalue problems and for certain bifurcation problems. 
Our study here reveals an intricate interaction of critical point theory with other fields such as convex analysis, combinatorial geometry, and nonlinear eigenproblems on graphs. 

\vspace{0.2cm}

\end{abstract}

\tableofcontents

\section{Introduction}

The concept of duality, as a central theme, has long been the focus of mathematicians from various fields. There are two classical duals for convex functions, one is the Fenchel dual, and the other is the polarity dual. 
This paper is concerned with critical dual equivalence, a research direction originated from a fundamental and classical paper of Clarke. Earlier works along these lines include those by many prominent mathematicians working both on  analysis and geometry. 

Research in Fenchel duality theory has permeated almost all areas of analysis, geometry and applied mathematics. 
In the 1970s, Clarke gave a vitally important observation showing that there is a natural one-to-one correspondence between the critical points of the action functional and the dual action functional 
\cite{Clarke75,Clarke90}. 
Subsequently, Ekeland gave a general approach to Clarke's duality \cite{Ekeland90}. 
In 2022, Abbondandolo and Kang proved an elegant theorem which states that the Floer complex that is associated with a convex Hamiltonian function on $\R^{2n}$ is isomorphic to the Morse complex of Clarke’s dual action functional that is associated with the Fenchel-dual Hamiltonian \cite{AbbondandoloKang}. This extends the deep works by  Abbondandolo and Schwarz which construct an isomorphism between the Floer homology of a convex quadratic-growth Hamiltonian on cotangent bundles and the Morse homology of its Fenchel-dual Lagrangian action functional \cite{Abbondandolo15,Abbondandolo06}. 
For the study of  variational problems on DC functions, Toland \cite{dual-Toland79} and Singer \cite{Singer79} established a duality theorem which extends the concept of duality to a broader class of non-convex optimization problems \cite{dual-Toland79,Singer79}, namely, DC programming, providing global optimality conditions and characterizing the global optimal solutions of these problems. 
It is worth noting that 
a Toland type nonconvex duality principle for certain Lagrangian was  rigorously proved by Ricciardi and Suzuki \cite{Ricciardi}. 
Recent remarkable developments on Fenchel dual can be found in different fields. 
In particular, 
Bergmann,  Herzog, Louzeiro, Tenbrinck and Vidal-Núñez \cite{Bergmann} introduce a new notion of a Fenchel conjugate, which generalizes the classical Fenchel conjugation to functions defined on Riemannian manifolds. 
In 2022, Bergmann,  Herzog and Louzeiro \cite{Mauricio}  introduce a definition of Fenchel conjugate and Fenchel biconjugate on Hadamard manifolds based on the tangent bundle. 
In 2023, Gutman and Peña \cite{Gutman} show 
an interesting connection between two major threads in convex optimization, namely first-order methods and Fenchel duality.

Polarity dual also plays a crucial role in analysis and geometry. 
The concept of polar dual for convex sets \cite{Artstein04,Boroczky,Gardner} is of great importance not only in geometry, but also in other fields. As a remarkable example of this, we would like to point out that in 2022, Atamtürk and Narayanan \cite{Atamturk} apply polarity dual to combinatorial optimization and give an outer polyhedral approximation for the epigraph of set functions, which is a novel and important perspective. 
In the field of convex optimization, polar dual for convex sets also plays a significant role (both conceptually and methodologically) \cite{Friedlander,Aravkin,Freund}. 
Rockafellar introduced the polarity dual for functions in his celebrated book \cite[Chapter 15]{Rockafellar}, which has been recently rediscovered and further systematically investigated by Artstein-Avidan and Milman \cite{Artstein10,Milman11}, as well as Artstein-Avidan and Rubinstein \cite{Artstein17}. The polarity dual leads to a novel functionalization of the Brunn-Minkowski inequality proposed by Artstein-Avidan,  Florentin and Segal \cite{Artstein20}. 

In this paper, we  investigate the  duality of criticality of RC functions on Banach reflexive spaces. As a generalization of Rayleigh quotients, RC functions appear 
extensively in many fields, either as a nonlinear Rayleigh functional or as an auxiliary function used to define various quantities: 
e.g., the Cheeger constants on manifolds or discrete spaces, variational eigenvalues of nonlinear eigenproblems \cite{Burger,JMZ-book} and bifurcation problems \cite{Ekeland90}. 
In fact, the above important quantities and concepts, can be interpreted as \emph{critical data} (i.e., critical values or critical points) of the ratio $f/g$ of two particular convex functions $f$ and $g$. 
Our main intension in this paper is to establish a strong dual variational principle for $f/g$. 
We achieve this goal by exploring a systematic critical duality theory for RC functions in terms of polarity duality. 
In such duality theory, we focus primarily on three classical criticalities: Lagrange criticality, Morse criticality and Lusternik-Schnirelman  min-max criticality, which play key roles in critical point theory. Below, we provide a brief introduction to these  criticalities.

As a special homotopy invariant, the Lusternik-Schnirelmann category was created in order to provide estimates on the number of critical points for any smooth function on the manifold. While the Lusternik-Schnirelmann theory was mainly used in topology and analysis, it had far-reaching consequences in geometry as well, such as the well-known results on minimal hypersurfaces by Marques and Neves \cite{Marques}. 
Moreover, in the study of nonlinear eigenvalues on $p$-Laplacians, a sequence of critical values of the $p$-Rayleigh quotient can be described by means of the minimax variational principle of Lusternik-Schnirelman type, which leads to many useful results. See Section \ref{sec:def-minmax} for a more detailed mathematical explanation.

Morse theory \cite{Morse} enables us to analyze the topology of an object $M$ by studying  functions  on $M$. 
In the classical case, $M$ is a manifold and the function is generic and differentiable. 
There are, however, many extensions of Morse theory  in modern mathematics that do not   require a smooth structure,  such as
the metric and topological Morse theory 
\cite{Degiovanni94,Ioffe,Katriel,Degiovanni11}, the PS (piecewise smooth) or stratified Morse theory by Thom, Goresky and MacPherson \cite{Goresky}, the PL  Morse theory by Banchoff \cite{Banchoff67} and K\"uhnel \cite{Brehm-Kuhnel87,Edelsbrunner-Harer10}
, as well as the discrete Morse theory by Forman \cite{Forman98
}. 
In all such cases, a typical function on $M$ will reflect the topology quite directly,  allowing one to find CW structures and handle decompositions on $M$ and to obtain  information about their homology.  
According to modern Morse theory, the core of the entire Morse structure can be reduced to the study of Morse critical points which we will explain in Section \ref{sec:def-Morse}. 

The Lagrange criticality\footnote{See Section \ref{sec:def-Lagrange} for details.} comes from calculus of variations as well as optimization problems, which gives necessary conditions for the critical point theory of functions. 
For example, for one-homogeneous convex nonnegative functions $f$ and $g$, the critical point problem of a function $f$ under the constraint $g=1$ can be reduced to the Lagrange critical point satisfying $\partial f(x)\cap\lambda \partial g(x)\ne\varnothing$, which leads to certain nonlinear eigenvalue problems, see the systematic works by Burger \cite{Burger} and Bungert and Korolev \cite{Bungert-CAMS}. 
In this setting, both Lusternik-Schnirelman min-max critical values and Morse critical points of $f$ on $\{x:g(x)=1\}$ satisfy the Lagrange criticality. 
It is known that the three types of criticality have the following  implication relation: 
$$\text{Lusternik-Schnirelman  min-max criticality}\Longrightarrow\text{Morse criticality}\Longrightarrow\text{Lagrange criticality}$$
We progressively delve deeper and give dual equivalence theorems for the above three types of critical values for RC functions of the form $f\circ A/g$, where $A$ is a bounded linear operator, $f$ and $g$ are homogeneous nonnegative convex functions on Banach spaces. 
It is worth noting that in our setting, $f$ and $g$ are quite general, which allow to take the value of infinity and do not require twice G$\hat{\text{a}}$teaux-differentiable.  
Dealing with such general functions presents some analytical difficulties, which we will mention in due time.

Another point we would like to highlight is that we provide two polarity dual versions of $f\circ A/g$, one is  $g^\polar/(f\circ A)^\polar$; and the other is  $g^\polar\circ A^*/f^\polar$.  
These lead to two analogues of the critical duality theory, which have significantly enriched duality theory and broadened its scope of application. 
For example, Cheeger constants in geometry and graph theory, zonotopes in combinatorics, as well as hypergraph $\infty$-Laplacian in nonlinear eigenvalue problems, these three problems, which originate from different fields and were originally unrelated, are linked through the duality theory discussed in this paper. We present some 
essential discoveries, including: (1) a geometric characterization of Lagrange criticality using convex geometry; (2) a characterization of Cheeger constant using zonotopes; (3) an equivalence among the eigenproblems of hypergraph 1-Laplacian and $\infty$-Laplacian as well as the contact problems of zonotopes; (4) a zonotope contact representation for origin-symmetric convex polytopes; (5) the first necessary and sufficient condition for a general convex body to be a zonotope; (6) an equivalence between certain primal eigenvalue problems and bifurcation problems and their duals; (7) various equalities involving exact and coexact Cheeger constants on polyhedral  manifolds.  
These results provide a dictionary between Cheeger problems, nonlinear eigenvalue problems, zonotopes and convex bodies. 
The translations from one object to another give new tools for attaching problems and drawing out connections, and also open up new perspectives for understanding and researching classic concepts such as Cheeger constants, zonotopes, and hypergraphs.


As an extension of the idea used in the proof of our duality theory for RC functions, we also develop a critical duality theory for general DC functions using Fenchel dual. We refer to Examples \ref{exam:f-g-cannot-polar}
and \ref{exam:f/g-cannot-Fenchel} in Appendix to explain why the  suitable duality used for RC functions is the polarity dual, while the appropriate dual for DC functions is the Fenchel dual. 
We summarize these as shown in the diagram below:   
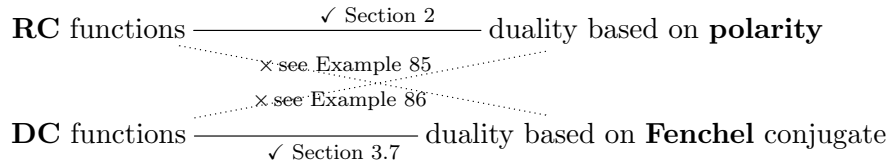
\[
\xymatrix{
  \text{{\bf RC} functions} \ar@{-}[rrr]^{\checkmark \text{ Section \ref{sec:ratio}}}\ar@{..}[drrr]^{\text{\texttimes\,see Example \ref{exam:f/g-cannot-Fenchel}\;\;\;\;\;\;\;\;\;\;\;\;\;\;\;\;\;\;\;\;}}
  & & & \text{duality based on {\bf polarity}} \ar@{..}[dlll]^{\text{\texttimes\,see Example \ref{exam:f-g-cannot-polar}}\;\;\;\;\;\;\;\;\;\;\;\;\;\;\;\;\;\;\;\;\;\;} \\
  \text{{\bf DC} functions} & & & \text{duality based on {\bf Fenchel} conjugate}  \ar@{-}^{\checkmark\text{ Section \ref{sec:dc}}\;\;\;\;\;\;\;\;\;\;\;\;\;\;}[lll]
}
\]


\subsection{Main results}\label{sec:main-result}

Our main findings on complete critical equivalence for RC functions under duality are collected in Theorems \ref{th:fgmain:polar} and \ref{th:fAgmain:polar}. 
Basic to the theory is the notion of critical points and other related concepts, and we now discuss the notation to be used. 
 Let $X$ be a reflexive Banach space, let   
$$\cvx(X):=\big\{\text{proper, lower semi-continuous, and convex functions }f:X\to \R\cup\{+\infty\}\},$$
$$\cvx_0(X)=\{f\in \cvx(X): f(x)\ge f(0)=0,\forall x\in X\big\},$$
and
$$\cvx_0^p(X)=\{f\in \cvx_0(X): f\text{ is positively $p$-homogeneous}\},$$
where $p\ge1$, $\dom f:=\{x\in X:f(x)<+\infty\}$, and we say that $f$ is positively $p$-homogeneous if $f(tx)=t^pf(x)$, $\forall t>0$, $\forall x\in X$. 

For $f\in \cvx_0(X)$, let $f^\polar:X^*\to\R\cup\{+\infty\}$ be the polarity dual of $f$, defined by 
\begin{equation}\label{eq:original-A}
\f^\polar( x^*):
=\inf\left\{c\in\R:\langle  x^*, x\rangle\le c\f( x)+1,\forall x\in X \right\}    
\end{equation}
where $x^*\in X^*$. We say that $f$ is positive-definite if $f(x)>0$, $\forall x\ne0$. 
Let 
\begin{equation}\label{eq:defn-cvxc}
\cvp_c^p(X)=\{f\in \cvx_{0}^p(X):\text{ both }f\text{ and }f^\polar\text{ are continuous and positive-definite}\}.
\end{equation}
We then have the inclusion relation $$\cvp_c^p(X)
\subset \cvx_{0}^p(X)\subset \cvx_{0}(X),$$
and moreover, the polarity transform restricted on each of the above function classes on $X$ maps onto the corresponding function class on $X^*$, see Corollary \ref{cor:function-spaces}.

To clearly present the main results, we shall first introduce several concepts of ‘critical equivalence’.

\begin{defn}[Min-max equivalence]
Two 
continuous  functions $F_1:\Omega_1\to\R$ and $F_2:\Omega_2\to\R$ are said to be \emph{min-max equivalent} if 
the nontrivial Lusternik-Schnirelman min-max critical values of $F_1$ coincide with that of $F_2$. Here the formal definition of Lusternik-Schnirelman min-max critical values can be found in Section \ref{sec:def-minmax}.
\end{defn}

\begin{defn}[Strong Morse equivalence]
Two 
continuous functions $F_1:\Omega_1\to\R$ and $F_2:\Omega_2\to\R$ are \emph{strong Morse equivalent} if they 
satisfy: \begin{enumerate}[({C}1)]
\item Homotopy of level sets: for any $c\in\R$, the sublevel set of $F_1$ at level $c$, i.e., $\{x\in\Omega_1:F_1(x)\le c\}$, is homotopy equivalent to the sublevel set of $F_2$ at the same level, i.e., $ \{x\in\Omega_2:F_2(x)\le c\}$.

\item Morse criticality: 
The nontrivial Morse critical values of $F_1$ coincide with that of $F_2$. And there is a natural 1-to-1 correspondence  $\xi$ between the nontrivial Morse critical points of $F_1$ and that of $F_2$, such that $x$ is a nontrivial Morse critical point of $F_1$ if and only if $\xi(x)$ is a nontrivial Morse critical point of $F_2$, and $F_1(x)=F_2(\xi(x))$ is a nontrivial Morse critical value of both $F_1$ and $F_2$. 
Here, the definition of Morse critical values and Morse critical points can be found in Section \ref{sec:def-Morse}.

\item Rothe critical groups: 
the Rothe critical groups of $F_1$ at $x$ is isomorphic to that of $F_2$ at $\xi(x)$ whenever $x$ is a nontrivial Morse critical point of $F_1$, where $\xi$ is described in (C2).
\end{enumerate}
\end{defn}

The two definitions above are concerned about certain equivalences between two continuous functions, while the definition of Lagrange equivalence below considers pairs of lower semi-continuous convex functions.
\begin{defn}[Lagrange equivalence]\label{defn:Lagrange-intro}
Two function-pairs $(f_1,g_1)$ and $(f_2,g_2)$, in which $f_1,g_1,f_2,g_2$ are lower semi-continuous convex functions, are \emph{Lagrange equivalent} if they satisfy: 
\begin{itemize}
\item[(L1)] 
The nontrivial Lagrange critical values of $(f_1,g_1)$ coincide with that of $(f_2,g_2)$.
\item[(L2)] 
There exist two set-valued maps $T_1$ and $T_2$ such that for any nontrivial  Lagrange critical point $x$ of $(f_1,g_1)$, every $y\in T_1(x)$ is a nontrivial  Lagrange critical point of $(f_2,g_2)$; and conversely,  for any nontrivial  Lagrange critical point $y$ of $(f_2,g_2)$, every $x\in T_2(y)$ is a nontrivial  Lagrange critical point of $(f_1,g_1)$; as well as  $x\in T_2(T_1(x))$ and $y\in T_1(T_2(y))$. 
See Section \ref{sec:def-Lagrange} for the definition of Lagrange critical values/points.
\end{itemize}    
\end{defn}

With the above prepared definitions in hand, we are in a position to state the main theorem. 
Let $X$ and $Y$ be reflexive Banach spaces, and let 
$ \mathcal{B}(X,Y)$  denote the set of all bounded linear operators from $X$ to $Y$. 
The main result of this paper is the establishment of the following critical duality theorems. 


\begin{theorem}\label{th:fgmain:polar}
Suppose $f,g\in \cvx_{0}^p(X)$. Then, the function-pairs $(f,g)$ and $(g^\circ,f^\circ)$ are Lagrange equivalent. 
If we further assume that $f,g\in\cvp_c^p(X)$, then $f/g$ and $g^\polar/f^\polar$ are min-max equivalent. 
If $f$ and $g^\circ$ are further assumed to be $C^1$-smooth (or, $g$ and $f^\circ$ are $C^1$-smooth), then $f/g$ and $g^\polar/f^\polar$  are strong Morse equivalent. 
\end{theorem}

\begin{theorem}\label{th:fAgmain:polar}
Suppose $f\in \cvx_{0}^p(Y)$, $g\in \cvx_{0}^p(X)$,  $A\in \mathcal{B}(X,Y)$, as well as $g^\polar$ and $f$ are continuous (or, $f^\polar$ and $g$ are continuous). Then, the function-pairs $(f\circ A,g)$, $(g^\circ\circ A^*,f^\circ)$, $(g^\circ,(f\circ A)^\circ)$ are Lagrange equivalent. If we further assume that $X$ and $Y$ are Hilbert spaces, $f\in\cvp_c^p(Y)$,  $g\in\cvp_c^p(X)$,  and $A:X\to Y$ is a Fredholm operator
, 
then $f\circ A/g$ and $g^\circ\circ A^*/f^\circ$ are min-max equivalent. 
Moreover, if $f$ and $g^\circ$ are further assumed to be $C^1$-smooth (or, $g$ and $f^\circ$ are $C^1$-smooth), then $f\circ A/g$ and $g^\circ\circ A^*/f^\circ$ are strong Morse equivalent. 
\end{theorem}


In proving Theorems \ref{th:fgmain:polar} and \ref{th:fAgmain:polar}, we have to overcome many different difficulties, for example, (i) the feasible domain of a RC function is usually not a convex set, which makes many approaches from convex analysis ineffective; (ii) 
the function can be non-differentiable, which prevents methods from smooth analysis; (iii) 
the Lagrange critical values, which are particularly associated with the RC functions, involve an Euler-Lagrange equation. 
The entire proof needs to make full use of non-smooth analysis. 
This is why we devote the whole of Section \ref{sec:ratio} to proving Theorem \ref{th:fAgmain:polar}. 
As an application of our idea, we obtain a critical duality theory for DC functions, see Section \ref{sec:dc} for details. We refer to Section \ref{sec:first-c-dual} (resp., Section \ref{sec:second-dual-bounded-ope}) for the detailed explanation and proof of Theorem \ref{th:fgmain:polar} (resp., Theorem \ref{th:fAgmain:polar}), and for basic knowledge and fundamental properties on critical point theory and duality theory regarding RC functions, see 
Section \ref{sec:polar/RC}. 


We would like to call the phenomenon displayed in the duality result (Theorem \ref{th:fgmain:polar} or Theorem \ref{th:fAgmain:polar}) in this paper a \textbf{critical duality equivalence}, which can be summarised briefly in Figures \ref{fig:explain-critical} and \ref{fig:value}.

\begin{figure}[H]
    \centering

\begin{tikzpicture}
      \draw[thick] (-3.2,1) -- (2.6,1)-- (2.6,-1)--(-3.2,-1)--(-3.2,1);
\draw[purple, thick] 
(0,1)-- (0,-1);
\node at (1.3,0) {critical values};
\node at (-1.6,0) {\color{purple}critical points};
\node at (-0.2,1.5) {\color{blue}handlebody decompositions}
;
\draw[-,thick, blue] (-3.2,1) to[out=60, in=120] (2.6,1);
\node at (-0.3,3-0.2) {Morse \& Criticality of $(f\circ A,g)$};

\draw[thick] (-3.2+9,1) -- (2.6+9,1)-- (2.6+9,-1)--(-3.2+9,-1)--(-3.2+9,1);
\draw[purple, thick] 
(9-0.6,1)-- (9-0.6,-1);
\node at (-1.9+9,0) {critical values};
\node at (1+9,0) {\color{purple}critical points};
\node at (-0.2+9,2.1-0.6) {\color{blue}handlebody decompositions};
\draw[-,thick, blue] (-3.2+9,1) to[out=60, in=120] (2.6+9,1);
\node at (0+8.3,3-0.2) {Morse \& Criticality of $(g^\polar\circ A^*,f^\polar)$ or $(g^\circ,(f\circ A)^\circ)$};

\draw[<->,thick, black,dashed] (1.9,1) to[out=15, in=165] (7.1,1);
\draw[<->,thick, blue,dotted] (1,2.3) to[out=5, in=175] (8,2.3);

\draw[<->,thick,purple,dashed] (-0.6,-0.92) to[out=-10, in=-170] (9.6,-0.92);
\end{tikzpicture}    \caption{Description of the invariant of Morse theory and the entire critical data under duality: For a RC function $f\circ A/g$, the critical points and critical values of $f\circ A/g$ are one-to-one correspondence to that of $g^\circ\circ A^*/f^\circ$ (and $g^\circ/(f\circ A)^\circ$), and moreover, their  handlebody decompositions are isomorphism. }
    \label{fig:explain-critical}
\end{figure}
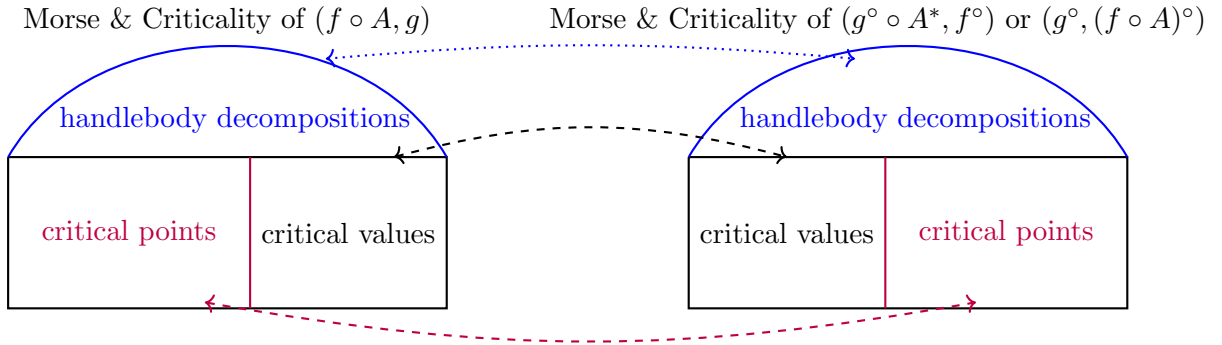

Moreover, Theorem \ref{th:fAgmain:polar} also reveals  that a refined set of different types of critical points and critical values is preserved under duality, which we will show in Figure \ref{fig:value}. 

\begin{figure}[H]
    \centering

\begin{tikzpicture}
\draw (2,0)--(-2,0);
\draw (2+6,0)--(-2+6,0);
    \draw (2,0) 
arc[start angle=0, end angle=180, x radius=2cm, y radius=4cm];    
      \draw[thick] 
      (1.8,0) 
arc[start angle=0, end angle=180, x radius=1.8cm, y radius=3cm]; 
\draw[red, thick] (1.6,0) arc[start angle=0, end angle=180, x radius=1.6cm, y radius=1.9cm]; 
\draw[red, thick] (0+6+1.6,0) arc[start angle=0, end angle=180, x radius=1.6cm, y radius=1.9cm]; 
    \draw (2+6,0) 
arc[start angle=0, end angle=180, x radius=2cm, y radius=4cm];    
\draw[thick] 
      (1.8+6,0) 
arc[start angle=0, end angle=180, x radius=1.8cm, y radius=3cm]; 
\node at (0,2.3) {Morse};
\node at (0,1.19) {\color{red}min-max};
\node at (0,3.3) {Lagrange};\node at (-0.5,4.69) {
Various types of critical pairs};
\node at (-0.5,4.19) {of {\large $f\circ A/g$}};

\node at (0+6,2.3) {Morse};
\node at (0+6,1.19) {\color{red}min-max};
\node at (0+6,3.3) {Lagrange};\node at (0+6+0.5,4.69) {Various types of critical pairs};
\node at (0+6+0.5,4.19) {of {\large $g^\polar\circ A^* /f^\polar $} (or $g^\polar/(f\circ A)^\polar$)};
\draw[<->, black] (1,3.5) to[out=25, in=155] (5,3.5);
\draw[<->,thick, black] (1,2) to[out=25, in=155] (5,2);
\draw[<->,thick, red] (1,1) to[out=25, in=155] (5,1);
\end{tikzpicture}    \caption{Description of the invariant of the refined critical pairs (critical points and critical values) under duality:  For a RC function $f\circ A/g$, the Lagrange (resp., Morse, min-max) critical points and critical values are preserved under polar duality. Moreover, the indexes and critical groups of critical points, 
as well as the multiplicities of critical values are also invariant. 
 }
    \label{fig:value}
\end{figure}
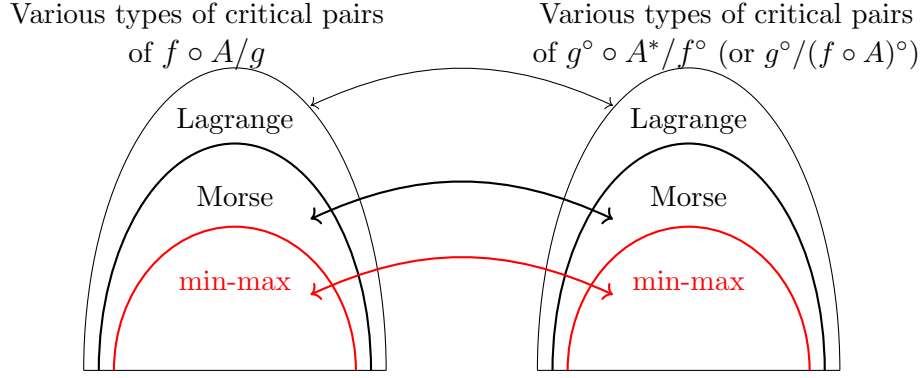

Theorems \ref{th:fgmain:polar} and  \ref{th:fAgmain:polar} have broad applications in various fields, including manifolds, polyhedrons, hypergraph $p$-Laplacians, zonotopes, nonlinear eigenvalue problems and bifurcation problems.  We list some of them 
as follows, for more results, see Section \ref{sec:app}.
\begin{itemize}
\item 
Exact and coexact Cheeger constants on polyhedral/smooth manifolds

We establish various equalities linking exact and coexact Cheeger constants on polyhedral manifolds and smooth manifolds. 
We refer to Section \ref{sec:Cheeger-dual} for details.

\item A geometric interpretation of Lagrange criticality

We introduce the contact data of convex bodies to establish a geometric reformulation of the Lagrange critical points and values for RC functions. See Section \ref{sec:geometric-inter}.
\item Hypergraph 1-Laplacian and $\infty$-Laplacians as new characterizations of zonotopes

We present a new characterization of zonotope: a convex body is a zonotope, if and only if the subdifferential of its support function equals the 1-Laplacian of some generalized hypergraph, if and only if the subdifferential of its Minkowski functional equals the $\infty$-Laplacian of a certain generalized hypergraph. 
A remarkable fact is that we present the necessary and sufficient conditions for a general convex body to be a zonotope. While, the other known equivalent characterizations of zonotopes, are only limited to the case of convex polytopes. 

In addition, we give an equivalent characterization of Cheeger's constant using zonotope. We prove that the contact data between a hypergraphic zonotope and a hypercube is equivalent to the nonzero eigenvalues of hypergraph 1-Laplacian as well as that of the dual hypergraph $\infty$-Laplacian. 
See Section \ref{sec:zonotope-1&infty}. 

\item Zonotope contact representation theorem 

We prove that the contact data of a centrally symmetric convex polyhedron and a zonotope can always be expressed in terms of the contact data between a zonotope and a hypercube. See Section \ref{sec:zonotope-representation}.


\item Nonlinear eigenproblems and bifurcation problems

Many nonlinear eigenvalue problems and bifurcation problems are indeed derived from a variational problem of the ratio form $f/g$, that is, the Lagrange critical data of a RC function, so that our results in this paper can be applied directly. We show that 
some useful nonlinear eigenproblems and bifurcation problems are equivalent to their dual versions. 
We also show that the relaxed Dinkelbach scheme is equivalent to the nonlinear power iteration for solving extreme eigenvalues, see Section \ref{sec:eigen&bifurcation}.
\item A critical duality theory for DC functions

We establish a series of duality results on critical point theory and Morse theory for DC functions, including the homotopy type of sublevel sets, the Morse critical points and their Rothe critical groups,  Lusternik-Schnirelman min-max critical values, Poincare polynomials, as well as the structure of handlebody decompositions, all of which are proved to be preserved under Fenchel dual.
This is actually the first critical duality theory of DC functions that does not depend on the DC decomposition, which answers a question left open from the works of Toland on DC functions and the works of Le Thi and Pham Dinh on DC programming.  See Section \ref{sec:dc}. 

\end{itemize}


\subsection{Supplementary 
}

Chapter 6 in the monograph \cite{JMZ-book} studies 
the ratio of two Lov\'asz extensions. Since the Lov\'asz extension (or Choquet integral) of the submodular function must be one-homogeneous and convex, the critical data of the ratio of two submodular functions can be transformed to that of the ratio of two one-homogeneous convex functions, and then the critical duality theory in the present paper can be applied 
subsequently. 
Although not directly relevant to the present article, we mention that the hypergraph Morse theory based on Lov\'asz extension \cite{Jost/Zhang21a} has also received attention recently. 






\section{Theory and Framework}
\label{sec:ratio}

\subsection{Preliminaries 
}
\label{sec:polar/RC}


Let $X$ be a reflexive Banach space. 
In this section, we recall the basic properties of polarity transform of a convex function $\f\in \cvx_0(X)$  which is defined as \eqref{eq:original-A}, and provide further results on particular function classes which would be useful. 


We note that when $f\in\cvx_0(X)$ is further assumed to be one-homogeneous, then  $$f^\polar(x^*)=\inf\left\{c\in\R_+:\langle  x^*, x\rangle\le c\f( x),\forall x\in X \right\}   
$$ 
and $f^\polar\in \cvx_0(X^*)$. 
For a convex set $K\subset X$, we use $$K^\polar=\{x^*\in X^*:\langle x^*,x\rangle\le 1,\,\forall x\in K\}$$ to denote the polar dual of $K$. 
By Riesz's representation theorem, if $X$ is a Hilbert space (e.g. $X=\R^n$) over the real field, then $X^*$ can be identified with $X$ and the dual action pair $\langle\cdot,\cdot\rangle$ reduces to the inner product on $X$. 
Therefore, our setting includes the infinite-dimensional spaces and also spaces without inner product, which is a generalization of the setting of Rockafellar \cite{Rockafellar},  Artstein-Avidan and Milman \cite{Artstein10,Milman11}, as well as Artstein-Avidan and Rubinstein \cite{Artstein17}. 


\begin{definition}[subdifferential \cite{Rockafellar,Clarke,Bauschke}]\label{def:eigenproblem}

For any $f\in\cvx(X)$, 
we use $\partial f(x)$ to denote the subdifferential of $f$ at $x$, which is defined by
$$ \partial f(x)=\left\{\xi \in X^* \;|\; \langle \xi, v \rangle \leq f(x+v)-f(x)\,,\;\forall v \in X \right\}.$$
\end{definition}
The subdifferential of a convex function $f$ is a set-valued map $\partial f:X\rightrightarrows X^*$. Below, we present a lemma which will be used latter. Since such lemma is known for $X=\R^n$, we put its proof in Appendix under a general setting.  

\begin{lemma}\label{lem:normal-cone}
Given a convex set $K\subset X$ with $0$ lies in the relative interior of $K$, let $\mathrm{NC}_x(K):=\{x^*:\langle x^*,y-x\rangle\le 0,\forall y\in K\}$ denote the normal cone of $K$ at $x$. 
Then  $\mathrm{NC}_x(K)=\mathrm{cl}\,\mathrm{cone}(\partial f_K(x))$ where $f_K$ denotes the Minkowski functional of $K$.  
\end{lemma}

\subsubsection{The function space $\cvx_0^p(X)$ 
}
We study the function space $\cvx_0^p(X)$ introduced in Section \ref{sec:main-result}. 
Here and in the following, we equivalently write $ \ker f$ and $f^{-1}(  0)$ to denote the set $\{  x : f(  x) =   0\}$, while we let $f^{-1}(  y)=\{  x:f(  x)=  y\}$ denote the preimage of $f$ at 
$y$. For any $f\in \cvx(X)$, we use $\dom f$ to denote the set $\{x\in X: f(x)<+\infty\}$. 
We simply use $\{F\le c\}$ to express the sublevel set $\{x:F(x)\le c\}$. 
\begin{defn}\label{def:CV_p}
For $p\geq 1$, let $\cvx_0^p(X)$ denote the collection of all  $\f:X\to \R\cup\{+\infty\}$ with the following properties:
\begin{enumerate}
\item $f$ is convex and nonnegative, i.e., $f(tx+(1-t)y)\le tf(x)+(1-t)f(y)$, $\forall x,y\in X$, $\forall t\in[0,1]$, and $f(  x) \geq 0$ for all $  x \in X$;
\item $f$ is positively $p$-homogeneous, i.e., $f(\lambda x)=\lambda^pf(x)$, $\forall x\in X$, $\forall \lambda>0$; 
\item $f$ is lower semi continuous.
\end{enumerate}
\end{defn}

Below, we list some useful properties of a function $f\in \cvx_0^p(X)$ that we will apply.  
\begin{prop}\label{pro:CV_p-basic}
For any $f\in \cvx_0^p(X)$ where $p\ge1$, we have:
\begin{enumerate}[(i)]
\item $\ker f$ is a closed convex cone of $X$.
\item $f(x)\ge f(  x+  z)$ for any $  z\in  \ker f$ and $  x\in X$.    
\item $f^{\frac1p}\in \cvx_0^1(X)$.
\item $f^\circ \in \cvx_0^p(X^*)$, $\ker f^\polar=(\dom f)^\polar$ and $\overline{\dom f^\polar}=(\ker f)^\polar$.
\item $f^{\polar\polar}=f$ and $\partial f(x)\subset (\ker f)^\polar$.
\end{enumerate}
\end{prop}

\begin{prop}\label{pro:CV_1&Cv_p}
For any $f\in \cvx_0^1(X)$, we 
have:
\begin{enumerate}[(i)]
\item $f^p\in \cvx_0^p(X)$, $\partial f^p(x)=pf^{p-1}(x)\partial f(x)$ whenever $x\not\in\ker f$, and $(f^p)^\polar(x^*)=\frac{(p-1)^{p-1}}{p^{p}}(f^\polar(x^*))^p$ 
\item for any $x^*\in X^*$, 
$$f^\polar(x^*)=\sup_{f(x)\le 1}\langle x^*,x\rangle
$$
\item $\partial f(0)=\{x^*\in X^*:f^\polar(x^*)\le 1\}$
\end{enumerate}
\end{prop}

\begin{prop}\label{prop:f-continuous-equiva}
Given $f\in\cvx_0^p(X)$, then, the following three descriptions are equivalent:
\begin{enumerate}[(i)]
\item $f$ is continuous at $0$,
\item $f$ is continuous at some point $x_0\in X$,
\item  $f$ is continuous on $X$.
\end{enumerate}
\end{prop}

We then have the following basic properties for special RC functions:
\begin{corollary}\label{cor:RC-property}
For any $f,g\in \cvx_0^p(X)$, there hold $f^{\frac1p},g^{\frac1p}\in \cvx_0^1(X)$, and $$\frac{g^\polar}{f^\polar}=\Big(\frac{(g^{\frac1p})^\polar}{(f^{\frac1p})^\polar}\Big)^p.$$     
\end{corollary}

We then characterize some subclasses of $\cvx_0^p(X)$.  
Let \begin{equation}\label{eq:defn-cvx0+}
\cvx_{0,+}^p(X)=\{f\in \cvx_0^p(X): f(x)>0,\forall x\ne0,\,\overline{\dom f}=X\},    
\end{equation}
and let 
$$\cvp_c^p(X)=\{f\in \cvx_{0,+}^p(X):\text{ both }f\text{ and }f^\polar\text{ are continuous at }0\} $$
which is equivalent to the form introduced in 
\eqref{eq:defn-cvxc}, due to Proposition \ref{prop:f-continuous-equiva}.

\begin{prop}\label{prop:character:CV_c^1}We have the following characterization of $\cvp_c^p(X)$:
$$\cvp_c^p(X)=\big\{f\in \cvx_{0,+}^p(X):\exists C_2>C_1>0\text{ s.t. }C_1\|x\|^p\le f(x)\le C_2\|x\|^p,\forall x\in X\big\} $$ 
\end{prop}

\begin{corollary}\label{cor:function-spaces}
Given a reflexive space $X\ne0$, we have: 
$$f\in \cvx_0^p(X)\Longleftrightarrow f^\polar\in \cvx_0^p(X^*)$$
$$f\in \cvx_{0,+}^p(X)\Longleftrightarrow f^\polar\in \cvx_{0,+}^p(X^*)$$
$$f\in \cvp_c^p(X)\Longleftrightarrow f^\polar\in\cvp_c^p(X^*)$$
\end{corollary}

For the detailed proofs of the above auxiliary results, see Appendix.

Sometimes we consider quotient spaces since a closed subspace of $X$ may not have a closed linear complement in $X$. 
For any closed linear subspace $\Pi\subset X$, let $[x]=\Pi+x$ be the translation of $\Pi$ along $x$. 
Let $\pi:X\to X/\Pi$, $x\mapsto [x]:=\pi(x)$ be the quotient map, and let $\pi^*:(X/\Pi)^*\to X^*$ be the conjugate operator of $\pi$. Then, $\pi^*((X/\Pi)^*)= \Pi^\bot$, and $\pi^*$ induces the canonical isomorphism $(X/\Pi)^*\cong \Pi^\bot$. For convenience, such isomorphism will be still denoted by $\pi^*:(X/\Pi)^*\to \Pi^\bot$. Thus, we have 
$\langle \bar v,[x]\rangle=\langle \bar v,\pi(x)\rangle=\langle \pi^*(\bar v),x\rangle$ for any $[x]\in X/\Pi$ and $\bar v\in ( X/\Pi)^*$, and moreover, $\langle v,x\rangle=\langle (\pi^*)^{-1}v,[x]\rangle$ for any $v\in \Pi^\bot$ and $x\in X$.    

For any $f\in \cvx_0^p(X)$, we define $f_\Pi\in \cvx_0^p(X)$ as
$$f_\Pi(x)=\inf_{x'\in [x]}f(x')$$
 and define the quotient function
$\bar f_\Pi\in \cvx_0^p(X/\Pi)$ as
$$\bar f_\Pi([x])=\inf_{x'\in [x]}f(x').$$

\begin{example}
Consider a simple case that  $\Pi\subset \ker f$ is a closed linear subspace, then, we have $\bar f_\Pi([x])=f(x)$ and $\pi^*(\partial\bar f_\Pi([x]))=\partial f(x)$. 
A verification is as follows. 

For any $v\in \partial f(x)\subset(\ker f)^\polar\subset\Pi^\polar=\Pi^\bot$, we shall prove $(\pi^*)^{-1}v\in \partial \bar f_\Pi([x])$. 
In fact, $\bar f_\Pi([y])-\bar f_\Pi([x])=f(y)-f(x)\ge\langle v,y-x\rangle=\langle (\pi^*)^{-1}v,[y]-[x]\rangle$, $\forall [y]\in X/\Pi$, implying that $(\pi^*)^{-1}v\in \partial\bar f_\Pi([x])$. On the other direction, for any $\bar v\in \partial\bar f_\Pi([x])\subset( X/\Pi)^*$, for any $y\in X$, $f(y)-f(x)=\bar f_\Pi([y])-\bar f_\Pi([x])\ge \langle \bar v,[y]-[x]\rangle=\langle \bar v,\pi(y-x)\rangle=\langle \pi^*(\bar v),y-x\rangle$, meaning that $\pi^*(\bar v)\in \partial f(x)$. Therefore, we obtain $\pi^*(\partial\bar f_\Pi([x]))=\partial f(x)$.
\end{example}

In general, we have the following property linking $\partial \bar f_\Pi$ and $\partial f$.

\begin{prop}\label{pro:quotient-subdifferential}
For any $x\in X$, we have
$$ \partial f(x)\cap \Pi^\bot=\begin{cases}
\pi^*(\partial\bar f_\Pi([x])),&\text{ if }f(x)=\bar f_\Pi([x]),\\
\varnothing,&\text{ otherwise.}
\end{cases} $$
\end{prop}

Since this result is merely a supplement and does not form the main line of this paper, we have included its proof to the appendix.


\subsubsection{Morse critical points and Rothe critical groups}\label{sec:def-Morse}

The Morse theory in the classical sense does not work for nonsmooth cases. 
Fortunately, Chang \cite{Chang81,Chang93}, Ioffe \cite{Ioffe}, Szulkin \cite{Szulkin}, Katriel \cite{Katriel}, Corvellec \cite{Corvellec},  Degiovanni \cite{Degiovanni94} and many experts in nonlinear analysis have developed a more general Morse theory, allowing us to utilise the framework they have established. 

\begin{defn}[Morse critical point \cite{Degiovanni11}]\label{defn:Morse-critical}
Let $M$ be a topological space and  $F:M\to \R$ be a continuous function. 
A point $x_0\in M$ is a {\sl Morse regular point} of $F$ if there exist a neighborhood $U$  of $x_0$ 
and a continuous map $$\h:U\times[0,1]\to M,\;\; \h(x,0)=  x$$
satisfying $$F(\h( x,t))< F(x),$$
for any $ x\in U$ and $t>0$. We say that $x_0$ is a {\sl Morse critical point} of $F$ on $M$ if it is not Morse regular. 

We say that $c$ is a {\sl Morse critical value} of $F$ if $c=F(x_0)$ for some Morse critical point $x_0$ of $F$. 
\end{defn}

A Clarke critical point that is not Morse critical can be ``removed'' from the perspective of topology as it is unstable under taking perturbation or local deformation (see Definition \ref{defn:Morse-critical}). 
Therefore, from a topological perspective, we may regard Morse critical points as essential critical points; they carry the key information of Morse theory.

\begin{defn}[Rothe's 
  critical group]
 Rothe's {\sl  critical group} at an isolated critical point $\al$ of a function $F$ is defined as $C_q(F,\al):=H_q(\{F\le c\}\cap U_{\al},\{F\le c\}\cap
U_{\al}\setminus \{\al\})$, $q\in\mathbb{Z}$, where $c=F(\alpha)$, and $H_*(\cdot,\cdot)$
is the singular relative homology  with real field coefficients, and $U_{\al}$ is an open neighborhood of $\al$.
\end{defn}

\subsubsection{Lusternik-Schnirelman theory and min-max critical values}\label{sec:def-minmax}
Let $X$ be a reflexive Banach space with a compact Lie group action $\mathtt{G}$. 
Let $\mathcal{S}$ be a collection of all $\mathtt{G}$-invariant closed subsets of $X$, and let $\Phi$ be the set of all $\mathtt{G}$-equivariant continuous mappings from $X$ into itself, i.e., $\varphi\in \Phi$ if and only if $\varphi:X\to X$ is continuous and $\varphi\circ g=g\circ \varphi$, $\forall g\in \mathtt{G}$. For more on this topic, we refer to \cite{Chang93} for details. 
\begin{defn}[admissible index]\label{def:admissible-index}
A quadruple‌ $(\mathcal{S},\mathrm{ind},\Phi,\mathtt{G})$ is called an \emph{index quadruple‌} equipped with an \emph{\textbf{admissible index}}  $\mathrm{ind}:\mathcal{S}\to \mathbb{N}\cup\{+\infty\}$ if it satisfies the following  properties (I1), (I2) and (I3), or it satisfies (I1), (I2) and (I4): 
\begin{itemize}
\item[(I1)] monotonicity: For any $S\subset S'$ with $S,S'\in \mathcal{S}$, we have $\mathrm{ind}(S)\le \mathrm{ind}(S')$;
\item[(I2)] continuity: For any $S\in \mathcal{S}$, there exists a closed neighborhood $U$ of $S$ such that $U\in \mathcal{S}$ and $\mathrm{ind}(U)= \mathrm{ind}(S)$;
\item[(I3)] homotopy: If $S$ is homotopy equivalent to $S'$ with $S,S'\in \mathcal{S}$, then $\mathrm{ind}(S)= \mathrm{ind}(S')$;
\item[(I4)]  nondecreasing under continuous map: For any $S\in \mathcal{S}$ and any $\mathtt{G}$-equivariant continuous map $\varphi\in \Phi$, $\mathrm{ind}(S)\le  \mathrm{ind}(\varphi(S))$.
\end{itemize}
\end{defn}

Fixed an admissible index, it is standard to define a  sequence of critical values.

\begin{defn}[Lusternik-Schnirelman min-max critical values] \label{defn:minimax-cri}
Given a function $F:X\to \R\cup\{+\infty\}$ satisfying $F(g(x))=F(x)$ for any $ g\in \mathtt{G}$ and $x\in X$,  the quantities \begin{equation}\label{eq:index-LS-minmax-c}
c_k(F):=\inf_{\mathrm{ind}(S)\ge k}\sup_{x\in S}F(x) ,\;\; k=1,2,\cdots   
\end{equation}
are called the Lusternik-Schnirelman min-max critical values of $F$ (with respect to a given $\mathrm{ind}$). 
\end{defn}

If $F$ is further assumed to be Lipschitz continuous, 
then $c_k(F)$ in \eqref{eq:index-LS-minmax-c} defines the critical value of $F$ in the sense of Clarke \cite{Chang81}. 

For a general $F$ that is not symmetric, we simply taking $\mathtt{G}=\{1\}$ as the trivial group. In this case, the standard Lusternik-Schnirelman category is an admissible index. 
In most of our applications, the function $F$ is assumed to be even (i.e., $F(-x)=F(x)$, $\forall x$), and thus we can take $\mathtt{G}=\{-1,1\}$ as a $\mathbb{Z}_2$-action. See the next example for explanation.

\begin{example}
 There are many admissible indexes commonly used for 
 even functionals:
\begin{itemize}
\item Krasnoselskii genus: The \emph{Krasnoselskii genus} \cite{Coffman} of an origin-symmetric compact set $S$ is defined by $$\mathrm{ind}_{\mathrm{Kras}}(S):=\min\{k\in\mathbb{Z}^+:\exists\text{ odd continuous }\eta:S\to \mathbb{R}^k\setminus\{0\}\}.$$ 
\item  Conner-Floyd index: The \emph{Conner-Floyd index} $\mathrm{ind}_{\mathrm{CF}}$ of an origin-symmetric compact set $S$ is defined by $$\mathrm{ind}_{\mathrm{CF}}(S):=\min\{k\in\mathbb{Z}^+:\exists\text{ odd continuous }\eta:\mathbb{S}^{k-1}\to S\setminus\{0\}\},$$ 
where $\mathbb{S}^{k-1}$ stands for the unit sphere of $\mathbb{R}^k$. See \cite{Matousek} for details.
\item cohomological index: The cohomological index, denoted by $\mathrm{ind}_{\mathrm{AS}}(\cdot)$, works for $\mathbb{Z}_2$-spaces, is defined via Alexander-Spanier cohomology, see \cite[Chapter 2]{Perera}. 
\item Yang index: This index, denoted by $\mathrm{ind}_{\mathrm{Yang}}(\cdot)$, is defined via homology information \cite{Yang}. We will not write down the definition explicitly, but interested readers may refer to \cite{Perera}.
\end{itemize}
For example, if we take $\mathcal{S}$ to be the set of  origin-symmetric compact subsets of the Sobolev space $H_0^1(\Omega)$, and take the functional  $F(x)=\frac{\int_\Omega|\nabla x(\xi)|^pd\xi}{\int_\Omega|x(\xi)|^pd\xi}$, then $\{c_k(F)\}_{k\ge 1}$ proposed in \eqref{eq:index-LS-minmax-c} defines the variational min-max eigenvalues of the $p$-Laplacian on $\Omega$ with Dirichlet homogeneous boundary condition.  
In particular, the min-max eigenvalues of $p$-Laplacian defined in \eqref{eq:index-LS-minmax-c} by using $\mathrm{ind}=\mathrm{ind}_{\mathrm{CF}}$ are 
the Drabek-Robinson eigenvalues \cite{Drabek}.
\end{example}

According to Definitions \ref{def:admissible-index} and \ref{defn:minimax-cri}, to define the min-max critical values of a function $F$ with respect to a prescribed index, we implicitly require that $F$ is $\mathtt{G}$-symmetric. 
All the proofs regarding the properties of $c_k(F)$ require the verification of $\mathtt{G}$-invariants. However, such verification is standard with no difficulty, and to avoid complicated notation, we will not explicitly write the verification process involving the compact Lie group action. 


\begin{prop}\label{prop:index-discontinuous=critical}
A real number $c$ is a discontinuous point of the function $t\mapsto \mathrm{ind}\{F\le t\}$ if and only if $c=c_k(F)$ for some $k\in\mathbb{Z}_+$. 
Moreover,  
$\mathrm{ind}\{F\le c+\varepsilon\}\ge k>\mathrm{ind}\{F\le c-\varepsilon\}$ for any $\varepsilon>0$, if and only if $c=c_k(F)$. 
\end{prop}

\begin{proof}
We shall use the easy observation that for any $S\in\mathcal{S}$,  $$\sup_{x\in S}F(x)\le c\Longleftrightarrow S\subset\{F\le c\}.$$
In addition, since $F(g(x))=F(x)$ for any $ g\in G$, we have for any $c\in\R$, the sublevel set $\{F\le c\}$ is $G$-invariant, and thus $\{F\le c\}\in\mathcal{S}$. 

We first prove that $k>\mathrm{ind}\{F\le c_k(F)-\varepsilon\}$, $\forall \varepsilon>0$. Suppose the contrary that $k\le \mathrm{ind}\{F\le c_k(F)-\varepsilon_0\}$ for some $\varepsilon_0>0$. 
Then, by definition, $$c_k(F)\le \sup\limits_{x\in \{F\le c_k(F)-\varepsilon_0\}}F(x)\le c_k(F)-\varepsilon_0$$ which is a contradiction. 

Next, we prove that $k\le \mathrm{ind}\{F\le c_k(F)+\varepsilon\}$, $\forall \varepsilon>0$. In fact, for any $\varepsilon>0$, there exists $S$ with $\mathrm{ind}(S)\ge k$ such that $\sup_{x\in S}F(x)\le c_k(F)+\varepsilon$. Thus, $S\subset \{F\le c_k(F)+\varepsilon\}$ and then by the monotonicity of the index, we have $\mathrm{ind}\{F\le c_k(F)+\varepsilon\}\ge \mathrm{ind}(S)\ge k$. 

In consequence, we have proved that 
\begin{equation}\label{eq:ck(F)+-epsilon}
\mathrm{ind}\{F\le c_k(F)+\varepsilon\}\ge k>\mathrm{ind}\{F\le c_k(F)-\varepsilon\}    
\end{equation}
 for any $\varepsilon>0$. 

Note that, by the monotonicity of $\mathrm{ind}$, the function $t\mapsto \mathrm{ind}\{F\le t\}$ is non-decreasing. 
Thus, $c$ is a discontinuous point of $t\mapsto \mathrm{ind}\{F\le t\}$ if and only if $ \lim_{\varepsilon\to0^+}\mathrm{ind}\{F\le c-\varepsilon\}<\lim_{\varepsilon\to0^+}\mathrm{ind}\{F\le c+\varepsilon\}$. 

Suppose that $t\mapsto \mathrm{ind}\{F\le t\}$ is discontinuous at $t=c$. 
Without loss of generality, we assume 
\begin{equation}\label{eq:c+-epsilon=indk}
\lim_{\varepsilon\to0^+}\mathrm{ind}\{F\le c-\varepsilon\}\le k-1<k\le \lim_{\varepsilon\to0^+}\mathrm{ind}\{F\le c+\varepsilon\}
\end{equation} for some positive integer $k$. 
Then $$c_k(F)\le \lim_{\varepsilon\to0^+}\sup\limits_{x\in \{F\le c+\varepsilon\}}F(x)\le c.$$ 
If $c_k(F)< c$, then $c_k(F)<c-\varepsilon$ for sufficiently small $\varepsilon>0$ and hence by \eqref{eq:ck(F)+-epsilon}, $\mathrm{ind}\{F\le c-\varepsilon\}\ge k$, a contradiction to \eqref{eq:c+-epsilon=indk}. 
Therefore, $c_k(F)=c$. 
\end{proof}

\begin{prop}
If $c^-<c<c^+$ are consecutive discontinuous points of $t\mapsto \mathrm{ind}\{F\le t\}$, i.e., $c^-,c$ and $c^+$ are discontinuous points of $t\mapsto \mathrm{ind}\{F\le t\}$ 
such that $\mathrm{ind}\{F\le t\}$ is constant on $ (c^-,c)$ and constant on $ (c,c^+)$, respectively, then $c$ appears exactly $\mathrm{ind}\{F\le c\}-\mathrm{ind}\{F\le c^-\}$ times in the sequence \eqref{eq:index-LS-minmax-c}.
\end{prop}

\begin{proof}We continue with the proof of Proposition \ref{prop:index-discontinuous=critical}. 
Since the function 
$$ 
\begin{array}{lll}
    \R&\to&\{0,1,2,\cdots\}  \\
     t&\mapsto& \mathrm{ind}\{F\le t\} 
\end{array}
$$
is non-decreasing and right-continuous, we may suppose $\mathrm{ind}\{F\le t\}=k_j$, $\forall t\in [c_j,c_{j+1})$, $j\ge 1$, where $k_1<k_2<\cdots$ are nonnegative integers, and $c_1<c_2<\cdots$ are real numbers. We shall prove that 
$c_j$ appears $k_j-k_{j-1}$ times in the sequence \eqref{eq:index-LS-minmax-c}. 

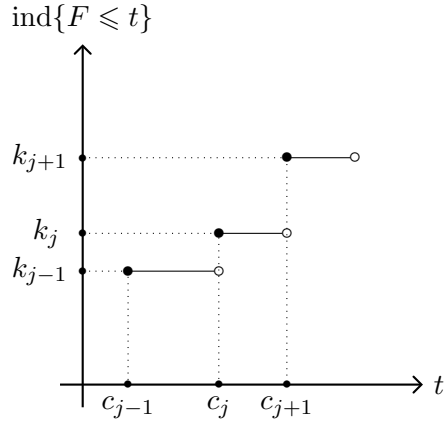
\begin{center}
\begin{tikzpicture}[scale=3]
\node (1) at (-0.19,0.5) {$k_{j-1}$};
\node (1) at (-0.16,2/3) {$k_{j}$};
\node (1) at (-0.19,1) {$k_{j+1}$};
\node (0) at (0.2,0.5) {\footnotesize $\bullet$};
\node (0) at (0.6,2/3) {\footnotesize $\bullet$};
\node (0) at (0.9,1) {\footnotesize $\bullet$};
\node (0) at (0.6,0.5) {\footnotesize $\circ$};
\node (0) at (0.9,2/3) {\footnotesize $\circ$};
\node (0) at (1.2,1) {\footnotesize $\circ$};
\node (0) at (0,0.5) {\tiny $\bullet$};
\node (0) at (0,2/3) {\tiny $\bullet$};
\node (0) at (0,1) {\tiny $\bullet$};
\node (0) at (0.2,0) {\tiny $\bullet$};
\node (0) at (0.6,0) {\tiny $\bullet$};
\node (0) at (0.9,0) {\tiny $\bullet$};
\node (0) at (0.2,-0.1) { $c_{j-1}$};
\node (0) at (0.6,-0.1) { $c_j$};
\node (0) at (0.9,-0.1) { $c_{j+1}$};
\draw[thick,->,>=angle 90](-0.1,0)--(1.5,0)node[right] {$t$};
\draw[thick,->,>=angle 90](0,-0.1)--(0,1.5) node[above] {$\mathrm{ind}\{F\le t\}$};
\draw[dotted] (0,2/3)--(0.6,2/3)--(0.6,0);
\draw[dotted] (0,0.5)--(0.2,0.5)--(0.2,0);
\draw[dotted] (0,1)--(0.9,1)--(0.9,0);
\draw[domain=0.2:0.58] plot (\x,{1/2});
\draw[domain=0.6:0.88] plot (\x,{2/3});
\draw[domain=0.9:1.18] plot (\x,{1});
\end{tikzpicture}
\end{center}

Then, on the one hand, 
$$\inf_{\mathrm{ind}(S)\ge k_j}\sup_{x\in S}F(x)\le \sup_{x\in \{F\le c_j\}}F(x)\le c_j.$$

On the other hand, for any $\varepsilon>0$, for any $S\in\mathcal{S}$ with $\mathrm{ind}(S)\ge k_{j-1}+1$, we have $S\setminus \{F\le c_j-\varepsilon\}\ne\varnothing$, because otherwise $S\subset \{F\le c_j-\varepsilon\}$ for some $\varepsilon>0$ implies $\mathrm{ind}(\{F\le c_j-\varepsilon\})\ge \mathrm{ind}(S)\ge k_{j-1}+1$, which contradicts the assumption that $\mathrm{ind}(\{F\le c_j-\varepsilon\})\le k_{j-1}$.  

Hence, for any $\varepsilon>0$, and for any $S$ with $\mathrm{ind}(S)\ge k_{j-1}+1$, $\sup_{x\in S}F(x)> c_j-\varepsilon$, which yields $$\inf_{\mathrm{ind}(S)\ge k_{j-1}+1}\sup_{x\in S}F(x)\ge c_j-\varepsilon.$$
Due to the arbitrariness of $\varepsilon>0$, we have $$\inf_{\mathrm{ind}(S)\ge k_{j-1}+1}\sup_{x\in S}F(x)\ge c_j.$$
Therefore, we obtain
$$c_j\ge \inf_{\mathrm{ind}(S)\ge k_j}\sup_{x\in S}F(x) \ge \inf_{\mathrm{ind}(S)\ge k_{j-1}+1}\sup_{x\in S}F(x)\ge c_j, $$
that is, $c_{k_{j-1}+1}(F)=\cdots=c_{k_j}(F)=c_j$, and hence 
$c_j$ appears $k_j-k_{j-1}$ times in the sequence \eqref{eq:index-LS-minmax-c}.
\end{proof}

\subsubsection{Lagrange criticality 
}\label{sec:def-Lagrange}

The Lagrange multiplier corresponding to the critical point problem of the ratio of convex functions is quite useful. 
It would be helpful if we consider the Lagrange multiplier as a ``generalized'' critical value of $f/g$, and this leads to the following definition. 
\begin{defn}\label{defn:critical(f,g)}
A point $x$ is called a  \emph{Lagrange critical point} of $(f,g)$ if there exists $c\in\R$ such that 
\begin{equation}\label{eq:Lagrange-critical-eigen}
0\in\partial f(x)-c\,\partial g(x),    
\end{equation}
and such $c$ is called a Lagrange critical value of $(f,g)$ with respect to $x$. Such $(c,x)$ is said to be a  Lagrange critical pair of $(f,g)$.

The \emph{multiplicity} of a Lagrange critical value $c$ of $(f,g)$ is defined as the index of the set of all the Lagrange critical points of $(f,g)$ corresponding to $c$, and we denote it by $\mathrm{mult}_{f,g} (c)$. 
\end{defn}

In this paper, the Lagrange critical values (resp., points) of $(f,g)$ are also called Lagrange critical values (resp., points) of the quotient $f/g$.  
When we simply say a critical point (or value) of $(f,g)$, the meaning is always by Definition \ref{defn:critical(f,g)}. 

We say $(c,x)$ is a \emph{nontrivial Lagrange critical pair} of $(f,g)$ if it satisfies  \eqref{eq:Lagrange-critical-eigen} as well as $c\ne0$, $f(x)\ne0$ and $g(x)\ne0$; in this case, $x$ is called a \emph{nontrivial Lagrange critical point} and $c$ is called a \emph{nontrivial Lagrange critical value}.

Given a nonempty set $S\subset X$, let $\mathrm{cone}(S):=\{tx:x\in S,t>0\}$ denote the cone (without apex). We have the following description of nontrivial Lagrange critical points for special function pairs.
\begin{prop}\label{pro:open-cone-intersect}
Let $f,g\in \cvx_0^p(X)$. Then $x$ is a nontrivial Lagrange critical point of $(f,g)$ if and only if $\mathrm{cone}(\partial f(x))\cap \mathrm{cone}(\partial g(x))\ne\varnothing$ and $f(x)$ and $g(x)$ are not both zero.
\end{prop}
\begin{proof}If  $x$ is a nontrivial Lagrange critical point of $(f,g)$, that is, $c\ne0$, $f(x)\ne0$, $g(x)\ne0$ and 
$$ \partial f(x)\bigcap c\,\partial g(x)\ne\varnothing.$$
Taking $x^*\in \partial f(x)\bigcap c\,\partial g(x)$, by Euler's identity for $p$-homogeneous function, we have $0<pf(x)=\langle x^*,x\rangle=c\langle x^*/c,x\rangle=cpg(x)>0$ since $f,g\ge 0$. Thus, $c>0$, meaning that $ \partial f(x)\bigcap\mathrm{cone}(\partial g(x))\ne\varnothing$. Hence, $\mathrm{cone}(\partial f(x))\cap \mathrm{cone}(\partial g(x))\ne\varnothing$. 

Conversely, if $\mathrm{cone}(\partial f(x))\cap \mathrm{cone}(\partial g(x))\ne\varnothing$, then there exists $y\in \mathrm{cone}(\partial f(x))\cap \mathrm{cone}(\partial g(x))$. And there exist $c_1,c_2>0$ such that $y\in c_1\partial f(x)\cap c_2\partial g(x)$. 
Therefore, $y/c_1\in \partial f(x)\cap c\partial g(x)$, where $c:=c_2/c_1>0$. By condition, either $f(x)>0$ or $g(x)>0$, by Euler’s identity, we have $f(x)=cg(x)>0$. Together with $\partial f(x)\cap c\partial g(x)\ne\varnothing$, we have proved that $x$ is a nontrivial Lagrange critical point of $(f,g)$.
\end{proof}

In this paper, we use $\mathrm{Cri}_{LP}(f,g)$ denote the nontrivial Lagrange critical points of $(f,g)$. 

\subsection{\textbf{Part I. The first critical duality theory} 
}
\label{sec:first-c-dual}
\subsubsection{A duality theorem on Lagrange criticality}

\begin{theorem}\label{th:critical(f,g)}
Let $f,g\in \cvx_0^p(X)$. 
Then the nontrivial Lagrange critical values of $(f,g)$ coincide with that of $(g^\polar , f^\polar )$. 

Furthermore, for any nontrivial Lagrange critical point $x$ of $(f,g)$, and for any $u\in\mathrm{cone}(\partial f( x))\cap  \mathrm{cone}(\partial g( x))$, $u$ is nontrivial Lagrange critical point of $(g^\polar , f^\polar )$. 
In particular, if $g$ (resp., $f$) is Fr\'echet differentiable at $x$, then $\partial g(x)$ (resp., $\partial f(x)$) is a nontrivial Lagrange critical point of $(g^\polar , f^\polar )$. 
\end{theorem}

\begin{remark}
Theorem \ref{th:critical(f,g)} implies the Lagrange equivalence\footnote{See Definition \ref{defn:Lagrange-intro}.} between $(f,g)$ and $(g^\polar,f^\polar)$, as stated in Theorem \ref{th:fgmain:polar}. In fact, to verify (L2) of Definition \ref{defn:Lagrange-intro}, we can take the set-valued maps $T_1$ and $T_2$ as $T_1(x)=\mathrm{cone}(\partial f( x))\cap  \mathrm{cone}(\partial g( x))$ and $T_2(y)=\mathrm{cone}(\partial f^\polar( y))\cap  \mathrm{cone}(\partial g^\polar( y))$, respectively. 
If $g$ and $f^\polar$ are Fr\'echet differentiable at nonzeros, then the maps  $T_1$ and $T_2$ reduce to $T_1(x)=\mathrm{cone}(\partial g( x))$ and $T_2(y)=\mathrm{cone}(\partial f^\polar( y))$, respectively. 
\end{remark}

\begin{proof}
According to Propositions  \ref{pro:CV_p-basic} and \ref{pro:CV_1&Cv_p}
, we only need to work with the case $p=1$. 

Suppose that 
$(c,x)$ is a nontrivial Lagrange critical pair of $(f,g)$, then $c>0$, $x\not\in \ker f\cup \ker g$, 
and $x\in\dom f\cap \dom g$, and there exists $u\in\partial g(x)$ such that $c u\in \partial f(x)$. By Euler's identity, 
$\langle u,x\rangle=g(x)$ and $\langle cu,x\rangle=f(x)$, implying that $cg(x)=f(x)$. 
Since $x\not\in \ker f\cup \ker g$, $c=f(x)/g(x)>0$. 

Note that $\partial g$ and $\partial f$ possess the  
 zero-homogeneity,  i.e., 
\begin{equation}\label{eq:0-homogeneity-partial}
\partial f(tx)=\partial f(x)\text{ and }\partial g(tx)=\partial g(x),\;\forall x\in X,\,\forall t>0.
\end{equation} 
Hence, without loss of generality, we may assume that $x$ further satisfies  $g(x)=1$. Then, for any $u\in\partial g(x)\subset X^*$, 
it is easy to see $\langle u,x\rangle=g(x)=1$ and $c = \langle c u,x\rangle=f(x)$. 

Claim: If $g(x)=1$ with $g\in \cvx_0^1(X)$, and $u\in\partial g(x)$, then $g^\polar(u)=1$ and $x\in \partial g^\polar(u)$.

Proof of Claim: By definition of subgradient, for any $y\in X$, $$\langle u,y\rangle-1=\langle u,y\rangle-\langle u,x\rangle=\langle u,y-x\rangle\le g(y)-g(x)=g(y)-1$$ which yields $\langle u,y\rangle\le g(y)$, $\forall y\in X$. Thus, it can be checked that $$ g^\polar (u)=\sup_{y\in X:g(y)\le 1}\langle u,y\rangle=1.$$
Let $J:X\to X^{**}$ be the canonical map, i.e., for each  $z\in X$, $\langle Jz,v\rangle=\langle v,z\rangle$, $\forall v\in X^*$. Then, 
for any $v\in X^*$, 
$$\langle Jx,v-u\rangle=\langle v-u,x\rangle=\langle v,x\rangle-1\le  g^\polar (v)-1=  g^\polar (v)- g^\polar (u)$$ which implies that $Jx\in \partial   g^\polar (u)$. 
For convenience, we identify $Jx$ with $x$, and  simply rewrite $x\in \partial  g^\polar (u)$. This concludes the proof of the claim.

\vspace{0.1cm}

Applying the above claim to $f$, it follows from $f(\frac{x}{c})=1$ and $c u\in\partial f(x)=\partial f(\frac{x}{c})$ that 
$$\frac{x}{c}\in \partial   f^\polar (c u)=\partial   f^\polar ( u),$$
where we used the zero-homogeneity of $\partial f$ and $\partial   f^\polar $ (see \eqref{eq:0-homogeneity-partial} for the explanation of zero-homogeneity). 
Accordingly, $x\in c \partial  f^\polar ( u)$. In consequence, 
$x\in   g^\polar (u)\cap c \partial   f^\polar ( u)$ meaning that $c$ is a critical value of $(  g^\polar ,  f^\polar )$ with a corresponding critical point $u$. 

Note that in the above proof, we only require $u\in \partial g(x) \cap \frac{1}{c}\partial f(x)$. By the zero-homogeneity of $\partial g$ and $\partial f$ shown in \eqref{eq:0-homogeneity-partial}, we can relax this condition to $u\in \mathrm{cone}(\partial g(x)) \cap  \mathrm{cone}(\partial f(x))$. 
Since $X$ is reflexive, the canonical map $J$ gives an isomorphism between $X$ and $X^{**}$, the nonzero critical values of $(g^\polar ,f^\polar )$ should also be critical values of $(f,g)$. 

If $g$ is further assumed to be Fr\'echet differentiable at $x$, then $\partial g(x)$ is a singleton and thus $\mathrm{cone}(\partial g(x)) \cap  \mathrm{cone}(\partial f(x))\ne\varnothing$ implies that $\partial g(x)\in \mathrm{cone}(\partial g(x)) \cap  \mathrm{cone}(\partial f(x))$, and then $\partial g(x)$ is a nontrivial Lagrange critical point of $(g^\polar , f^\polar )$. 
The proof is then completed.
\end{proof}

\begin{corollary}\label{cor:bijection-Lagra}
Under the same condition with Theorem \ref{th:critical(f,g)},  
if $g$ and $f^\polar$ are Fr\'echet differentiable at nonzeros, then, $$f\cdot\partial g|_{\mathrm{Cri}_{LP}(f,g)}:\mathrm{Cri}_{LP}(f,g)\to \mathrm{Cri}_{LP}(g^\polar,f^\polar)$$ is a bijection. 
Moreover, $\partial $ induces a bijection $$\partial g|_{\mathrm{Cri}_{LP}(f,g)\cap f^{-1}(1)}:\mathrm{Cri}_{LP}(f,g)\cap f^{-1}(1)\to \mathrm{Cri}_{LP}(g^\polar,f^\polar)\cap (g^\polar)^{-1}(1)$$ 
with the inverse $\partial f^\polar|_{\mathrm{Cri}_{LP}(g^\polar,f^\polar)\cap (g^\polar)^{-1}(1)}$.     
\end{corollary}


\subsubsection{
Homotopy equivalence of sublevel sets under polarity
}

In order to study Morse criticality of RC functions, we need to first investigate the topological properties of sublevel sets, which is also useful in min-max criticality and applications. 
It is natural to consider a RC function $f/g$ with $f$ and $g$ in  $\cvx_{0,+}^1(X)$.

Moreover, for any $f\in \cvx_{0,+}^1(X)$, by Proposition \ref{pro:CV_1&Cv_p}, we have  
(which is equivalent to norm-like dual in this case): 
$$f^\polar (x^*):=\sup_{x\in X\setminus\{0\}}\frac{\langle x^*,x\rangle}{f(x)},\;\forall x^*\in X^*.$$
We simply use $\{f/g\le c\}$ to denote the sublevel set 
$$\Big\{x\in X\setminus\{0\}: \frac{f(x)}{g(x)}\le c\Big\}.$$


Below, we show a technical lemma which will be used in the proof of Theorems \ref{th:ratio-sublevel-equi}, \ref{th:Morse-Rothe} and \ref{th:index-critical}
\begin{lemma}\label{lem:f/g-zyx-tech}
Let $f,g\in \cvx_{0,+}^1(X)$. For any $x\ne0$,  $y\in\partial g(x)$, and $z'\in\partial f^\circ (y)$, we have $y\ne0$ and $tz+(1-t)x/f(x)\ne0$ for any  $t\in [0,1]$, where $z\in X$ is the unique element related to $z'$ via the natural isomorphism 
map $J:X\to X^{**}$. Furthermore, we have
$$\frac{g^\polar (y)}{f^\polar (y)}  \le\frac{f(x)}{g(x)}\;\text{ and }\;\frac{f(tz+(1-t)x/f(x))}{g(tz+(1-t)x/f(x))}\le \frac{f(x)}{g(x)},\;\forall t\in [0,1]$$
Equality holds if and only if $x$ is a Lagrange critical point of $(f,g)$. 
\end{lemma}

\begin{proof}
It follows from $y\in\partial g(x)$ that $g^\polar (y)=1$ and $x/g(x)\in \partial g^\polar (y)$ and $\langle y,x\rangle=g(x)$. Thus, $y\ne0$ and  
$$f^\polar (y)=\sup_{z}\frac{\langle y,z\rangle}{f(z)}\ge \frac{\langle y,x\rangle}{f(x)}=\frac{g(x)}{f(x)}$$
which implies
\begin{equation} \label{eq:g*/f*<f/g}
\frac{g^\polar (y)}{f^\polar (y)}=\frac{1}{f^\polar (y)} \le\frac{f(x)}{g(x)}. 
\end{equation}
The equality holds if, and only if $f^\polar ( y)f(x)=\langle  y,x\rangle$, if and only if $ y/f^\polar ( y)\in \partial f(x)$. Also, notice that $y\in\partial g(x)$. We finally obtain that $ y/f^\polar ( y)\in \partial f(x)$ if and only if $$\frac{y}{f^\polar ( y)}\in \partial f(x)\bigcap  \frac{1}{f^\polar ( y)}\partial g(x)$$
meaning that the value $f(x)/g(x)=1/f^\polar ( y)$ is a Lagrange critical value of $(f,g)$.

For the second statement, taking $x'=x/f(x)$, we have $f(x')=1$ and $f(x')/g(x')=f(x)/g(x)$. Then, we may assume without loss of generality that $f(x)=1$. 

Since $z'\in\partial f^\circ (y)\subset X^{**}$, there exists the unique $z\in X$ such that $z'=Jz$.  Then  \begin{align*}
g(tz+(1-t)x)-g(x)&\ge \langle y,tz+(1-t)x-x\rangle
\\&=t\langle y,z-x\rangle
\\&=t\langle y,z\rangle- t\langle y,x\rangle
\\&=t\langle Jz,y\rangle-tg(x)
\\&=t\langle z',y\rangle-tg(x)
\\&=tf^\polar (y)-tg(x)
\\&=t\Big(\frac{f^\polar (y)}{g^\polar (y)}-\frac{g(x)}{f(x)}\Big)\ge0
\end{align*}
due to the fact that $f(x)=1$ and $g^\polar (y)=1$ and the inequality \eqref{eq:g*/f*<f/g}. 
Since $x\ne0$, and the above inequality, we have $g(tz+(1-t)x)\ge g(x)>0$. This means $tz+(1-t)x\ne 0$. 

By the convexity of $f$, $f(tz+(1-t)x)\le tf(z)+(1-t)f(x)=1=f(x) $, we obtain 
$$  \frac{f(tz+(1-t)x)}{g(tz+(1-t)x)}\le \frac{f(x)}{g(x)}  .$$
For general $x\ne0$, we replace $x$ by $x/f(x)$, and then the second statement follows. 
\end{proof}

\begin{defn}[Ambient Homotopy Equivalence]
Let \(X\) and \(Y\) be topological spaces, and let \(S \subseteq X\) and \(T \subseteq Y\) be subspaces. Suppose there exist continuous maps
\[
\varphi: S \to T \qquad \text{and} \qquad \psi: T \to S
\]
such that the following conditions hold:
\begin{enumerate}
    \item The composition \(\psi \circ \varphi: S \to X\) is homotopic to the inclusion map \(i_S: S \hookrightarrow X\).
    \item The composition \(\varphi \circ \psi: T \to Y\) is homotopic to the inclusion map \(i_T: T \hookrightarrow Y\).
\end{enumerate}
Then the subspaces \(S\) and \(T\) are said to be \textbf{ambient homotopy equivalent} with respect to the ambient spaces \(X\) and \(Y\).

\end{defn}

\begin{theorem}\label{th:ratio-sublevel-equi-w}
Let $f,g\in \cvx_{0,+}^1(X)$ be such that 
$f^\polar $ and $g$ are $C^1$-smooth\footnote{We say a convex function is $C^1$-smooth at a point if its subgradient is unique valued and continuous at this point. } (or, $f$ and $g^\polar $ are $C^1$-smooth) at nonzeros. 
Then, $\forall c>0$, the restricted sub-level set $\{x\in f^{-1}(1):f(x)/g(x)\le c\}$ is ambient  homotopy equivalent to $\{x\in (g^\polar) ^{-1}(1):g^\polar (x)/f^\polar (x)\le c\}$ with respect to the ambient sublevel sets $ \{f/g\le c\}$ and $\{g^\polar /f^\polar \le c\}$. 

\end{theorem}

\begin{proof}
First, we focus on the case that both $f^\polar $ and $g$ are smooth.  
Consider the  continuous map $\partial g: X\setminus\{0\}\to X^*\setminus\{0\}$, and 
restrict it on $f^{-1}(1)$, we 
have 
$\partial g(f^{-1}(1))\subset (g^\polar) ^{-1}(1)$. 

\vspace{0.2cm}


Claim 1: 
$\partial g$ maps $\{x\in f^{-1}(1):f(x)/g(x)\le c\}$ into $\{x\in (g^\polar) ^{-1}(1):g^\polar (x)/f^\polar (x)\le c\}$.  

Proof of Claim 1: 
By definition of the sublevel set $\{x\in f^{-1}(1):f(x)/g(x)\le c\}$, for any $x\in \{x\in f^{-1}(1):f(x)/g(x)\le c\}$, we have $x\ne 0$, $f(x)=1$ and $f(x)/g(x)\le c$. Now, for any $y\in\partial g(x)$, by Lemma \ref{lem:f/g-zyx-tech}, we have $g^\polar (y)$, $y\ne0$ and $g^\polar (y)/f^\polar(y)\le f(x)/g(x)\le c$.

\vspace{0.2cm}

Claim 2: $\partial f^\polar$ maps $\{x\in (g^\polar) ^{-1}(1):g^\polar (x)/f^\polar (x)\le c\}$ into $\{x\in f^{-1}(1):f(x)/g(x)\le c\}$. 

The proof of Claim 2 is similar to that of Claim 1, and thus we omit it. 

We shall prove that $\partial f^\polar \circ \partial g$ is homotopic to the inclusion map $i_{S}$, where $S:=\{x\in f^{-1}(1): f(x)/g(x)\le c\}$ and $c> 0$ is fixed. 

Define the map $h:f^{-1}(1)\to f^{-1}(1)$ by $h(x)=J^{-1}\partial f^\polar \big(\partial g(x)\big)$. Note that $h(x)$ is actually the element $z$ appearing in Lemma \ref{lem:f/g-zyx-tech}, and thus $th(x)+(1-t)x\ne0$ for any $t\in[0,1]$. 


Let $H:S\times[0,1]\to \{f/g\le c\}$ be defined by  $$H(x,t)=th(x)+(1-t)x,\;\; x\in S,\, 0\le t\le1.$$
It is easy to see that $H$ is continuous, $H(\cdot,0)=i_S(\cdot)$ and $H(x,1)=h(x)$. 
For any $x\in S$, by Lemma \ref{lem:f/g-zyx-tech}, 
$$ \frac{f(H(x,t))}{g(H(x,t))}=\frac{f(th(x)+(1-t)x)}{g(th(x)+(1-t)x)}\le \frac{f(x)}{g(x)} \le c$$
which proves $H(S,t)\subset \{f/g\le c\}$,  $\forall t\in [0,1]$. 
In consequence,  $h:=J^{-1}\partial f^\polar \circ \partial g$ is homotopic to the  inclusion map $i_S$. 

In a similar manner, one can show that $\partial g\circ\partial f^\polar  $ is homotopic to the  inclusion map $i_{S'}$, where $S'=\{y\in (g^\polar) ^{-1}(1):g^\polar (y)/f^\polar (y)\le c\}$. We then proved that the topological spaces $S$ and $S'$ are ambient homotopy equivalent with respect to the ambient sublevel sets $ \{f/g\le c\}$ and $\{g^\polar /f^\polar \le c\}$.

The case of smoothness for $f$ and $g^\polar $ is similar and we omit the details. 
\end{proof}

In the next theorem, we establish a stronger version of Theorem \ref{th:ratio-sublevel-equi-w}.

\begin{theorem}\label{th:ratio-sublevel-equi}
Let $f,g\in \cvp_c^1(X)$ be such that 
$f^\polar $ and $g$ are $C^1$-smooth 
(or, $f$ and $g^\polar $ are $C^1$-smooth) at nonzeros. 
Then, $\forall c>0$, the restricted sub-level set $\{x\in f^{-1}(1):f(x)/g(x)\le c\}$ is homotopy equivalent to $\{x\in (g^\polar) ^{-1}(1):g^\polar (x)/f^\polar (x)\le c\}$, and moreover, the sublevel sets $ \{f/g\le c\}$ and $\{g^\polar /f^\polar \le c\}$ are homotopy equivalent. 

\end{theorem}

\begin{proof}
The proof is almost the same as that of Theorem \ref{th:ratio-sublevel-equi-w} with only slight modification of the construction of continuous deformation map $H$. 
To show the homotopy equivalence between the restricted sublevel sets $S:=\{x\in f^{-1}(1):f(x)/g(x)\le c\}$ and $S':=\{x\in (g^\polar) ^{-1}(1):g^\polar (x)/f^\polar (x)\le c\}$, we 
need to use $\tilde H:S\times[0,1]\to S$ defined as 
$$\tilde{H}(x,t)=\frac{th(x)+(1-t)x}{f(th(x)+(1-t)x)},\;\; x\in S,\, 0\le t\le1,$$
instead of $H$ in the proof of Theorem \ref{th:ratio-sublevel-equi-w}. 
By such $\tilde H$, it is clear that  $h:=\partial f^\polar \circ \partial g$ is homotopic to the  identity $\mathrm{id}|_S$.  That is, $\gradientp g$ is a homotopy equivalence, and $\gradientp f^\polar$ a homotopy inverse to $\gradientp g$. 
In a similar manner, $\partial g\circ\partial f^\polar  $ is homotopic to the  identity $\mathrm{id}|_{S'}$. 

To show the homotopy equivalence between the sublevel sets $ \{f/g\le c\}$ and $\{g^\polar /f^\polar \le c\}$, we should 
replace $\gradientp g$ and $\gradientp f^\polar$ by $\varphi:=f\cdot\gradientp g$ and $\psi:=g^\polar\cdot\gradientp f^\polar$, respectively. 
Note that $\varphi$ and $\psi$ are continuous, and $\psi(\varphi(x))=g^\polar(f(x)\gradientp g(x))\gradientp f^\polar(f(x)\gradientp g(x))=f(x)\gradientp f^\polar( \gradientp g(x))$ due to $g^\polar(\gradientp g(x))=1$, the zero-homogeneity of $\gradientp f^\polar$, and the one-homogeneity of $g^\polar$. 
Replace $H$ in the proof of Theorem \ref{th:ratio-sublevel-equi-w} by the continuous map $H':\{f/g\le c\}\times [0,1]\to \{f/g\le c\}$ defined as 
$$H'(x,t)=t \psi\circ\varphi(x)+(1-t)x,\;\; x\in \{f/g\le c\},\, 0\le t\le1.$$
Then, $H'(x,t)=f(x)(tz+(1-t)x/f(x))$ with $z=\gradientp f^\polar( \gradientp g(x))$, and we can use Lemma \ref{lem:f/g-zyx-tech} to derive $H'(\{f/g\le c\},t)\subset\{f/g\le c\}$, $\forall t\in[0,1]$.
Hence, $\psi|_{\{g^\polar /f^\polar \le c\}}\circ\varphi|_{\{f/g\le c\}}\simeq \mathrm{id}|_{\{f/g\le c\}}$. 
Similarly, we have $\varphi|_{\{f/g\le c\}}\circ\psi|_{\{g^\polar /f^\polar \le c\}} \simeq \mathrm{id}|_{\{g^\polar /f^\polar \le c\}}$. The verification is done.
\end{proof}

According to Theorem \ref{th:ratio-sublevel-equi}, we have: 
\begin{corollary}\label{cor:handle-RC}
If $f/g$ and $g^\polar /f^\polar $ are Morse functions, then 
the handle decompositions associated to $f/g$ are isomorphism to that of $g^\polar /f^\polar $. 
\end{corollary}

\begin{corollary}
If $f/g$ and $g^\polar /f^\polar $ are Morse functions, then 
the  Poincare polynomial $P(\{f/g<b\},\{f/g\le a\})(t)=P(\{g^\polar /f^\polar <b\},\{g^\polar /f^\polar \le a\})(t)$ for any $a<b$ and for any $t\in\R$. 
\end{corollary}

\begin{corollary}
The smallest positive 
critical value (resp., the maximum) 
of $f/g$ and $g^\polar /f^\polar $ coincide exactly. 
\end{corollary}

\subsubsection{Polarity dual theorem on Morse critical points and Rothe critical groups}

\begin{theorem}
\label{th:Morse-Rothe}Let $f,g\in \cvp_c^1(X)$ be such that $f^\polar $ and $g$ are $C^1$-smooth (or, $f$ and $g^\polar $ are $C^1$-smooth) at nonzeros. 
Let $\al$ be a Morse critical point of $f/g$ with the critical group $C_*(f/g,\al)$. Then, $\tau=\gradientp g(\al)$ is a Morse critical point of $g^\polar /f^\polar $, and $f(\al)/g(\al)=g^\polar (\tau)/f^\polar (\tau)$, and the critical groups satisfy
$$C_*\big(\frac fg,\al\big)\cong C_*\big(\frac{g^\polar }{f^\polar },\tau\big) .$$

\end{theorem}

\begin{proof}
By Theorems \ref{th:critical(f,g)} and Corollary \ref{cor:bijection-Lagra}, 
the gradient $\gradientp g$ restricted on the set $\mathrm{Cri}_{LP}(f,g)\cap f^{-1}(1)$, that is, $\gradientp g|_{\mathrm{Cri}_{LP}(f,g)\cap f^{-1}(1)}:\mathrm{Cri}_{LP}(f,g)\cap f^{-1}(1)\to \mathrm{Cri}_{LP}(g^\polar ,f^\polar )\cap (g^\polar)^{-1}(1)$ is a bijection with the inverse $(\gradientp g)^{-1}|_{\mathrm{Cri}_{LP}(g^\polar ,f^\polar )\cap (g^\polar)^{-1}(1)}$ equals  
\begin{equation}\label{eq:LP-inverse}
\gradientp f^\polar |_{\mathrm{Cri}_{LP}(g^\polar ,f^\polar )\cap (g^\polar)^{-1}(1)}:\mathrm{Cri}_{LP}(g^\polar ,f^\polar )\cap (g^\polar)^{-1}(1)\to \mathrm{Cri}_{LP}(f,g)\cap f^{-1}(1) 
\end{equation}
where $\mathrm{Cri}_{LP}(f,g)$ denote the nontrivial Lagrange critical points of $(f,g)$.

Since the set $\mathrm{Cri}_{MP}(f/g)$ of Morse critical points is a subclass of $\mathrm{Cri}_{LP}(f,g)$, it is clear that $\gradientp g(\mathrm{Cri}_{MP}(f/g)\cap f^{-1}(1))\subset \mathrm{Cri}_{LP}(g^\polar ,f^\polar )\cap (g^\polar)^{-1}(1)$. 

In the sequel, we shall further prove $\gradientp g(\mathrm{Cri}_{MP}(f/g))\subset \mathrm{Cri}_{MP}(g^\polar /f^\polar )$. 

Consider the set of Morse regular points of $f/g$ and $g^\polar /f^\polar $ restricted on unit spheres:
$$MR(f/g)=\{\text{Morse regular points of }f/g\text{ restricted on }f=1\}$$
and 
$$MR(g^\polar /f^\polar )=\{\text{Morse regular points of }g^\polar /f^\polar \text{ restricted on }g^\polar =1\} .$$
We only need to show that $MR(f/g)\cap \mathrm{Cri}_{LP}(f,g)$ and $MR(g^\polar /f^\polar )\cap \mathrm{Cri}_{LP}(g^\polar ,f^\polar )$ 
are one-to-one correspondence via $\gradientp g$ and $\gradientp f^\polar$, 
that is, 
\begin{equation}\label{eq:Mregualar-Lag-g}
\gradientp g\big(MR(f/g)\cap \mathrm{Cri}_{LP}(f,g)\big)\subset MR(g^\polar /f^\polar \big)\cap \mathrm{Cri}_{LP}(g^\polar ,f^\polar )   
\end{equation} and \begin{equation}\label{eq:Mregualar-Lag-f}
\gradientp f^\polar \big(MR(g^\polar /f^\polar )\cap \mathrm{Cri}_{LP}(g^\polar ,f^\polar )\big)\subset MR(f/g)\cap \mathrm{Cri}_{LP}(f,g).
\end{equation}

In fact, if 
\eqref{eq:Mregualar-Lag-g} and \eqref{eq:Mregualar-Lag-f} hold, 
then by \eqref{eq:LP-inverse},  $\gradientp g|_{\mathrm{Cri}_{MP}(f/g)}:\mathrm{Cri}_{MP}(f/g)\to \mathrm{Cri}_{MP}(g^\polar /f^\polar )$ is a bijection. 

Let's prove \eqref{eq:Mregualar-Lag-g}. 
For a Morse regular point $x_0\in MR(f/g)\cap \mathrm{Cri}_{LP}(f,g)$, there is a decreasing flow $\eta(x,t)$ near $x_0$, i.e., there exists a neighborhood $U_{x_0}$ such that for any $x\in U_{x_0}$, $t\in(0,1]$,
$$\frac fg(\eta(x,t))<\frac fg(x) $$
and $\eta(x,0)=x$. Note that $\gradientp g(\eta(x,t))$ is a continuous flow along $\gradientp g(x)$, and $g^\polar (\gradientp g(\eta(x,t)))=1$. 
In particular, when $x=x_0$, and $t>0$,  by Lemma \ref{lem:f/g-zyx-tech}, 
$$ \frac{g^\polar (\gradientp g(\eta(x_0,t)))}{f^\polar (\gradientp g(\eta(x_0,t)))}\le \frac{f(\eta(x_0,t))}{g(\eta(x_0,t))}<\frac{f(x_0)}{g(x_0)}=\frac{g^\polar (\gradientp g(x_0))}{f^\polar (\gradientp g(x_0))}.$$

For any $y\in V_{y_0}$ where $y_0=\gradientp g(x_0)$ and $t\in(0,1]$,
\begin{align*}
\frac{g^\polar (\gradientp g(\eta(\gradientp f^\polar (y),t)))}{f^\polar (\gradientp g(\eta(\gradientp f^\polar (y),t)))}\le \frac{f(\eta(\gradientp f^\polar (y),t))}{g(\eta(\gradientp f^\polar (y),t))}<\frac{f(\gradientp f^\polar (y))}{g(\gradientp f^\polar (y))}\le\frac{g^\polar (y)}{f^\polar (y)}.
\end{align*}
However, $t\mapsto \gradientp g(\eta(\gradientp f^\polar (y),t)))$ is not a proper flow, as in general, $\gradientp g(\eta(\gradientp f^\polar (y),0)))=\gradientp g(\gradientp f^\polar (y))\ne y$. 
While we can define $\delta(y)=\|y-\gradientp g(\gradientp f^\polar (y))\|$ which satisfies that $y$ is a Lagrange critical point of $(g^\polar,f^\polar)$ iff $\delta(y)=0$. Moreover, $\delta(\cdot)$ is continuous. 
Let  $$H(y,t)=\frac{t \gradientp g(\gradientp f^\polar (y)) +(1-t)y}{g^\polar \left(t \gradientp g(\gradientp f^\polar (y)) +(1-t)y\right)},\;\; x\in X,\, 0\le t\le1.$$
It is not difficult to verify that $H$ is continuous, $H(y,0)=y$ when $g^\polar (y)=1$ and $H(y,1)=\gradientp g(\gradientp f^\polar (y))$. 

Let $\eta^\#:V_{y_0}\times[0,1]\to X$ be defined by
$$\eta^\#(y,t)=\begin{cases}
H(y,t/\delta(y)),&\text{ if }1\ge \delta(y)>t\ge0\\
\gradientp g(\eta(\gradientp f^\polar (y),t-\delta(y)))),&\text{ if }1\ge t\ge\delta(y)\ge0
\end{cases}$$
Then $\eta^\#(y,0)=y$ and by the above discussion and the 2nd inequality in Lemma \ref{lem:f/g-zyx-tech}, we have
$$\frac{g^\polar (\eta^\#(y,t))}{f^\polar (\eta^\#(y,t))}< \frac{g^\polar (y)}{f^\polar (y)},\;\forall t\in (0,1]. $$
This implies that $y_0$ is actually a Morse regular point of $g^\polar /f^\polar $. 
The proof of \eqref{eq:Mregualar-Lag-g} 
is then completed.

The proof of \eqref{eq:Mregualar-Lag-f} 
is similar, and thus we omit it. Therefore, we have proved the one-to-one correspondence between the Morse critical points of $f/g$ and $g^\polar /f^\polar $. 

We are in a position to show the isomorphisms of Rothe critical groups. 

In fact, if $\al$ is an isolated Morse critical point of $f/g$, then by the homology excision property, $C_*(f/g,\al)\cong H_*(\{f/g\le f(\al)/g(\al)\},\{f/g\le f(\al)/g(\al)\}\setminus\{\al\})$. 
Since $\gradientp g(\al)=\bet$,  $f(\al)/g(\al)=g^\polar (\tau)/f^\polar (\tau)$ (by Lemma \ref{lem:f/g-zyx-tech}), and 
$\gradientp g:\{f/g\le f(\al)/g(\al)\}\to \{g^\polar /f^\polar \le g^\polar (\tau)/f^\polar (\tau)\}$ is a homotopy equivalence, we have 
a homomorphism $$(\gradientp g)_*: C_*(f/g,\al)\to C_*(g^\polar /f^\polar ,\tau).$$
Similarly, $\gradientp f^\polar $ induces a homomorphism 
$$(\gradientp f^\polar )_*:  C_*(g^\polar /f^\polar ,\tau)\to C_*(f/g,\al).$$
Note that $(\gradientp f^\polar )_*\circ (\gradientp g)_*=(\gradientp f^\polar \circ \gradientp g)_*:C_*(f/g,\al)\to C_*(f/g,\al)$ is an isomorphism because $\gradientp f^\polar \circ \gradientp g\simeq \mathrm{id}$ by Theorem \ref{th:ratio-sublevel-equi}. Thus, we obtain that both $(\gradientp g)_*$ and $(\gradientp f^\polar )_*$ are isomorphisms between $C_*(f/g,\al)$ and $ C_*(g^\polar /f^\polar ,\tau)$. 
\end{proof}

\subsubsection{Dual equivalence on Lusternik-Schnirelman min-max critical values}

\begin{theorem}
\label{th:index-critical}Let $f,g\in \cvp_c^1(X)$. 
For any admissible 
index $\mathrm{ind}$, 
for any min-max critical value\footnote{This has been already defined in \eqref{eq:index-LS-minmax-c}. For readers' convenience, we write the formula here again.}  
$$c_k(f/g):=\inf_{\mathrm{ind}(S)\ge k}\sup_{x\in S}\frac{f(x)}{g(x)}$$
defined via such index, 
we have $c_k(f/g)=c_k(g^\polar /f^\polar )$ for any $k$. 
\end{theorem}

\begin{proof}
By definition, it can be verified that $$c_k(f/g)=\inf_{\mathrm{ind}\{f/g\le c\}\ge k }c.$$
To show $c_k(f/g)=c_k(g^\polar /f^\polar )$, it suffices to prove 
\begin{equation}\label{eq:index=dual}
\mathrm{ind}\{f/g\le c\}=\mathrm{ind}\{g^\polar /f^\polar \le c\}    
\end{equation}

We first prove the case that $f^\polar $ and $g$ are further assumed to be $C^1$-smooth (or, $f$ and $g^\polar $ are $C^1$-smooth). 

If the admissible  index is a homotopy invariant (see (I3) in Definition \ref{def:admissible-index}), then 
by Theorem \ref{th:ratio-sublevel-equi}, we have $\{f/g\le c\}\simeq \{g^\polar /f^\polar \le c\}$, and then we immediately obtain 
the equality \eqref{eq:index=dual}.

If the admissible  index satisfies the nondecreasing property under continuous map (i.e., (I4) in Definition \ref{def:admissible-index}), we can also obtain the equality \eqref{eq:index=dual}.  
First, suppose that $g$ and $f^\polar $ are $C^1$-smooth, then similar to 
Claim 1 in the proof of Theorem \ref{th:ratio-sublevel-equi-w}, the continuity of $\gradientp g$ and the nondecreasing property of $\mathrm{ind}$, 
we have $\gradientp g\{f/g\le c\}\subset \{g^\polar /f^\polar \le c\}$ and 
$$\mathrm{ind}\{f/g\le c\}\le \mathrm{ind}(\gradientp g\{f/g\le c\})\le \mathrm{ind}\{g^\polar /f^\polar \le c\}.$$
Analogously, we have 
$$\mathrm{ind}\{g^\polar /f^\polar \le c\}\le\mathrm{ind}\big(\partial f^\polar \{g^\polar /f^\polar \le c\} \big)\le \mathrm{ind}\{f/g\le c\}. $$
Thus, \eqref{eq:index=dual} holds.

Now, we remove the $C^1$-smooth assumption, and use standard approximation method, the equality $c_k(f/g)=c_k(g^\polar /f^\polar )$ can be also verified.  
\end{proof}

\subsubsection{Multiplicities of Lagrange critical values}\label{sec:multiplicity}
Based on the results in the preceding sections, we are able to study the multiplicities of Lagrange critical values. 

\begin{prop}Let $f,g\in \cvp_c^p(X)$. 
For any nontrivial Lagrange critical value $c$ of $(f,g)$, we have  
$\mathrm{mult}_{f,g} (c)=\mathrm{mult}_{g^\polar ,f^\polar } (c)$.
\end{prop}
The proof is quite challenging. In fact, we develop an approach to estimate the admissible indexes of nonlinear eigenspaces in the nonsmooth case based on the Moreau-Yosida approximation and a piecing gluing approach together with constructions (the partitions of unity), the proof also needs to make  full use of non-smooth analysis, which is quite non-trivial.
\begin{proof}
According to Propositions  \ref{pro:CV_p-basic} and \ref{pro:CV_1&Cv_p}, we only need to work with the case $p=1$. 

Next, we shall prove the multiplicity equality. 
From Theorem \ref{th:critical(f,g)} and Corollary \ref{cor:bijection-Lagra}, 
it is clear to see that  $\mathrm{ind}(\mathrm{LCri}_c(f,g))\le \mathrm{ind}(\mathrm{LCri}_c(g^\polar ,f^\polar ))$ if  $f$ or $g$ is $C^1$-smooth on $X\setminus\{0\}$, where $$\mathrm{LCri}_c(f,g)=\{x:x\text{ is a nontrivial Lagrange critical point of }(f,g)\text{ with }f(x)/g(x)=c\}.$$  
Conversely, if  $f^\polar $ or $g^\polar $ is is $C^1$-smooth on $X^*\setminus\{0\}$, 
$\mathrm{ind}(\mathrm{LCri}_c(f,g))\ge \mathrm{ind}(\mathrm{LCri}_c(g^\polar ,f^\polar ))$. Thus, we obtain that the multiplicity of $c$ as a critical value of $(f,g)$ coincides with the multiplicity of $c$ as a  critical value of $(g^\polar ,f^\polar )$. 

Next, we prove that the same property holds when we remove 
the $C^1$-smoothness  condition.

Fix an $\varepsilon>0$ such that 
\[
\mathrm{ind}\; \mathbb{B}_\varepsilon\!\!\left(\bigcup\limits_{  x\in \mathrm{LCri}_c(f,g)}\!\!\!\!\mathrm{cone}(\partial f(  x))\cap \partial g(  x)\right)\!\!=\mathrm{ind}\left(\bigcup\limits_{  x\in \mathrm{LCri}_c(f,g)}\!\!\!\!\mathrm{cone}(\partial f(  x))\cap \partial g(  x)\right).
\]
Take $\varepsilon'<\frac12\varepsilon$. 
Due to smooth approximation approaches, 
we consider a family of open sets  $\{\mathbb{B}_{\delta_x}(  x):  x\in \mathrm{LCri}_c(f,g)\}$ and the corresponding smooth function family  $\{g_{ x}:  x\in \mathrm{LCri}_c(f,g)\}$  such that for any $  y\in \mathbb{B}_{2\delta_x}(  x)$, we have 
$\mathrm{cone}(\partial f(  y))\cap \partial g(  y)\subset \mathbb{B}_{\varepsilon'}(\mathrm{cone}(\partial f(  x))\cap \partial g(  x))$ and 
\begin{equation}\label{eq:partial-g_x(delta_x)}
\partial g_x( \mathbb{B}_{2\delta_x}(  x))\subset \mathbb{B}_{\varepsilon'}(\mathrm{cone}(\partial f(  x))\cap \partial g(  x))    
\end{equation}
for a sufficiently small~$\delta_x>0$, where $g_x\in \cvp_c^1(X)$ and is $C^1$-smooth in a neighborhood of $x$. 

Since $\mathrm{LCri}_c(f,g)$ is paracompact and $\{\mathbb{B}_{\delta_x}(  x):  x\in \mathrm{LCri}_c(f,g)\}$  induces an open cover of $\mathrm{LCri}_c(f,g)$, we can take a locally finite subfamily $\{\mathbb{B}_{\delta_i}(  x_i)\}$ of $\{\mathbb{B}_{\delta_x}(  x):  x\in \mathrm{LCri}_c(f,g)\}$  such that 
$\partial  g_i( \mathbb{B}_{2\delta_i}(  x_i))\subset \mathbb{B}_{\varepsilon'}(\mathrm{cone}(\partial  f(   x_i))\cap \partial  g(   x_i))$, where we simply write $ g_{ x_i}$ as $ g_i$.  
Then, there exist  
partitions of unity  $\{\psi_i\}$  subordinate to the open cover $\{\mathbb{B}_{\delta_i}(  x_i)\}$, i.e., $\mathrm{supp}(\psi_i)\subset \mathbb{B}_{\delta_i}(  x_i)$, $\psi_i\ge0$, $\sum_i\psi_i=1$. 
For example, we can simply take 
$$\psi_i(  y)=\frac{\max\{0,\delta_i-\|  y-  x_i\|_2\}}{\sum_j\max\{0,\delta_j-\|  y-  x_j\|_2\}},\;\;\forall   y\in X.$$
Taking $\Psi(  x)=\sum_i\psi_i(  x)\gradientp  g_i(  x)$, then $\Psi$ is a  continuous map. 

Given $  x\in \mathrm{LCri}_c(f,g)$, let $I(  x)=\{i:  x\in \mathbb{B}_{\delta_i}(  x_i)\}$. 
Then, $I(x)$ is a finite set due to the locally finite property of $\{\mathbb{B}_{\delta_i}(x_i)\}$. 
Note that $\psi_i(  x)>0$ implies $  x\in \mathbb{B}_{\delta_i}(  x_i)$, and thus it holds $\Psi(  x)=\sum_{i\in I(  x)}\psi_i(  x)\gradientp  g_i(  x)$ and  $\gradientp  g_i(  x)\in \mathbb{B}_{\varepsilon'}(\mathrm{cone}(\partial  f(   x_i))\cap \partial  g(   x_i))$, whenever $  x\in \mathbb{B}_{\delta_i}(  x_i)$.  
 
Let  $i(x)=\mathrm{argmax}\{\delta_i:i\in I(  x)\}$. 
Then, for any $i\in I(  x)$, $  x_i\in \mathbb{B}_{\delta_i}(  x) \subset\mathbb{B}_{\delta_i}(\mathbb{B}_{\delta_{i(x)}}(  x_{i(x)})) =\mathbb{B}_{\delta_i+\delta_{i(x)}}(  x_{i(x)})\subset  \mathbb{B}_{2\delta_{i(x)}}(  x_{i(x)})$. Thus, it follows from \eqref{eq:partial-g_x(delta_x)} that $\forall i\in I(  x)$, $\mathrm{cone}(\partial  f(   x_{i}))\cap \partial  g(   x_{i})\subset \mathbb{B}_{\varepsilon'}(\mathrm{cone}(\partial  f(   x_{i(x)}))\cap \partial  g(   x_{i(x)}))$.  
Therefore, $\partial g_i(  x)\in \mathbb{B}_{2\varepsilon'}(\mathrm{cone}(\partial  f(   x_{i(x)}))\cap \partial  g(   x_{i(x)}))$ for any $i\in I(  x)$. Consequently, we have
\begin{align*}
\Psi(  x)&=\sum_{i\in I(  x)}\psi_i(  x)\gradientp  g_i(  x)\in \mathbb{B}_{2\varepsilon'}(\mathrm{cone}(\partial  f(   x_{i(x)}))\cap \partial  g(   x_{i(x)}))
\\&\subset \mathbb{B}_\varepsilon\left(\bigcup\limits_{  x\in \mathrm{LCri}_c(f,g)}\mathrm{cone}(\partial  f(   x))\cap \partial  g(   x)\right)
\end{align*}
which implies that $\Psi(\mathrm{LCri}_c(f,g))\subset \mathbb{B}_\varepsilon\left(\cup_{  x\in \mathrm{LCri}_c(f,g)}\mathrm{cone}(\partial  f(   x))\cap \partial  g(   x)\right)$.  Thus, 
\begin{align*}
\mathrm{ind}(\mathrm{LCri}_c(f,g))&\le 
\mathrm{ind}(\Psi(\mathrm{LCri}_c(f,g))
\le \mathrm{ind}\; \mathbb{B}_\varepsilon\left(\bigcup\limits_{  x\in \mathrm{LCri}_c(f,g)}\mathrm{cone}(\partial  f(   x))\cap \partial  g(   x)\right)
\\&=\mathrm{ind}\left(\bigcup\limits_{  x\in \mathrm{LCri}_c(f,g)}\mathrm{cone}(\partial  f(   x))\cap \partial  g(   x)\right)
\end{align*}
where the first inequality is due to the fact that  $\Psi$ is 
continuous,  the second inequality is based on the nondecreasing property of the index, and the last equality follows from the continuity of the index.   

In summary, we have proved that for any  
 critical value $c$ of $(f,g)$ there always holds  
\[
\mathrm{ind}\big(\mathrm{LCri}_c( g^\polar , f^\polar )\big)\ge \mathrm{ind}\left(\bigcup\limits_{  x\in \mathrm{LCri}_c(f,g)}\mathrm{cone}(\partial  f(   x))\cap \partial  g(   x)\right)\ge\mathrm{ind}\big( \mathrm{LCri}_c(f,g)\big)
\] 
where the first inequality is due to $\bigcup\limits_{  x\in \mathrm{LCri}_c(f,g)}\mathrm{cone}(\partial  f(   x))\cap \partial  g(   x)\subset \mathrm{LCri}_c( g^\polar , f^\polar ) $. 
By replacing $f$ with $f^\polar$ and $g$ with $g^\polar$, we obtain the reverse inequality.
\end{proof}

\subsection{\textbf{Part II. The second critical duality theory (with bounded operator)}}\label{sec:second-dual-bounded-ope}

In this section, we introduce the linear operators as a new component of critical duality theory presented in Section \ref{sec:first-c-dual}. With the linear operator, we can work on different spaces,  
and there are more fruitful duality results. We use $ \mathcal{B}(X,Y)$ to denote the set of all bounded linear operators from $X$ to $Y$. 

We list some basic properties of convex functions and bounded operators which will be used.
\begin{prop}
\label{prop:basic-in-convex}
Let $f\in \cvx(Y)$ and $A\in \mathcal{B}(X,Y)$. Then the following  properties hold:
\begin{enumerate}[(i)]
\item 
Let there be a point $Ax_0$ where $f$ is continuous and finite. Then for all points $x$ of $X$, we have $\partial (f\circ A)(x)=A^*\partial f(Ax)$. (see Proposition 5.7 in Chapter 1 of \cite{Ekeland99})
\item $R(A)^\bot=\ker A^*$, $(\ker A)^\bot=\overline{R( A^*)}$, $R(A^*)^\bot=\ker A$, $(\ker A^*)^\bot=\overline{R( A)}$ (see \cite{Rudin})
\item $R(A)$ is closed if and only if $R(A^*)$ is closed  (see \cite{Rudin})
\end{enumerate}
\end{prop}

\subsubsection{Lagrange critical equivalence 
}

According to Theorem
\ref{th:critical(f,g)}, we have the following conclusion.
\begin{corollary}
Let $g\in \cvx_{0}^p(X)$ and $f\in \cvx_{0}^p(Y)$ and $A\in \mathcal{B}(X,Y)$. Then, the nontrivial Lagrange critical values of  $(f\circ A,g)$ coincide with that of $(g^\circ,(f\circ A)^\circ)$.
\end{corollary}

We can further prove a duality theorem that involves the adjoint operator $A^*$.  

\begin{theorem}\label{th:critical(fA,g)}
Let $g\in \cvx_{0}^p(X)$ and $f\in \cvx_{0}^p(Y)$ and $A\in \mathcal{B}(X,Y)$. 
Suppose that $g^\polar$ and $f$ are continuous. 
Then the nontrivial Lagrange critical values of  $(f\circ A,g)$ coincide with that of  $(g^\polar\circ A^* , f^\polar )$. 

Furthermore, for any nontrivial Lagrange critical point $x$ of $(f\circ A,g)$, and for any $w\in\mathrm{cone}(\partial f(A x))\cap  \mathrm{cone}((A^*)^{-1}\partial g( x))$,  $w$ is a nontrivial Lagrange critical point of $(g^\polar\circ A^*  , f^\polar )$. 
\end{theorem}

\begin{proof}
By Propositions  \ref{pro:CV_p-basic} and \ref{pro:CV_1&Cv_p}, we only need to work with the case $p=1$, that is, $g\in \cvx_{0}^1(X)$ and $f\in \cvx_{0}^1(Y)$.  

We shall use the chain rule for subdifferential: $\partial (f\circ A)(x)=A^* \partial f(Ax)$ for any $x\in X$, and $\partial ( g^\polar \circ A^*)( w)=A\partial  g^\polar (A^* w)$ for any $w\in Y^*$ (see Proposition \ref{prop:basic-in-convex}).

Suppose that $(c,x)$ is a nontrivial Lagrange critical pair of $(f\circ A,g)$, i.e., $c$, $f(Ax)$ and $g(x)$ are positive numbers, and $\partial(f\circ A)(x) \cap c\partial g(x)\ne\varnothing$. Then there exists $u\in\partial g(x)$ such that $c u\in \partial (f\circ A)(x)=A^* \partial f(Ax)$. 
Note that $\partial g$ and $\partial (f\circ A)$ possess the zero-homogeneity, i.e., $\partial (f\circ A)(tx)=\partial (f\circ A)(x)$ and $\partial g(tx)=\partial g(x)$, $\forall t>0$, $\forall x\in X$, 
Hence, without loss of generality, we may further assume that $x$  satisfies  $g(x)=1$. Then, for any $u\in\partial g(x)\subset X^*$, 
it follows from Euler's identity that $\langle u,x\rangle=g(x)=1$ and $c = \langle c u,x\rangle=f(Ax)$. 

By definition of subgradient, for any $x'\in X$, $$\langle u,x'\rangle-1=\langle u,x'\rangle-\langle u,x\rangle=\langle u,x'-x\rangle\le g(x')-g(x)=g(x')-1$$ which yields $\langle u,x'\rangle\le g(x')$, $\forall x'\in X$. Thus, it can be checked that $$ g^\polar (u)=\sup_{x'\in X:g(x')\le1}\langle u,x'\rangle=\langle u,x\rangle=g(x)=1.$$
Let $J:X\to X^{**}$ be the canonical map, i.e., for each  $x'\in X$, $\langle Jx',u'\rangle=\langle u',x'\rangle$, $\forall u'\in X^*$. Then, 
for any $u'\in X^*$, 
$$\langle Jx,u'-u\rangle=\langle u'-u,x\rangle=\langle u',x\rangle-\langle u,x\rangle=\langle u',x\rangle-1\le  g^\polar (u')-1=  g^\polar (u')- g^\polar (u)$$ which implies that $Jx\in \partial   g^\polar (u)$. 
For convenience, we identify $Jx$ with $x$, and simply 
rewrite $x\in \partial  g^\polar (u)$. 


By $c u\in\partial (f\circ A)(x)=A^*\partial f(Ax)=A^*\partial f(\frac{Ax}{c})$, there exists $w\in\partial f(\frac{Ax}{c})$ such that $A^*w=cu$. Note that $f(\frac{Ax}{c})=1$, for any $y\in Y$, 
$$\langle w,y\rangle-1=\langle w,y\rangle-f(\frac{Ax}{c})=\langle w,y\rangle-\langle w,\frac{Ax}{c}\rangle=\langle w,y-\frac{Ax}{c}\rangle\le f(y)-f(\frac{Ax}{c})=f(y)-1$$
which implies $\langle w,y\rangle\le f(y)$, $\forall y\in X$. Thus, it can be checked that $$ f^\polar (w)=\sup_{y\in X:f(y)\le1}\langle w,y\rangle=\langle w,\frac{Ax}{c}\rangle=f(\frac{Ax}{c})=1.$$
Let $J':Y\to Y^{**}$ be the natural map, i.e., for each  $y\in Y$, $\langle Jy,v\rangle=\langle v,y\rangle$, $\forall v\in Y^*$.
Then, 
$$\langle J'\frac{Ax}{c},v-w\rangle=\langle v-w,\frac{Ax}{c}\rangle=\langle v,\frac{Ax}{c}\rangle-\langle w,\frac{Ax}{c}\rangle=\langle v,\frac{Ax}{c}\rangle-1\le  f^\polar (v)-1=  f^\polar (v)- f^\polar (w)$$ 
for any $v\in Y^*$, 
and therefore by identifying $J'\frac{Ax}{c}$ with $\frac{Ax}{c}$, we have
$$\frac{Ax}{c}\in \partial   f^\polar (w)
.$$
Accordingly, $Ax\in c \partial  f^\polar (w)$, and then by the zero-homogeneity of $\partial   g^\polar $ and chain rule, we have 
\begin{align*}
Ax\in   A\partial  g^\polar (u)\bigcap c \partial   f^\polar ( w)&=A\partial  g^\polar (cu)\bigcap c \partial   f^\polar ( w)
\\&=A\partial  g^\polar (A^* w)\bigcap c \partial   f^\polar ( w)=\partial ( g^\polar \circ A^*)( w)\bigcap c \partial   f^\polar ( w)    
\end{align*}
implying that $c$ is a nontrivial Lagrange critical value of $(  g^\polar\circ A^*,  f^\polar )$ with a corresponding nontrivial Lagrange critical point $w$. 

Note that in the above process, we only require $u\in \partial g(x) \cap \frac{1}{c}\partial (f\circ A)(x)$. By the zero-homogeneity of $\partial g$ and $\partial f$, we can relax this condition to $u\in \mathrm{cone}(\partial g(x)) \cap  \mathrm{cone}(\partial (f\circ A)(x))$. 
Since $X$ is reflexive, the canonical map $J$ gives an isomorphism between $X$ and $X^{**}$, the nontrivial Lagrange  critical values of $(g^\polar\circ A^* ,f^\polar )$ should also be nontrivial Lagrange critical values of $(f\circ A,g)$. 
The proof is then completed.
\end{proof}

We have the following kernel reduction property:

\begin{lemma}[kernel reduction lemma]\label{lem:g_R(A)g_K(A)}
For any $A\in \mathcal{B}(X,Y)$, $f\in\cvx_{0}^p(Y)$ and $g\in\cvx_{0}^p(X)$, suppose that $g^\polar$ and $f$ are continuous. 
The nonzero Lagrange critical values of $(f\circ A,g)$, $(f,\bar g_{\RA})$ and $(f\circ A ,g_{\KA})$ coincide, where $g_{\KA}:X\to [0,+\infty]$ is defined by 
$$ g_{\KA}(x)=\inf_{x'\in x+\mathrm{Kernel}(A)} g(x')$$
and  $\bar g_{\RA}:\mathrm{Range}(A)\to [0,+\infty]$ is defined by 
$$ \bar g_{\RA}(y)=\inf_{x\in A^{-1}y} g(x)$$
\end{lemma}
The proof is surprisingly based on the critical duality theorems, i.e., Theorems \ref{th:critical(f,g)} and  \ref{th:critical(fA,g)}.
\begin{proof}By Propositions  \ref{pro:CV_p-basic} and \ref{pro:CV_1&Cv_p}, we only need to verify the case $p=1$, that is, $g\in \cvx_{0}^1(X)$ and $f\in \cvx_{0}^1(Y)$. 
The core of the proof is then to prove: 
For any $y^*\in Y^*$, 
\begin{equation}\label{eq:g-three-terms-dualequal}
g^\polar(A^*y^*)=(g_{\KA})^\polar(A^*y^*)=(\bar g_{\RA})^\polar(y^*)    
\end{equation}

We split the proof of \eqref{eq:g-three-terms-dualequal} into the following steps. 

\begin{enumerate}[{Step }1.]
    \item $g^\polar(A^*y^*)=(\bar g_{\RA})^\polar(y^*)$


In fact, for any $x\in X$ with $g(x)\le 1$, 
$$\bar g_{\RA}(Ax)=\inf_{x'\in X:Ax'=Ax}g(x')\le g(x)\le 1.$$
Hence, 
$$g^\polar(A^*y^*)=\sup_{x\in X:g(x)\le1}\langle A^*y^*,x\rangle=\sup_{x\in X:g(x)\le1}\langle y^*,Ax\rangle\le \sup_{y\in Y:\bar g_{\RA}(y)\le1}\langle y^*,y\rangle=(\bar g_{\RA})^\polar(y^*).$$

On the other hand, for any $y\in \mathrm{Range}(A)$ with $\bar g_{\RA}(y)\le1$, and for any  $\varepsilon>0$, there exists $x\in X$ such that $Ax=y$ and $g(x)< \bar g_{\RA}(y)+\varepsilon\le 1+\varepsilon$. 
Since $g$ is positively 1-homogeneous, we have  
\begin{align*}
(\bar g_{\RA})^\polar(y^*)&=\sup_{y\in Y:\bar g_{\RA}(y)\le1}\langle y^*,y\rangle\le \sup_{x\in X:g(x)\le1+\varepsilon}\langle y^*,Ax\rangle
\\&=\sup_{x\in X:g(x)\le1+\varepsilon}\langle A^*y^*,x\rangle=(1+\varepsilon)g^\polar(A^*y^*)
\end{align*}
By the arbitrariness of $\varepsilon>0$, we derive $g^\polar(A^*y^*)=(\bar g_{\RA})^\polar(y^*)$. 
\item $(g_{\KA})^\polar(A^*y^*)=(\bar g_{\RA})^\polar(y^*) $

The proof is straightforward by the definition of $g_{\KA}$ and $\bar g_{\RA}$:
\begin{align*}
(g_{\KA})^\polar(A^*y^*)&=\sup_{x\in X:g_{\KA}(x)\le1}\langle A^*y^*,x\rangle
=\sup_{x\in X:g_{\KA}(x)\le1}\langle y^*,Ax\rangle
=\sup_{x\in X:\bar g_{\RA}(Ax)\le1}\langle y^*,Ax\rangle
\\&=\sup_{y\in Y:\bar g_{\RA}(y)\le1}\langle y^*,y\rangle
=
(\bar g_{\RA})^\polar(y^*)    
\end{align*}
 \end{enumerate}

We have then verified \eqref{eq:g-three-terms-dualequal} via Step 1 and Step 2.  
Therefore, $(g^\polar \circ A^*,f^\polar)$, $((\bar g_{\RA})^\polar ,f^\polar)$, $((g_{\KA})^\polar \circ A^*,f^\polar)$ are in fact the same function pair, and thus the Lagrange critical values of $(g^\polar \circ A^*,f^\polar)$, $((\bar g_{\RA})^\polar ,f^\polar)$ and $((g_{\KA})^\polar \circ A^*,f^\polar)$ coincide. 

By our duality theorems, the nonzero Lagrange critical values of $(g^\polar \circ A^*,f^\polar)$ and $(f\circ A,g)$ coincide (see Theorem \ref{th:critical(fA,g)}); the nonzero Lagrange critical values of $((\bar g_{\RA})^\polar ,f^\polar)$ and $(f ,\bar g_{\RA})$ coincide (see Theorem \ref{th:critical(f,g)});  the nonzero Lagrange critical values of $((g_{\KA})^\polar \circ A^*,f^\polar)$ and $(f\circ A ,g_{\KA})$ coincide (see Theorem \ref{th:critical(fA,g)}). Thus, the proof of Lemma \ref{lem:g_R(A)g_K(A)} is finished. 
\end{proof}

\subsubsection{Level set homotopy equivalence}

In order to study Morse criticality of RC functions, we need to first investigate the topological properties of sublevel sets, which is also useful in min-max criticality and applications. 
It is natural to consider a RC function $f\circ A/g$ with $A\in\mathcal{B}(X,Y)$, $f\in \cvx_{0,+}^p(Y)$ and $g\in \cvx_{0,+}^p(X)$.

We simply denote $$\{f\circ A/g\le c\}:=\Big\{x\in X\setminus\{0\}\Big|\;\frac{f(Ax)}{g(x)}\le c\Big\}.$$

\begin{lemma}\label{lem:fA/g-zyx-tech1}
For any $A\in \mathcal{B}(X,Y)$, $f\in\cvx_{0,+}^p(Y)$ and $g\in\cvx_{0,+}^p(X)$. For any $x\not\in \ker A$, for any $y'\in\partial f(Ax)$ and any $y$ with  $A^* y\in \partial g(x)$, we have $f^\polar(y)\ne 0$ and $f^\polar (y')\ne0$, and $$\frac{g^\polar (A^* y')}{f^\polar (y')}\ge \frac{f(Ax)}{g(x)}\ge \frac{g^\polar (A^* y)}{f^\polar (y)}$$
with each equality if and only if $x$ is a Lagrange critical point of $(f\circ A,g)$. 
\end{lemma}
\begin{proof}
By Propositions  \ref{pro:CV_p-basic} and \ref{pro:CV_1&Cv_p}, we only need to work with the case $p=1$, that is, $g\in \cvx_{0,+}^1(X)$ and $f\in \cvx_{0,+}^1(Y)$. 

There are two inequalities to be verified. We first prove the inequality involving $y'$. 

Since $Ax\ne0$ and  $y'\in\partial f(Ax)$, we have $\langle y',Ax\rangle=f(Ax)>0$, and for any $z$,
$$\langle y',z-Ax\rangle\le f(z)-f(Ax)$$
i.e., $\langle y',z\rangle\le f(z)$, which implies
$$f^\polar (y')=\sup_{z\ne0}\frac{\langle y,z\rangle}{f(z)}=1.$$
Moreover, 
$$g^\polar (A^* y')=\sup_{x'\ne0}\frac{\langle A^* y',x'\rangle}{g(x')}\ge \frac{\langle A^* y',x\rangle}{g(x)}=\frac{\langle   y',Ax\rangle}{g(x)}=\frac{f(Ax)}{g(x)}.$$
The equality holds if and only if $g^\polar (A^* y')g(x)=\langle A^* y',x\rangle$, if and only if $A^* y'/g^\polar (A^* y')\in \partial g(x)$. Also, notice that $A^* y'\in A^* \partial f(Ax)=\partial (f\circ A)(x)$, which fulfills the chain rule on composition relation $X\mathop{\rightarrow}\limits^{A} Y\mathop{\rightrightarrows}\limits^{\partial f} Y^*\mathop{\rightarrow}\limits^{A^*}X^*$. In consequence, We finally obtain that $A^* y'/g^\polar (A^* y')\in \partial g(x)$ if and only if  $$A^* y'\in \partial (f\circ A)(x)\bigcap g^\polar (A^* y')\partial g(x)$$
meaning that 
$x$ is a Lagrange critical point of $(f\circ A,g)$. 

We then prove the inequality involving $y$. 
It follows from $A^* y\in  \partial g( x)$ that $g^\polar (A^* y)=1$ and $ x/g(x)\in \partial g^\polar (A^* y)$ and $\langle y,A  x\rangle=\langle A^* y,  x\rangle=g(x)$. Thus,
$$f^\polar (y)=\sup_{z}\frac{\langle y,z\rangle}{f(z)}\ge \frac{\langle y,Ax\rangle}{f(Ax)}=\frac{g(x)}{f(A x)}$$
which implies
$$\frac{g^\polar (A^* y)}{f^\polar (y)}=\frac{1}{f^\polar (y)} \le\frac{f(Ax)}{g(x)}. $$
The analysis of the equality case is the same to the inequality involving $y'$ as discussed before.
\end{proof}

\begin{lemma}\label{lem:fA/g-zyx-tech2}
For any $A\in \mathcal{B}(X,Y)$, $f\in\cvx_{0,+}^p(Y)$ and $g\in\cvx_{0,+}^p(X)$. For any $x\not\in \ker A$, and  $y\in Y^*$ with $A^*y\in\partial g(x)$, and $z\in X$ with $J(Az)\in \partial f^\circ (y)$, we have $y\ne0$ and $tz+(1-t)x/f(Ax)\ne0$ for any  $t\in [0,1]$, where map $J:Y\to Y^{**}$ is the natural map. 
Furthermore, we have
$$\frac{f(tAz+(1-t)Ax/f(Ax))}{g(tz+(1-t)x/f(Ax))}\le \frac{f(Ax)}{g(x)},\;\forall t\in [0,1]$$
\end{lemma}
\begin{proof}
Taking $x'=x/f(Ax)$, we have $f(Ax')=1$ and $f(Ax')/g(x')=f(Ax)/g(x)$. Then, we may assume without loss of generality that $f(Ax)=1$. 
According to Euler's identity and the inequality in Lemma \ref{lem:fA/g-zyx-tech1}, we have
\begin{align*}
g(tz+(1-t)x)-g(x)&\ge \langle A^*y,tz+(1-t)x-x\rangle
\\&=t\langle A^*y,z-x\rangle
\\&=t\langle A^*y,z\rangle- t\langle A^*y,x\rangle
\\&=t\langle y,Az\rangle- tg(x)
\\&=t\langle J(Az),y\rangle- tg(x)
\\&=tf^\polar (y)-tg(x)
\\&=t\Big(\frac{f^\polar (y)}{g^\polar (A^*y)}-\frac{g(x)}{f(Ax)}\Big)\ge0
\end{align*}
where the last equality is due to the fact that $f(Ax)=1$ and $g^\polar (A^*y)=1$
. 
This implies $g(tz+(1-t)x)\ge g(x)>0$. 

By $J(Az)\in \partial f^\polar (y)$, $f(Az)=1$. 
By the convexity of $f$, $f(A(tz+(1-t)x))\le tf(Az)+(1-t)f(Ax)=1=f(Ax) $, we obtain 
$$  \frac{f(A(tz+(1-t)x))}{g(tz+(1-t)x)}\le \frac{f(Ax)}{g(x)} .$$
Removing the restriction $f(Ax)=1$ we obtain  Lemma \ref{lem:fA/g-zyx-tech2} via renormalization.  
\end{proof}

For a Fredholm operator $A:X\to Y$, let $\mathrm{indF}(A)=\dim \ker A-\mathrm{codim}\, R(A)$ denote the \emph{Fredholm index} of $A$. 

Using the above established inequalities in Lemmas  \ref{lem:fA/g-zyx-tech1} and \ref{lem:fA/g-zyx-tech2}, we show level sets homotopy equivalence of the primal RC
function $f\circ A/g$ and its dual RC
function $g^\polar\circ A^*/f^\polar$ under mild regularity assumptions.

\begin{theorem}\label{th:ratio-sublevel-equiA}
Let $X$ and $Y$ be Hilbert spaces. 
Let $A:X\to Y$ be a Fredholm operator, $f\in\cvp_c^p(Y)$ and $g\in\cvp_c^p(X)$. Assume that $f^\polar $ and $g$ are $C^1$-smooth (or, $f$ and $g^\polar $ are $C^1$-smooth) at nonzeros. 
Then, for any $c>0$, we have the following homotopy equivalence: 
$\{f\circ A/g_{\KA}\le c\}\cap  \mathrm{ker}(A)^\bot\simeq \{g^\polar \circ A^*/(f^\polar)_{\KAs} \le c\}\cap \mathrm{ker}(A^*)^\bot$, and if $\mathrm{indF}(A)=0$, we have $\{f\circ A/g\le c\}\simeq\{g^\polar\circ A^*/f^\polar\le c\}$. 
\end{theorem}

\begin{proof}



We first show that $\{f\circ A/g_{\KA}\le c\}\cap  \mathrm{ker}(A)^\bot\simeq \{g^\polar \circ A^*/(f^\polar)_{\KAs} \le c\}\cap \mathrm{ker}(A^*)^\bot$. 

Since $g$ is $C^1$-smooth at nonzeros, for any $x\not\in\ker A$, $\partial g_{\KA}(x)=\partial g(x_{\KA})\cap (\ker A)^\bot$ is single-valued, where $x_{\KA}\in x+\ker A$ is a point such that $g(x_{\KA})=g_{\KA}(x)$, the existence of which is due to the fact that $\dim \ker A<\infty$  (since $A$ is Fredholm). Moreover, $g_{\KA}\in \cvp_c^p((\ker A)^\bot)$, we then have that $\partial g_{\KA}$ is continuous on $(\ker A)^\bot\setminus\{0\}$. Moreover, $\partial g_{\KA}(x)\subset (\ker A)^\bot=R(A^*)$ as $R(A^*)$ is closed (since $A$ is Fredholm).  
We can then define $(A^*|_{\mathrm{ker}(A^*)^\bot})^{-1}\partial g_{\KA}(x)$ to be the unique point $y$ in $\mathrm{ker}(A^*)^\bot$ with $A^* y=\partial g_{\KA}(x)$. Similarly, $(A|_{\mathrm{ker}(A)^\bot})^{-1}\partial (f^\polar)_{\KAs} $ is also continuous. 
Let $\varphi(x)=f(Ax)(A^*|_{\mathrm{ker}(A^*)^\bot})^{-1}\partial g_{\KA}(x)$ and $\psi(y)=g^\polar(A^*y)(A|_{\mathrm{ker}(A)^\bot})^{-1}\partial (f^\polar)_{\KAs}(y)$. 
We will prove that $\varphi$ and $\psi$ are homotopy equivalences for 
$$X_c:=\Big\{x\in X\Big|\;\frac{f(Ax)}{g_{\KA}(x)}\le c\Big\}\bigcap \,\mathrm{ker}(A)^\bot \;{\Large \simeq}\; \left\{x\in Y^*\left|\;\frac{g^\polar ( A^*x)}{(f^\polar)_{\KAs} (x)}\le c\right.\right\}\bigcap\, \mathrm{ker}(A^*)^\bot=:Y_c^*.$$



By Lemma \ref{lem:g_R(A)g_K(A)},
$$\frac{((f^\polar)_{\KAs})^\polar(Ax)}{g_{\KA}(x)}  = \frac{f(Ax)}{g_{\KA}(x)} \;\text{ and }\; \frac{(g_{\KA})^\polar ( A^*y)}{(f^\polar)_{\KAs} (y)} = \frac{g^\polar ( A^*y)}{(f^\polar)_{\KAs} (y)}.$$
Then, by Lemma \ref{lem:fA/g-zyx-tech1}, 
$$\frac{g^\polar(A^*\varphi(x))}{(f^\polar)_{\KAs}(\varphi(x))}
=\frac{(g_{\KA})^\polar(A^*\varphi(x))}{(f^\polar)_{\KAs}(\varphi(x))}
=\frac{(g_{\KA})^\polar(A^*(A^*|_{\mathrm{ker}(A^*)^\bot})^{-1}\partial g_{\KA}(x))}{(f^\polar)_{\KAs}((A^*|_{\mathrm{ker}(A^*)^\bot})^{-1}\partial g_{\KA}(x))}\le\frac{f(Ax)}{g_{\KA}(x)}$$
which implies 
$\varphi(X_c)\subset Y_c^*$. Similarly, $\psi(Y_c^*)\subset X_c$. 
Define the map $h:X_c\to X_c$ by 
$$ h(x)=(A|_{\mathrm{ker}(A)^\bot})^{-1}\partial (f^\polar)_{\KAs} \big((A^*|_{\mathrm{ker}(A^*)^\bot})^{-1}\partial g_{\KA}(x)\big) $$
Note that $h(x)$ is actually the element $z$ appearing in Lemma \ref{lem:fA/g-zyx-tech2}, and thus $th(x)+(1-t)x/f(Ax)\ne0$ for any $t\in[0,1]$. And by the definition of $h(x)$, we have $h(x)\in (\ker A)^\bot$. Since $x\in (\ker A)^\bot$, we finally obtain $th(x)+(1-t)x/f(Ax)\in (\ker A)^\bot\setminus\{0\}$. In consequence, $A(th(x)+(1-t)x/f(Ax))\ne0$. 
Note that $$\psi(\varphi(x))=f(Ax)(A|_{\mathrm{ker}(A)^\bot})^{-1}\partial (f^\polar)_{\KAs} \big((A^*|_{\mathrm{ker}(A^*)^\bot})^{-1}\partial g_{\KA}(x)\big)=f(Ax)h(x).$$ 
Let  $$H(x,t)=t\psi\circ\varphi(x)+(1-t)x= f(Ax)\big(th(x)+(1-t)x/f(Ax)\big),\, x\in X_c,\, 0\le t\le1.$$
It is easy to see that $H$ is continuous, $H(\cdot,0)=\mathrm{id}|_{X_c}$ and $H(x,1)=\psi\circ\varphi(x)$. 

For any $x\in X_c$, by Lemma \ref{lem:fA/g-zyx-tech2}, 
we obtain 
$$ \frac{f(AH(x,t))}{g_{\KA}(H(x,t))}=\frac{f(tAh(x)+(1-t)Ax/f(Ax))}{g_{\KA}(th(x)+(1-t)x/f(Ax))}\le \frac{f(Ax)}{g_{\KA}(x)} \le c$$
Thus we have proved that 
$H(X_c,t)\subset X_c$,  $\forall t\in [0,1]$. 

In consequence,  $\psi\circ\varphi$ is homotopic to the identity $\mathrm{id}|_{X_c}$. By the same approach, $\varphi\circ\psi$ is homotopic to the identity $\mathrm{id}|_{Y^*_c}$. 
Hence, $\varphi$ is a homotopy equivalence, and $\psi$ a homotopy inverse to $\varphi$. 
We then proved that the topological spaces $X_c$ and $Y^*_c$ are homotopy equivalent.

In the next part, we show that $\{f\circ A/g\le c\}\simeq (\{f\circ A/g_{\KA}\le c\}\cap  \mathrm{ker}(A)^\bot)*\mathbb{S}^{\dim \mathrm{ker}(A)-1}$. 
Clearly, for any $x\ne0$, 
$$\frac{f(Ax)}{g(x)}\le \frac{f(Ax)}{g_{\KA}(x)},$$
which implies 
$\{f\circ A/g_{\KA}\le c\}\subset \{f\circ A/g\le c\}$. 
Since $X=(\ker A)^\bot\oplus \ker A$, for any $x\in X$, there is a unique  direct sum decomposition which will be denoted as $x=x_\bot+x_A$, where $x_\bot\in (\ker A)^\bot$ and $x_A\in \ker A$. 
Note that each slice $\{x_\bot+x_A\ne0:x_A\in \ker A,\,f(Ax_\bot)/c\le g(x_\bot+x_A)\}$ is homotopy equivalent to either $\R^{\dim \ker A}$ or $\mathbb{S}^{\dim \ker A-1}$, depending on whether $g_{\KA}(x_\bot)\ge f(Ax_\bot)/c$. 
Precisely, \begin{align*}
&\{x_\bot+x_A\ne0:x_A\in \ker A,\,f(Ax_\bot)/c\le g(x_\bot+x_A)\}
\\ \cong\;& \begin{cases}
\ker A, & \text{ if } f(Ax_\bot)/c\le g_{\KA}(x_\bot)\text{ and }x_\bot\ne0 \\
\{x_A\in \ker A:f(Ax_\bot)/c\le g(x_\bot+x_A)\}& \text{ if }f(Ax_\bot)/c> g_{\KA}(x_\bot) \text{ and }x_\bot\ne0 \\
\ker A\setminus\{0\} & \text{ if }x_\bot=0.
\end{cases}   
\\ \simeq\; &
\begin{cases}
\R^{\dim \ker A}, & \text{ if } x_\bot\in \{f\circ A/g_{\KA}\le c\}\cap  \mathrm{ker}(A)^\bot\\
\mathbb{S}^{\dim \ker A-1}, & \text{ otherwise} 
\end{cases}
\end{align*}
This is because $g$ restricted on each slice $x_\bot+\ker A$ is a continuous convex function with compatible increasing property with the norm, which implies that the upper level set $\{x_A\in \ker A:g(x_\bot+x_A)\ge f(Ax_\bot)/c\}$ is either the whole subspace $\ker A$ or it is homotopy equivalent to $\mathbb{S}^{\dim \ker A-1}$.  
It is then not difficult to derive 
\begin{align*}
\{f\circ A/g\le c\}&=\{x\ne0:f(Ax)/c\le g(x)\}
\\&=\bigcup_{x_\bot\in (\ker A)^\bot}\{x_\bot+x_A\ne0:x_A\in \ker A,\,f(Ax_\bot)/c\le g(x_\bot+x_A)\} 
\\ &\simeq\;   \left(\{f\circ A/g_{\KA}\le c\}\cap  \mathrm{ker}(A)^\bot\right) * \mathbb{S}^{\dim \ker A-1}
\end{align*}
By the same reason, we have $$\{g^\polar\circ A^*/f^\polar\le c\}\simeq \left(\{g^\polar \circ A^*/(f^\polar)_{\KAs} \le c\}\cap \mathrm{ker}(A^*)^\bot\right) * \mathbb{S}^{\dim \ker A^*-1}.$$
Together with all the discussions above, one can prove in a similar manner that if $\dim \ker A= \dim \ker A^*$, then 
$\{f\circ A/g\le c\}\simeq \{g^\polar\circ A^*/f^\polar\le c\}$; if $\dim \ker A> \dim \ker A^*$, then 
$\{f\circ A/g\le c\}\simeq \{g^\polar\circ A^*/f^\polar\le c\}*\mathbb{S}^{\dim \ker A-\dim \ker A^*-1}$; and if $\dim \ker A< \dim \ker A^*$, then 
$\{g^\polar\circ A^*/f^\polar\le c\}\simeq  \{f\circ A/g\le c\}*\mathbb{S}^{\dim \ker A^*-\dim \ker A-1}$. 
We then complete the proof by noting that $\mathrm{indF}(A)=\dim \ker A-\dim \ker A^*$. 
\end{proof}  

\begin{remark}
In fact, the proof of Theorem \ref{th:ratio-sublevel-equiA} leads to even more conclusions. For example, if $\mathrm{indF}(A)>0$, we have $$\{f\circ A/g\le c\}\simeq\{g^\polar\circ A^*/f^\polar\le c\}*\mathbb{S}^{\mathrm{indF}(A)-1}$$
while, if $\mathrm{indF}(A)<0$, then we have 
 $$\{g^\polar\circ A^*/f^\polar\le c\}\simeq  \{f\circ A/g\le c\}*\mathbb{S}^{-\mathrm{indF}(A)-1}$$
\end{remark}

\subsubsection{
Lusternik-Schnirelman  critical equivalence 
}

\begin{theorem}
\label{th:fA/g-index-critical}Let $X$ and $Y$ be Hilbert spaces. Let $A:X\to\ Y$ be a Fredholm operator, let 
$f\in\cvp_c^p(Y)$ and $g\in\cvp_c^p(X)$. 
Then, the nonzero Lusternik-Schnirelman min-max critical values of 
$f\circ A/g$ 
coincide with that of 
$g^\polar\circ A^* /f^\polar$. 

\end{theorem} 

\begin{proof}

We first prove the case that $f^\polar $ and $g$ are further assumed to be $C^1$-smooth (or, $f$ and $g^\polar $ are $C^1$-smooth). 

Based on Lemma \ref{lem:g_R(A)g_K(A)}, the nonzero Lagrange critical values of $(f\circ A,g)$ coincide with that of $(f\circ A,g_{\KA})$, and by the proof of Theorem \ref{th:ratio-sublevel-equiA}, the nonzero Lusternik-Schnirelman min-max critical values of $f\circ A/g$ coincide with that of $f\circ A/g_{\KA}$ restricted on $(\ker A)^\bot$. 
Analogously, the nonzero Lusternik-Schnirelman min-max critical values of $g^\polar\circ A^* /f^\polar$ coincide with that of $g^\polar \circ A^*/(f^\polar)_{\KAs}$ restricted on $(\ker A^*)^\bot$.  
Thus, it suffices to prove that the Lusternik-Schnirelman min-max critical values of $f\circ A/g_{\KA}$ restricted on $(\ker A)^\bot$  coincide with that of $g^\polar \circ A^*/(f^\polar)_{\KAs}$ restricted on $(\ker A^*)^\bot$. Recall that the 
nonzero min-max critical values\footnote{This has been already defined in \eqref{eq:index-LS-minmax-c}. For readers' convenience, we write the formula here again.}  
are defined as 
$$c_k({\small \frac{f\circ A}{g_{\KA}}}):=\inf_{\substack{
S\subset \{f\circ A/g_{\KA}\le c\}\cap \mathrm{ker}(A)^\bot\\ \mathrm{ind}(S)\ge k     
}}\sup_{x\in S}\frac{f(Ax)}{g(x)}
$$
and 
$$
c_k({\small \frac{g^\polar\circ A^*}{ (f^\polar)_{\KAs}}}):=\inf_{
\substack{ S\subset \{g^\polar\circ A^* /f^\polar\le c\}\cap \mathrm{ker}(A^*)^\bot  \\ \mathrm{ind}(S)\ge k
}}\sup_{x\in S}\frac{g^\polar(A^*x)}{f^\polar(x)}$$ defined via such index. 
We shall prove that $c_k(f\circ A/g_{\KA})=c_k(g^\polar\circ A^* /(f^\polar)_{\KAs})$ for any positive integer $k$. 

By definition, it can be verified that $$c_k(f\circ A/g_{\KA})=\inf_{\mathrm{ind}(\{f\circ A/g_{\KA}\le c\}\cap \mathrm{ker}(A)^\bot)\ge k }c.$$
To show $c_k(f\circ A/g_{\KA})=c_k(g^\polar\circ A^* /(f^\polar)_{\KAs} )$, it suffices to prove 
\begin{equation}\label{eq:index=dualA}
\mathrm{ind}\big(\{f\circ A/g_{\KA}\le c\}\cap \mathrm{ker}(A)^\bot\big)=\mathrm{ind}\big(\{g^\polar\circ A^* /(f^\polar)_{\KAs} \le c\}\cap \mathrm{ker}(A^*)^\bot\big) .   
\end{equation}

If the admissible  index is a homotopy invariant (see (I3) in Definition \ref{def:admissible-index}), then 
by Theorem \ref{th:ratio-sublevel-equiA}, we have $\{f\circ A/g_{\KA}\le c\}\cap \mathrm{ker}(A)^\bot\simeq\{g^\polar \circ A^*/(f^\polar)_{\KAs} \le c\}\cap \mathrm{ker}(A^*)^\bot$ for any $c>0$, and then we immediately obtain 
the equality \eqref{eq:index=dualA}.

If the admissible  index satisfies the nondecreasing property under continuous map (i.e., (I4) in Definition \ref{def:admissible-index}), we can also obtain the equality \eqref{eq:index=dualA}.  
First, suppose that $g$ and $f^\polar $ are $C^1$-smooth at any nonzero, then by Lemma \ref{lem:fA/g-zyx-tech1}, the continuity of $\varphi$ and the nondecreasing property of $\mathrm{ind}$, 
we have 
\begin{align*}
 \mathrm{ind}\big(\{f\circ A/g_{\KA}\le c\}\cap \mathrm{ker}(A)^\bot\big)
&\le \mathrm{ind}\big(\varphi(\{f\circ A/g_{\KA}\le c\}\cap \mathrm{ker}(A)^\bot)\big)
\\&\le \mathrm{ind}\big(\{g^\polar \circ A^*/(f^\polar)_{\KAs} \le c\}\cap \mathrm{ker}(A^*)^\bot\big).   
\end{align*}
The converse holds similarly. Hence, we have $c_k(f\circ A/g_{\KA})=c_k(g^\polar \circ A^*/(f^\polar)_{\KAs})$ for any $k$. 

We also note from Definition \ref{defn:minimax-cri} that 
$$c_k({\small \frac{f\circ A}{g}}):=\inf_{\tiny \begin{array}{c}
S\subset \{f\circ A/g\le c\} \\ \mathrm{ind}(S)\ge k     
\end{array}}\sup_{x\in S}\frac{f(Ax)}{g(x)},
$$
$$
c_k({\small \frac{g^\polar\circ A^*}{ f^\polar}}):=\inf_{\tiny
\begin{array}{c} S\subset \{g^\polar\circ A^* /f^\polar\le c\}  \\ \mathrm{ind}(S)\ge k
\end{array}}\sup_{x\in S}\frac{g^\polar(A^*x)}{f^\polar(x)}$$
defined via such index, satisfy $$c_1(f\circ A/g)=\cdots=c_{k_0}(f\circ A/g)=0<c_{k_0+1}(f\circ A/g)\le\cdots$$ and $$c_1(g^\polar\circ A^* /f^\polar )=\cdots=c_{k_*}(g^\polar\circ A^* /f^\polar )=0<c_{k_*+1}(g^\polar\circ A^* /f^\polar )\le\cdots$$   where $k_0:=\dim \ker A$ and $k_*:=\dim \ker A^*$. 
Moreover, by the proof of Theorem \ref{th:ratio-sublevel-equiA}, for any $j\ge 1$, we have $$c_{j+k_0}(f\circ A/g)=c_j(f\circ A/g_{\KA})\;\text{ and }\;c_{j+k_*}(g^\polar\circ A^* /f^\polar )=c_j(g^\polar \circ A^*/(f^\polar)_{\KAs})$$ 
which means that the nonzero min-max critical values of $f\circ A/g$ and $g^\polar\circ A^* /f^\polar$ coincide.
%

Finally, we remove the $C^1$-smoothness condition, and use standard approximation method, the equivalence between nonzero min-max critical values of $f\circ A/g$ and of $g^\polar\circ A^* /f^\polar$ 
can also be verified.  
\end{proof}

\begin{remark}The proof of Theorem \ref{th:fA/g-index-critical} leads to even more conclusions. For example, 
$$c_k({\small \frac{g^\polar\circ A^*}{ f^\polar}})=c_{k+\mathrm{indF}(A)}({\small \frac{g^\polar\circ A^*}{ f^\polar}})$$
for any $k\in\mathbb{Z}$, where we set $c_k(\cdot)=0$ whenever $k\le 0$. 
\end{remark}

\subsubsection{Equivalence of Morse criticality and isomorphism of Rothe critical groups}

\begin{theorem}
\label{th:Morse-Rothe-A}Let $X$ and $Y$ be Hilbert spaces. 
Let $f\in \cvp_c^p(Y)$ and $g\in \cvp_c^p(X)$ be such that $f^\polar $ and $g$ are $C^1$-smooth (or, $f$ and $g^\polar $ are $C^1$-smooth) at nonzeros. Let $A\in\mathcal{B}(X,Y)$ with the range $R(A)$ closed. 

Then, there exists a bijection between $\mathrm{Cri}_{MP}(f\circ A,g)$ and $ \mathrm{Cri}_{MP}(g^\polar \circ A^*,f^\polar )$. 
For any  Morse critical point $\al$ of $f\circ A/g$ with the critical group $C_*(f\circ A/g,\al)$, $\tau:=(A^*|_{\mathrm{ker}(A^*)^\bot})^{-1}\gradientp g(\al)$ is a Morse critical point of $g^\polar\circ A^* /f^\polar $, and the critical groups satisfy
$$C_*\big(\frac {f\circ A}{g},\al\big) \cong C_*\big(\frac{g^\polar\circ A^* }{f^\polar },\tau\big) .$$
\end{theorem}

\begin{proof}
Let $\mathrm{Cri}_{LP}^S(f\circ A,g)$ be the set of Lagrange critical points with constraint $f(Ax)=1$, i.e., $\mathrm{Cri}_{LP}^S(f\circ A,g):=\{x\in \mathrm{Cri}_{LP}(f\circ A,g):f(Ax)=1\}$. 
By Theorem \ref{th:critical(fA,g)}, the map  $(A^*|_{\mathrm{ker}(A^*)^\bot})^{-1}\gradientp g$ restricted on the set $\mathrm{Cri}_{LP}^S(f\circ A,g)$, 
induces a bijection 
$(A^*|_{\mathrm{ker}(A^*)^\bot})^{-1}\gradientp g|_{\mathrm{Cri}_{LP}^S(f\circ A/g)}:\mathrm{Cri}_{LP}^S(f\circ A,g)\to \mathrm{Cri}_{LP}^S(g^\polar \circ A^*,f^\polar )$ 
with the inverse $\big((A^*|_{\mathrm{ker}(A^*)^\bot})^{-1}\gradientp g|_{\mathrm{Cri}_{LP}^S(g^\polar \circ A^*,f^\polar )}\big)^{-1}$ equaling  $$(A|_{\mathrm{ker}(A)^\bot})^{-1}\gradientp f^\polar |_{\mathrm{Cri}_{LP}^S(g^\polar\circ A^* /f^\polar )}:\mathrm{Cri}_{LP}^S(g^\polar \circ A^*,f^\polar )\to \mathrm{Cri}_{LP}^S(f\circ A,g).$$ 
To simplify the notation, we 
write $$\acg:=(A^*|_{\mathrm{ker}(A^*)^\bot})^{-1}\gradientp g|_{\mathrm{Cri}_{LP}^S(f\circ A/g)}\text{ and }\afc:=(A|_{\mathrm{ker}(A)^\bot})^{-1}\gradientp f^\polar|_{\mathrm{Cri}_{LP}^S(g^\polar\circ A^* /f^\polar )}.$$

Consider the set of Morse regular points of $f\circ A/g$ and $g^\polar \circ A^*/f^\polar $ restricted on unit spheres:
$$MR(f\circ A/g)=\{\text{Morse regular points of }f\circ A/g\text{ restricted on }f\circ A=1\}$$
and 
$$MR(g^\polar\circ A^* /g^\polar\circ A^* )=\{\text{Morse regular points of }g^\polar\circ A^* /f^\polar \text{ restricted on }g^\polar\circ A^* =1\} .$$
We only need to show that the Morse regular points of $f\circ A/g$ and $g^\polar\circ A^* /f^\polar $ are in one-to-one correspondence via $\acg$, that is, 
\begin{equation}\label{qe:MR_LP-A1}
\acg\big(MR(f\circ A/g)\cap \mathrm{Cri}_{LP}(f\circ A,g)\big)\subset MR(g^\polar\circ A^* /f^\polar )\cap \mathrm{Cri}_{LP}(g^\polar \circ A^*,f^\polar )    
\end{equation}
and 
\begin{equation}\label{qe:MR_LP-A2}
\afc \big( MR(g^\polar\circ A^* /f^\polar )\cap \mathrm{Cri}_{LP}(g^\polar \circ A^*,f^\polar ) \big)\subset MR(f\circ A/g)\cap \mathrm{Cri}_{LP}(f\circ A,g)    
\end{equation}

If \eqref{qe:MR_LP-A1} and \eqref{qe:MR_LP-A2} hold, then $\acg|_{\mathrm{Cri}_{MP}^S(f\circ A,g)}:\mathrm{Cri}_{MP}^S(f\circ A,g)\to \mathrm{Cri}_{MP}^S(g^\polar \circ A^*,f^\polar )$ is a bijection with the inverse $\afc|_{\mathrm{Cri}_{MP}^S(g^\polar \circ A^*,f^\polar )} :\mathrm{Cri}_{MP}^S(g^\polar \circ A^*,f^\polar )\to \mathrm{Cri}_{MP}^S(f\circ A,g)$. 

Now, let's show \eqref{qe:MR_LP-A1}. 

For a Morse regular point $x_0$ of $f\circ A/g$ which is also Lagrange critical point, there is a decreasing flow $\eta(x,t)$ near $x_0$, i.e., there exists a neighborhood $U_{x_0}$ such that for any $x\in U_{x_0}$, $t\in(0,1]$,
$$\frac {f\circ A}{g}\big(\eta(x,t)\big)<\frac {f\circ A}{g}(x) $$
and $\eta(x,0)=x$. Note that $
\acg(\eta(x,t))$ is a continuous flow along $
\acg(x)$, and $g^\polar (\gradientp g(\eta(x,t)))=1$. 
In particular, when $x=x_0$, and $t>0$, 
\begin{align*}
\frac{g^\polar (\gradientp g(\eta(x_0,t)))}{f^\polar ((A^*|_{\mathrm{ker}(A^*)^\bot})^{-1}\gradientp g(\eta(x_0,t)))}&= \frac{g^\polar (A^*\acg(\eta(x_0,t)))}{f^\polar (\acg(\eta(x_0,t)))}
\le \frac{f(A\eta(x_0,t))}{g(\eta(x_0,t))}<\frac{f(Ax_0)}{g(x_0)}\\&
=\frac{g^\polar (A^*\acg(x_0))}{f^\polar (\acg(x_0))} =\frac{g^\polar (\gradientp g(x_0))}{f^\polar ((A^*|_{\mathrm{ker}(A^*)^\bot})^{-1}\gradientp g(x_0))}   
\end{align*}
where the inequality is based on Lemma \ref{lem:fA/g-zyx-tech1}.

For any $y\in V_{y_0}$ where $y_0=
\acg(x_0)$ and $t\in(0,1]$,
\begin{align*}
\frac{g^\polar (A^*\acg(\eta(\afc (y),t)))}{f^\polar (\acg(\eta(\afc (y),t)))}\le \frac{f(A\eta(\afc (y),t))}{g(\eta(\afc (y),t))}
<\frac{f(A\afc (y))}{g(\afc(y))}
\le\frac{g^\polar (A^*y)}{f^\polar (y)}.
\end{align*}
However, $t\mapsto \acg(\eta(\afc (y),t))
$ is not a proper flow, as in general  $$\acg(\eta(\afc (y),0)))=\acg(\afc (y))\ne y$$ 
We have proved that
$$\frac{g^\polar(A^*\acg(\eta(x_0,t)))}{f^\polar( \acg(\eta(x_0,t)))}<\frac{g^\polar(A^*\acg(x_0))}{f^\polar( \acg(x_0))}\text{ and }\frac{g^\polar(A^*\acg(\eta(\afc(y),t)))}{f^\polar( \acg(\eta(\afc(y),t)))}<\frac{g^\polar (A^*y)}{f^\polar (y)}.$$

We define $\delta(y)=\|y-\acg(\afc (y))\|$ which satisfies that $y$ is a Lagrange critical point  iff $\delta(y)=0$. Moreover, $\delta(\cdot)$ is continuous. 
Let  $$H(y,t)=\frac{t \acg(\afc (y)) +(1-t)y}{g^\polar \left(t \acg(\afc (y)) +(1-t)y\right)},\;\; x\in X,\, 0\le t\le1.$$
It is easy to see $H$ is continuous,  $H(y,0)=y$ and  $H(y,1)=\acg(\afc (y))$ when $g^\polar (y)=1$. 

Let  $\eta^\#:V_{y_0}\times[0,1]\to X$ be defined by
$$\eta^\#(y,t)=\begin{cases}
H(y,t/\delta(y)),&\text{ if }1\ge \delta(y)>t\ge0\\
\acg(\eta(\afc (y),t-\delta(y)))),&\text{ if }1\ge t\ge\delta(y)\ge0
\end{cases}$$
Then $\eta^\#(y,0)=y$ and 
$$\frac{g^\polar (A^*\eta^\#(y,t))}{f^\polar (\eta^\#(y,t))}< \frac{g^\polar (A^*y)}{f^\polar (y)},\;\forall t\in (0,1]. $$
This implies that $y_0$ is actually a Morse regular point of $g^\polar\circ A^* /f^\polar $. 
The proof is then completed.

We are in a position to show the isomorphisms of Rothe critical groups. 

In fact, if $\al$ is an isolated Morse critical point of $f\circ A/g$, then by the homology excision property, $C_*(f\circ A/g,\al)\cong H_*(\{f\circ A/g\le f(A\al)/g(\al)\},\{f\circ A/g\le f(A\al)/g(\al)\}\setminus\{\al\})$. 
Since $\acg(\al)=\bet$, $f(A\al)/g(\al)=g^\polar (A^*\tau)/f^\polar (\tau)$, and 
$\acg:\{f\circ A/g\le f(A\al)/g(\al)\}\to \{g^\polar\circ A^* /f^\polar \le g^\polar (A^*\tau)/f^\polar (\tau)\}$ is a homotopy equivalence, we have 
a homomorphism $$\acg_*: C_*(f\circ A/g,\al)\to C_*(g^\polar\circ A^* /f^\polar ,\tau).$$
Similarly, $\afc $ induces a homomorphism 
$$\afc_*:  C_*(g^\polar\circ A^* /f^\polar ,\tau)\to C_*(f\circ A/g,\al).$$
Note that $\acg_*\circ \afc_*=(\acg \circ \afc)_*:C_*(f\circ A/g,\al)\to C_*(f\circ A/g,\al)$ is an isomorphism because $\afc \circ \acg\simeq \mathrm{id}$ by Theorem \ref{th:ratio-sublevel-equiA}. Thus, we obtain that both $\acg_*$ and $\afc_*$ are isomorphisms between $C_*(f\circ A/g,\al)$ and $ C_*(g^\polar\circ A^* /f^\polar ,\tau)$. 
\end{proof}

\section{Applications and Extensions}
\label{sec:app}
Our theory of duality has significant practical applications.  
In this section, we apply the duality theory established in Section \ref{sec:ratio} to many areas, including manifolds, polyhedrons, hypergraph $p$-Laplacians, zonotopes, nonlinear eigenvalue problems and bifurcation problems. 
Meanwhile, the idea of the duality theory for RC functions can be extended to DC functions.


\subsection{Application 
to Cheeger constants on manifolds and polyhedrons}\label{sec:Cheeger-dual}
Let $M$ be a closed manifold of dimension $n$  and let $\|\cdot\|$ be a norm on its $i$-th chain group $C_i$. 
Roughly speaking, we denote by $(M,\|\cdot\|,C_i)$ the critical data (i.e., 
Lagrange critical values, and Lusternik-Schnirelman min-max critical values
) of the Rayleigh-type quotient $\|\partial \cdot\|/\|\cdot\|$.  
Similarly, we can define the cochain version $(M,\|\cdot\|,C^i)$.

Let $X$ be a polyhedral manifold, e.g., a triangulation of $M$. Then, as a cell complex, it has a natural dual polyhedral manifold $X^\circ$ by reversing the inclusion order of the faces. We can similarly consider the chain complexes on $X$ and $X^\circ$, respectively. 

\begin{defn}[critical equivalence]
We say that $(M,\|\cdot\|,C)$ is critical equivalent to $(M',\|\cdot\|',C')$ if the function-pairs $(\|\partial \cdot\|,\|\cdot\|)$ and $(\|\partial \cdot\|',\|\cdot\|')$ are Lagrange equivalent, as well as the Rayleigh quotients $\|\partial \cdot\|/\|\cdot\|$ and $\|\partial \cdot\|'/\|\cdot\|'$ are min-max equivalent, 
where $C$ and $C'$ are certain chain (or cochain) groups on $M$ and $M'$, and $\partial$ and $\partial'$ are boundary (or coboundary) operators on $C$ and $C'$, respectively. And we simply write $$(M,\|\cdot\|,C)\sim(M',\|\cdot\|',C')$$
if $(M,\|\cdot\|,C)$ is critical equivalent to $(M',\|\cdot\|',C')$.
\end{defn}

\begin{theorem}\label{thm:cri-complex}
Let $M$ be an $n$-dimensional compact closed Riemannian manifold. Then, 
$$ (M,\|\cdot\|_*,C^i)\sim  (M,\|\cdot\|,C_{n-i})\sim (M,\|\cdot\|_*,C^{n-i})\sim  (M,\|\cdot\|,C_{i}), $$

Furthermore, let $X$ be a polyhdralization of $M$, and $X^\circ$ be its dual polyhedral complex.  Then
$$(X,\|\cdot\|,C_i)\sim(X,\|\cdot\|_*,C^i)\sim(X^\circ,\|\cdot\|,C^{n-i})\sim (X^\circ,\|\cdot\|_*,C_{n-i}).$$
\end{theorem}

For instance, we can think of $X$ as the boundary of a convex polytope.

\begin{proof}
The proof relies on three different types of ``duality'', the Poincare duality on the polyhedral complexes $X\leftrightarrow X^\circ$, the norm duality on the norms $\|\cdot\|\leftrightarrow\|\cdot\|_* $, the dual relation of chain and cochain groups on finite sets $C_i(X)\cong C^i(X)$. 
The key framework is the following diagram:
$$
\xymatrix{C^0(X)\ar[r]^{d}\ar@{=}[d]^{\text{linear isomorphism}}& C^1(X)\ar[r]^{d}\ar@{=}[d]&\cdots\ar[r]^{d~~~~}&C^{n-1}(X)\ar[r]^{~~~~d}\ar@{=}[d]&C^n(X)\ar@{=}[d]\\
C_0(X)\ar@{=}[d]& C_1(X)\ar[l]_{\partial}\ar@{=}[d]&\cdots\ar[l]_{\partial}&C_{n-1}(X)\ar[l]_{\partial~~}\ar@{=}[d]&C_n(X)\ar[l]_{~~~~\partial}\ar@{=}[d]
\\
C_n(X^\circ)\ar[r]^{\partial}\ar@{=}[d]& C_{n-1}(X^\circ)\ar[r]^{~~\partial}\ar@{=}[d]&\cdots\ar[r]^{\partial~~}&C_{1}(X^\circ)\ar[r]^{\partial~~}\ar@{=}[d]&C_n(X^\circ)\ar@{=}[d]
\\
C^n(X^\circ)& C^{n-1}(X^\circ)\ar[l]_{d}&\cdots\ar[l]_{~~d}&C^{1}(X^\circ)\ar[l]_{d~~}&C^0(X^\circ)\ar[l]_{d}
}
$$
in which all the equalities `$=$'  mean linear isomorphisms, some of them are due to the assumption that $X$ is a finite polyhedral complex, and some isomorphisms are based on Poincare's duality. 

To complete the proof, we shall make full use of the critical duality theory (Theorems \ref{th:critical(fA,g)} and \ref{th:fA/g-index-critical}). 
In fact, the positive Lagrange critical values of $(\|\partial \cdot\|,\|\cdot\|)$ coincide with that of $(\|\partial^\top\cdot\|_*,\|\cdot\|_*)$. Also, since $X$ is a finite polyhedral  complex,   $\partial^\top$ and $d$ are indeed the same linear transformation. 
\end{proof}


We shall apply Theorem \ref{thm:cri-complex} to Cheeger constants. Here, we use the concepts of \textbf{exact} Cheeger constant
$$h_i(X,\|\cdot\|
)=\inf\limits_{\beta\in C_{i+1},\partial\beta\ne0}\frac{\|\partial\beta\|}{\inf\limits_{\partial\beta'=\partial\beta}\|\beta'\|}$$
and the \textbf{coexact} Cheeger constant
$$h^i(X,\|\cdot\|
)=\inf_{\beta\in C^{i-1},d\beta\ne0}\frac{\|d\beta\|}{\inf\limits_{d\beta'=d\beta}\|\beta'\|}.$$
The formulation of the above exact and coexact Cheeger constants is 
motivated by the Cheeger constants on simplicial complexes \cite{Steenbergen14}, and the Cheeger constants on  differential forms 
\cite{Boulanger}. For example, if $X$ is a simplicial complex, and $\|\cdot\|$ represents the Hamming norm on $i$-dimensional
cochains (or chains), then $h_i$ and $h^i$ express the Cheeger constants on simplicial complexes introduced in \cite{Steenbergen14}. 
Moreover, if $X$ is a compact Riemannian manifold,  then $h_i$ and $h^i$ represent the exact and coexact Cheeger constants on differential forms 
\cite{Boulanger}. 

All in all, as a direct application to higher dimensional Cheeger constants on differential forms or simplicial complexes, we obtain sequences of equalities:  
\begin{theorem}\label{th:Cheeger-dual}
Let $M$ be an oriented compact closed  manifold of dimension $n$. Let $X$ be a simplicial complex that is realized as a triangulation of $M$.  Then, for any norm $\|\cdot\|$ on the $i$-th chain group of $M$ (resp., $X$), 
$$h_i(M,\|\cdot\|)=h^{i+1}(M,\|\cdot\|_*)=h^{n-i}(M,\|\cdot\|)=h_{n-i-1}(M,\|\cdot\|_*),$$
$$h_i(X,\|\cdot\|)=h^{i+1}(X,\|\cdot\|_*)=h^{n-i}(X^\circ,\|\cdot\|)=h_{n-i-1}(X^\circ,\|\cdot\|_*).$$
\end{theorem}

\begin{proof}
We only show the case on the manifold $M$. 
Taking $A=\partial$, $f=\|\cdot\|$ on $C_i$ and $g=\|\cdot\|$ on $C_{i+1}$, we find that $h_i(M,\|\cdot\|)$ indicates the smallest nonzero Lagrange critical value of $(f\circ A,g_{\KA})$, and by
Lemma \ref{lem:g_R(A)g_K(A)}, it coincides with the smallest nonzero Lagrange critical value of $(f\circ A,g)$, which is nothing but the smallest nonzero min-max critical value of $(M,\|\cdot\|,C_i)$. 
Similarly, $h^{i+1}(M,\|\cdot\|_*)$ equals the smallest nonzero min-max critical value of  $(M,\|\cdot\|,C^i)$. By Theorem \ref{thm:cri-complex}, the nonzero min-max critical values of $(M,\|\cdot\|,C_i)$ and $(M,\|\cdot\|,C^i)$ coincide, which implies $h_i(M,\|\cdot\|)=h^{i+1}(M,\|\cdot\|_*)$.
\end{proof}

These concepts and the corresponding Theorem \ref{th:Cheeger-dual} unify many facts:
\begin{itemize}
\item If $\|\cdot\|$ represents the mass norm on $i$
currents, then $h^1(M,\|\cdot\|)$ equals the usual Cheeger constant. 
\item If $\|\cdot\|_2$ is the $L^2$-norm, since the nonzero spectrum of $\partial d$ and  the nonzero spectrum of $d\partial $ agree, it is clear that $h_i(M,\|\cdot\|_2)=h^{i+1}(M,\|\cdot\|_2)$ and $h_i(X,\|\cdot\|_2)=h^{i+1}(X,\|\cdot\|_2)$. 
This fact is well known in the theory of simplicial complexes, see for instance \cite{Botnan,Knudson}. 
However, for $L^p$-norm $\|\cdot\|_p$ or other norms, that is new. It is interesting  we can derive from Theorem \ref{th:Cheeger-dual} that $h_i(X,\|\cdot\|_p)=h^{i+1}(X,\|\cdot\|_q)$ where $p$ and $q$ are H\"older conjugates. 
\item When $X$ is a finite simplicial complex that is realized as a triangulation of a compact Riemannian manifold, and $\|\cdot\|$ is the $l^1$-norm on $C_{n-1}$, then as a corollary of Theorem \ref{th:Cheeger-dual}, $h_{n-1}(X,\|\cdot\|_1)=h^1(X^\circ,\|\cdot\|_\infty)$ which can be used to derive a Cheeger type inequality on the $(n-1)$-dimensional faces of the simplicial complex $X$ (see \cite{jostzhang24+}). 
\end{itemize}
 
In Theorem \ref{pro:zonotope-Cheeger-constnat}, we also give a zonotope characterization of Cheeger's constant on connected graphs. For details, see Section \ref{sec:zonotope-1&infty}. 

\subsection{
Rephrasing vector-valued distances 
in terms of critical values
}
\label{sec:geometric-inter}

\subsubsection{
Vector-valued distances on $\mathrm{GL}_n(\R)$ 
}

The distance of 
positive definite matrices was introduced by  López,  Pozzetti,  Trettel,  Strube, and  Wienhard \cite{Lopez21}: 

Let $\mathrm{SPD}_n$ denote positive definite real symmetric $n\times n$ matrices.  
For two points $P,Q\in \mathrm{SPD}_n$, the vector-valued distance (VVD) from $P$ to $Q$ is defined as
$$d_{vv}(P,Q):=\left(\log\lambda_1(P^{-1}Q),\cdots,\log\lambda_n(P^{-1}Q)\right).$$


It is important to note that the vector-valued distance enjoys some properties analogous to traditional metric distances, 
see \cite{Lopez21,Kapovich09}. 

Next, we give a slight generalization of the vector-valued distance. 
Given an admissible index, a norm $\|\cdot\|$ on $\R^n$, two $n\times n$ matrices $A$ and $B$, let $c_1(\|A\cdot\|/\|B\cdot\|)\le\cdots\le c_n(\|A\cdot\|/\|B\cdot\|)$ 
denote the Lusternik-Schnirelman min-max  critical values of the RC function 
$x\mapsto \|A x\|/\|B x\|$. 
Define $d_{vv,\|\cdot\|}:\mathrm{GL}_n(\R)\times \mathrm{GL}_n(\R)\to\R^n$ as 
$$d_{vv,\|\cdot\|}(A,B)=\big(\log c_1(\|A\cdot\|/\|B\cdot\|),\cdots,\log c_n(\|A\cdot\|/\|B\cdot\|)\big)$$
where $GL(\R^n)$ indicates the general linear group, i.e., the set of invertible matrices. 

When $\|\cdot\|$ is the $l^2$-norm on $\R^n$, then $c_i(\|Q^{\frac12}\cdot\|/\|P^{\frac12}\cdot\|)=\sqrt{\lambda_i(P^{-1}Q)}$, and thus $d_{vv,\|\cdot\|}(Q^{\frac12},P^{\frac12})=\frac12 d_{vv}(P,Q) $.

For general norm $\|\cdot\|$, if $A$ and $B$ are invertible, then by Theorem \ref{th:index-critical}, we have $$d_{vv,\|\cdot\|}(A,B)=d_{vv,\|\cdot\|_*}((B^\top)^{-1},(A^\top)^{-1}).$$
And if $AB^{-1}=B^{-1}A$, then we further have $d_{vv,\|\cdot\|}(A,B)=d_{vv,\|\cdot\|_*}(A^\top,B^\top)$. 

We can also extend the definition of vector-valued distance as
$\overline{d}_{vv,\|\cdot\|}:\mathrm{GL}_n(\R)\times \mathrm{GL}_n(\R)\to\R^{2n}$ which is defined by 
$$\overline{d}_{vv,\|\cdot\|}(A,B)=(c_1(A,B),\cdots,c_{2n}(A,B))
$$
where $c_1(A,B)\le\cdots\le c_{2n}(A,B)$ form a reordering of $c_i(\|A\cdot\|/\|B\cdot\|)$ and $c_i(\|B^{-1}\cdot\|/\|A^{-1}\cdot\|)$, $i=1,\cdots,n$.  
Then Theorem \ref{th:index-critical} simply implies
$$\overline{d}_{vv,\|\cdot\|}(A,B)=\overline{d}_{vv,\|\cdot\|_*}(A^\top,B^\top) .$$

\subsubsection{Vector-valued distances between norms}\label{sec:vector-valued-distance/norm}

Inspired by the previous subsection on vector-valued distances between nondegenerate matrices, in this subsection, we consider vector-valued distances between norms, and between convex bodies.  
First, we consider two norms $\|\cdot\|$ and $\|\cdot\|'$ on $\R^n$. Then given an admissible index, there are $n$ positive Lusternik-Schnirelman min-max critical values $c_1(\|\cdot\|/\|\cdot\|')\le c_2(\|\cdot\|/\|\cdot\|')\le \cdots\le c_n(\|\cdot\|/\|\cdot\|')$ 
of the nonlinear Rayleigh quotient $\|\cdot\|/\|\cdot\|'$. 
Let $$d_{vv}(\|\cdot\|,\|\cdot\|'):=\big(\log c_1(\|\cdot\|/\|\cdot\|'),\cdots,\log c_n(\|\cdot\|/\|\cdot\|')\big)$$

Theorem \ref{th:index-critical} implies $d_{vv}\big(\|\cdot\|,\|\cdot\|')=d_{vv}\big(\|\cdot\|'_*,\|\cdot\|_*)$.

\begin{itemize}

\item The Goldman-Iwahori distance \cite{Goldman} between 
two norms $\|\cdot\|$ and $\|\cdot\|'$ on a linear space 
is defined as
$$d_{GI}(\|\cdot\|,\|\cdot\|'):=\sup\limits_{x\ne 0}|\log \| x\|-\log \| x\|'|,$$
which can be rewritten as $d_{GI}(\|\cdot\|,\|\cdot\|')=\|d_{vv}\big(\|\cdot\|,\|\cdot\|')\|_\infty$. 
Sometimes experts would work on its multiplicative version, i.e., 
$$d_{MGI}(\|\cdot\|,\|\cdot\|'):=\sup\limits_{x\ne 0}\max\Big\{\frac{\|x\|}{\| x\|'},\frac{\|x\|'}{\| x\|}\Big\}.$$

Translating it into the language of convex bodies, we have
\[
d_{MGI}(K,L):=\inf\Big\{r\ge 1: \frac1rL\subset K\subset rL\Big\}\, .
\]
This distance is used for studying floating and illumination  bodies  \cite{Mordhorst19},  and is   equivalent to the  Goldman-Iwahori metric introduced for Bruhat-Tits buildings \cite{Haettel22}. 

\subsection{Contact problems}

In this section, we introduce a geometric notion, it describes when two convex bodies can be tangent with each other under homothetic transformation. 
It is surprisingly that such seeming purely geometric notion is proved to be equivalent to the set of Lagrange critical values, which is actually an analytic concept. 
It means that we derive a new geometric characterization for Lagrange critical values, see Lemma \ref{lem:basic-contact}. 

Moreover, the contact data has many other 
useful properties, such as: 
\begin{enumerate}[(1)]
\item 
$\mathrm{Contact}(K,\square)= \mathrm{Contact}(\square,Z)$ 

which reveals new and hidden structures between an origin-symmetric convex polytope $K$ and a zonotope $Z$, see Theorem \ref{thm:K-Z-dual} in Section \ref{sec:zonotope-representation} for details; 
\item $\min\{r\in \mathrm{Contact}(\square,Z_G)\}=h(G)$

which indicates that the Cheeger constant of a connected graph equals the minimum of the contact data of the hypercube and the graphical zonotope; accordingly, we actually give a new characterization of Cheeger constant using the language of convex bodies, see Theorem \ref{pro:zonotope-Cheeger-constnat} in Section \ref{sec:zonotope-1&infty}; 
\item 
$\mathrm{Contact}(\square,Z_\Gamma)=\mathrm{spec}_{\ne0}(\Delta_1(\Gamma))=\mathrm{spec}_{\ne0}(\Delta_\infty(\Gamma^*))$ 

which states that the contact data is equivalent to the nonzero eigenvalues of hypergraph 1-Laplacian as well as the dual hypergraph $\infty$-Laplacian, see Theorem \ref{thm:1-infty-zono} in Section \ref{sec:zonotope-1&infty}. 
\end{enumerate}

The results presented in this paper show the following relationship:

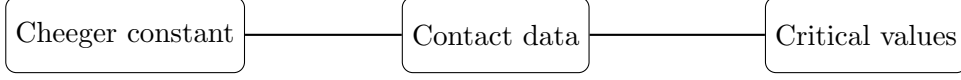
\tikzstyle{startstop} = [rectangle, rounded corners, minimum width=1cm, minimum height=1cm,text centered, draw=black, fill=red!0]
\tikzstyle{io1} = [rectangle, trapezium left angle=80, trapezium right angle=100, minimum width=1cm, minimum height=1cm, text centered, draw=black, fill=blue!0]
\tikzstyle{io2} = [trapezium,  rounded corners, trapezium left angle=100, trapezium right angle=100, minimum width=1cm, minimum height=1cm, text centered, draw=black, fill=yellow!0]
\tikzstyle{process} = [rectangle, minimum width=1cm, minimum height=1cm, text centered, draw=black, fill=orange!0]
\tikzstyle{decision} = [circle, minimum width=1cm, minimum height=1cm, text centered, draw=black, fill=green!0]
\tikzstyle{decision2} = [ellipse, rounded corners=10mm, minimum width=2cm, minimum height=2cm, text centered, draw=black, fill=green!0]
\tikzstyle{arrow} = [thick,-,>=stealth]

\begin{center}
\begin{tikzpicture}[node distance=4.9cm]
\node (Cheeg) [startstop] {  Cheeger constant };
\node (contact) [startstop, right of=Cheeg, xshift=0cm, yshift=0cm]  { Contact data };
\node (cri) [startstop, right of=contact, xshift=0cm, yshift=0cm]  { \color{black}  Critical values };
\draw [arrow](Cheeg) --node[anchor=south] { \small   } (contact);
\draw [arrow](contact) --node[anchor=south] { \small   } (cri);
 \end{tikzpicture}
 \end{center}
Furthermore, the introduction of contact problems gives new perspective for the study of Cheeger constants, critical point theory,  as well as nonlinear eigenvalue problems and convex bodies.

\subsubsection{A new geometric interpretation of Lagrange criticality}

Let $K$ and $L$ in $\R^n$ be two 
convex sets with the origin as their relative interior, and let $\partial K$ and $\partial L$ be their relative boundary, respectively. It is known that $\partial K=\{x\in\R^n:\|x\|_K=1\}$ and $\partial L=\{x\in\R^n:\|x\|_L=1\}$,  where $\|\cdot\|_K$ and $\|\cdot\|_L$ are the Minkowski functionals of $K$ and $L$, respectively.  

We say $\partial K$ and $\partial L$ are tangent at a point $x$ if $x\in\partial K\cap \partial L$ and there exists a common supporting hyperplane $H$ of both $K$ and $L$ at  $x\in H$, such that $K\not\subset H$ and $L\not\subset H$. 
\begin{defn}

Let 
\begin{equation}\label{eq:convex-body-pair-tangent}
\mathrm{Contact}(K,L)=\{\lambda>0 : \partial K\text{ and }\lambda\partial L\text{ are tangent at some point}\}    
\end{equation}
where $\lambda \partial L:=\{ \lambda y: y\in \partial L\}$.     
\end{defn}

\begin{lemma}\label{lem:basic-contact}
There are some basic properties:
\begin{enumerate}[(i)]
\item $\mathrm{Contact}(K,L)=\mathrm{Cri}_{LV}(\|\cdot\|_K,\|\cdot\|_L)=\mathrm{Cri}_{LV}(\|\cdot\|_{L^\polar},\|\cdot\|_{K^\polar})=\mathrm{Contact}(L^\circ,K^\circ)$,

where $\mathrm{Cri}_{LV}(\|\cdot\|_K,\|\cdot\|_L)$ denotes the set of nontrivial Lagrange critical values of $(\|\cdot\|_K,\|\cdot\|_L)$, see Section \ref{sec:def-Lagrange} 
\item $\mathrm{Contact}(L,K)=\{\lambda^{-1}:\lambda\in \mathrm{Contact}(K,L)\}$ 
\item $\mathrm{Contact}(T^\top L^\circ ,K^\circ )=\mathrm{Contact}(TK,L)$ for any matrix 
$T$. 
\item $\mathrm{Contact}(K,L)$ is a compact subset of $(0,+\infty)$
\end{enumerate}
\end{lemma}
\begin{proof}
(i): Suppose that $\partial K$  and $\lambda\partial L$   are tangent at $x$ with $\|x\|_K=1$ and $\|x\|_L=1/\lambda$. Then, $K$ and $\lambda L$ have a common supporting hyperplane at $x$, or, equivalently, $K$ and $\lambda L$ have have a common outer normal vector $x^*$ at $ x$. This is also equivalent to say, the normal cones of $K$ and $\lambda L$ non-trivially intersect at $x$. Note that the normal cone of $K$ at $x$ is  $\mathrm{NC}_x(K):=\{x^*:\langle x^*,y-x\rangle\le 0,\forall y\in K\}=\mathrm{cl}\,\mathrm{cone}(\partial \|x\|_K)$, see Lemma \ref{lem:normal-cone} in Appendix. 

Thus, $\partial K$  and $\lambda\partial L$   are tangent at $x$ if and only if $\mathrm{NC}_x(K)\cap \mathrm{NC}_x(\lambda L)\setminus (K^\bot\cup L^\bot)\ne\emptyset$. 
Note that $\mathrm{NC}_x(K)\setminus K^\bot=\mathrm{cl}\,\mathrm{cone}(\partial \|x\|_K)\setminus \mathrm{span}(M)^\bot=\mathrm{cone}(\partial \|x\|_K)$, where $\mathrm{cone}(S):=\{\lambda x:\lambda>0,x\in S\}$. Hence, $\mathrm{NC}_x(K)\cap \mathrm{NC}_x(\lambda L)\setminus (K^\bot\cup L^\bot)=(\mathrm{NC}_x(K)\setminus K^\bot)\cap (\mathrm{NC}_x(\lambda L)\setminus ( L)^\bot)=\mathrm{cone}(\partial \|x\|_K)\cap \mathrm{cone}(\partial \|x\|_L)$. 

Note that by Proposition \ref{pro:open-cone-intersect}, we have  $\mathrm{cone}(\partial \|x\|_K)\cap \mathrm{cone}(\partial \|x\|_L)\ne\emptyset$ if and only if there exists $\lambda>0$ such that $\partial \|x\|_K\cap \lambda\partial \|x\|_L\ne\emptyset$, and in this case, $1=\|x\|_K=\lambda\|x\|_L=\|x\|_{\lambda L}$, implying that $x\in\partial K\cap \lambda \partial L$. Therefore, $\lambda\in \mathrm{Contact}(K,L)$ if and only if $\lambda\in \mathrm{Cri}_{LV}(\|\cdot\|_K,\|\cdot\|_L)$.

For (ii): It is clear that $\lambda\in \mathrm{Cri}_{LV}(\|\cdot\|_K,\|\cdot\|_L)$ if and only if $\lambda^{-1}\in \mathrm{Cri}_{LV}(\|\cdot\|_L,\|\cdot\|_K)$, and combining this fact with (i) concludes (ii).

For (iii): Again, it follows from (i) that $\mathrm{Contact}(TK,L)=\mathrm{Cri}_{LV}(\|\cdot\|_{TK},\|\cdot\|_L)=\mathrm{Cri}_{LV}(\|T\cdot\|_{K},\|\cdot\|_L)$. According to Theorem \ref{th:critical(fA,g)}, and the basic property 
$\|T^\top\cdot\|_{L}^\polar=\|T^\top\cdot\|_{L^\polar}=\|\cdot\|_{T^\top L^\polar}$, we have 
\begin{align*}
\mathrm{Cri}_{LV}(\|T\cdot\|_{K},\|\cdot\|_L)&=\mathrm{Cri}_{LV}(\|T^\top\cdot\|_{L}^\polar,\|\cdot\|_K^\polar)\\&=\mathrm{Cri}_{LV}(\|T^\top\cdot\|_{L^\polar},\|\cdot\|_{K^\polar})=\mathrm{Cri}_{LV}(\|\cdot\|_{T^\top L^\polar},\|\cdot\|_{K^\polar})    
\end{align*}
Finally, by (i) we have $\mathrm{Cri}_{LV}(\|\cdot\|_{T^\top L^\polar},\|\cdot\|_{K^\polar})=\mathrm{Contact}(T^\top L^\circ ,K^\circ )$, and this then concludes the proof.

For (iv): The compactness of $\mathrm{Contact}(K,L)=\mathrm{Cri}_{LV}(\|\cdot\|_K,\|\cdot\|_L)$ is due to 
the assumption that the space is of finite dimension.
\end{proof}

\begin{remark}
Lemma \ref{lem:basic-contact} (i) can be viewed as a geometric interpretation 
of the nontrivial Lagrange critical values of a function pair (see Definition \ref{defn:critical(f,g)}).     
\end{remark}

\subsubsection{Contact formulation for 
various‌ distances}
Particularly, the smallest and the largest values of the contact data have other  geometric meanings as well. 
Let $\lambda_{\max}(K,L)=\max\{\lambda: \lambda \in \mathrm{Contact}(K,L)\}$ and $\lambda_{\min}(K,L)=\min\{\lambda: \lambda \in \mathrm{Contact}(K,L)\}$.  It is easy the check the following facts.
\begin{prop}\label{pro:MGI}
\begin{enumerate}[(i)]
\item $\lambda_{\max}(K,L)=\lambda_{\min}(L,K)^{-1}$
\item $d_{MGI}(K,L)=\max\{\lambda_{\max}(K,L), \lambda_{\max}(L,K)\}$
\end{enumerate}
\end{prop}

From Proposition \ref{pro:MGI} (ii), we  easily obtain  the duality identity $
d_{MGI}(K^\circ ,L^\circ )=
d_{MGI}(K,L)$ which is based on Theorem \ref{th:critical(f,g)} 
and the discussion above.  Moreover, by utilizing discussions in Section \ref{sec:vector-valued-distance/norm}, we have $d_{MGI}(K,L)=\exp\big(\|d_{vv}\big(\|\cdot\|_K,\|\cdot\|_L)\|_\infty\big)$.

\item For a pair of convex bodies $K$ and $L$ containing origin as their interior in $\R^n$, we can also 
consider their 
Banach-Mazur distance and use our critical duality to obtain new identities. 


\begin{prop}\label{prop:BM}
\begin{enumerate}[(i)]
\item $ 
d_{BM}(K,L)=
\inf\limits_{{\tiny \begin{array}{c}
T\in GL(\R^n)
\end{array}}}  \frac{\lambda_{\max}(TK,L)}{\lambda_{\min}(TK,L)}.
$ 
    \item 
$d_{BM}(K,L)=d_{BM}(K^\circ ,L^\circ )$
\end{enumerate}
\end{prop}

\begin{proof}
(i) is an easy and direct consequence of the definition of Banach-Mazur distance. 

For (ii), 
Lemma \ref{lem:basic-contact} (iii) (or Theorem \ref{th:critical(fA,g)}) implies that $\mathrm{Contact}(T^\top L^\circ ,K^\circ )=\mathrm{Contact}(TK,L)$ for any $n\times n$ invertible matrix $T$. As a direct consequence, we have $\lambda_{\max}(T^\top L^\circ ,K^\circ )=\lambda_{\max}(TK,L)$ and $\lambda_{\min}(T^\top L^\circ ,K^\circ )=\lambda_{\min}(TK,L)$, and thus,
$$d_{BM}(L^\circ ,K^\circ )=
\inf_{T^\top\in GL(\R^n)}  \frac{\lambda_{\max}(T^\top L^\circ ,K^\circ )}{\lambda_{\min}(T^\top L^\circ ,K^\circ )}=\inf_{T\in GL(\R^n)}  \frac{\lambda_{\max}(TK,L)}{\lambda_{\min}(TK,L)}=d_{BM}(K,L).
$$
\end{proof}

\begin{remark}
Proposition \ref{prop:BM} (i)  gives a new representation of the Banach-Mazur distance;     
Proposition \ref{prop:BM} (ii) generalizes the known equality for symmetric convex bodies to the nonsymmetric setting.

\end{remark}

\end{itemize}



\subsection{
Zonotope contact representation theorem}
\label{sec:zonotope-representation}

In this section, we consider the contact problem between convex polytopes in $\R^n$, that is, determine $\mathrm{Contact}(K,L)$ introduced in \eqref{eq:convex-body-pair-tangent} for two convex polytopes $K$ and $L$. Our main theorem below presents a zonotope contact representation theorem. Roughly speaking, it states that for any fixed dimension $n$, the contact problem between an origin-symmetric convex polytope and the standard hypercube is equivalent to the contact problem between a zonotope and the hypercube. Such zonotope can be constructed by the facet normal vectors of the convex polytope. 

\begin{theorem}\label{thm:K-Z-dual}
For any origin-symmetric polytope  $K\subset \R^n$ with $2m$ facets, there exists a zonotope $Z\subset \R^m$ such that $$\mathrm{Contact}(K,\square_n)= \mathrm{Contact}(\square_m,Z)$$
where 
$\square_n:=[-1,1]^n
$ denotes the standard $n$-dimensional hypercube. 
\end{theorem}

\begin{proof}
Let $K$ be an origin-symmetric polytope in $\R^n$ with $2m$ facets, where $m\ge n$. Suppose that the outer normal vectors at these facets are $\pm\mathbf{N}_1$, $\cdots$, $\pm\mathbf{N}_m$, written in column form, such that the hyperplanes spanned by the facets can be expressed as $\pm\mathbf{N}_j^\top\mathbf{x}=1$, $j=1,\cdots,m$. 
Let 
$$A=\left(\begin{array}{c}\mathbf{N}_1^\top \\
\vdots
\\ \mathbf{N}_m^\top
\end{array}\right)$$
be an $m\times n$ matrix, which will be regraded as a linear transformation from $\R^n$ to $\R^m$. 

\textbf{Step 1}. 
$\mathrm{Contact}(K,\square)=\mathrm{Cri}(\|A\cdot\|_\infty,\|\cdot\|_\infty)$ 

In this step, we shall prove that the contact data between $K$ and $\square$ agree with  the nontrivial Lagrange critical value of the function-pair $(\|A\cdot\|_\infty,\|\cdot\|_\infty)$. 

Let $\mathbf{N}_j=(w_{j1},\cdots,w_{jn})^\top\in\R^n$, $j=1,\cdots,m$. Then $A=(w_{ji})_{j\in [m],i\in [n]}\in\R^{m\times n}$, and 
$$\|Ax\|_\infty\le 1 \Longleftrightarrow\max\limits_{j=1,\cdots,m}\Big|\sum_{i\in [n]}w_{ji}x_i\Big|\le 1\Longleftrightarrow-1\le \sum_{i=1}^n w_{ji}x_i\le 1,\,\forall j=1,\cdots,m.$$
Thus, $K=\{\mathbf{x}\in \R^n:\pm\mathbf{N}_j^\top\mathbf{x}\le1,j=1,\cdots,m\}=\{x\in\R^n:\|Ax\|_\infty\le 1\} $, and then $\|A\cdot\|_\infty$ is the Minkowski functional of $K$. The obvious fact $\square=\{x\in\R^n:\|x\|_\infty\le 1\}$ means that $\|\cdot\|_\infty$ is the Minkowski functional of $\square$. 
By Lemma \ref{lem:basic-contact} (i),  we finish Step 1.

\textbf{Step 2}. $\mathrm{Cri}(\|A\cdot\|_\infty,\|\cdot\|_\infty)=\mathrm{Cri}(\|A^\top\cdot\|_1,\|\cdot\|_1)$

By Theorem \ref{th:critical(fA,g)}, the set of nonzero Lagrange critical values of the function-pair $(\|A\cdot\|_\infty,\|\cdot\|_\infty)$ on $\R^n$ coincides with the set of nonzero Lagrange critical values of the function-pair $(\|A^\top\cdot\|_1,\|\cdot\|_1)$ on $\R^m$. That is, we obtain $\mathrm{Cri}(\|A\cdot\|_\infty,\|\cdot\|_\infty)$ 
$=$ 
$\mathrm{Cri}(\|A^\top\cdot\|_1,\|\cdot\|_1)$. 

\textbf{Step 3}. $\mathrm{Cri}(\|A^\top\cdot\|_1,\|\cdot\|_1)=\mathrm{Contact}(L,\Diamond_m)$, where $L=L_\Diamond+\mathrm{Ker}(A^\top)$, and $L_\Diamond:=(A^\top|_{(\mathrm{Ker}(A^\top))^\bot})^{-1}(\Diamond_n)$ is a crosspolytope in $(\mathrm{Ker}(A^\top))^\bot$, and $\Diamond_m:=\{x\in\R^m:\|x\|_1\le 1\}$ is the standard crosspolytope in $\R^m$. 

Denote by $\bar L=\{x\in\R^m:\|A^\top x\|_1\le 1\}$. 
Solving the nonzero Lagrange critical values of the function-pair $(\|A^\top\cdot\|_1,\|\cdot\|_1)$ on $\R^m$ is equivalent to the contact problem between  $\bar L$ and $\Diamond_m$. It remains to determine the ``shape'' of $\bar L$. Consider the subspaces $X:=(\mathrm{Ker}(A))^\bot=\mathrm{Range}(A^\top)\subset \R^n$ and $Y:=\mathrm{Range}(A)=(\mathrm{Ker}(A^\top))^\bot\subset \R^m$. Note that $A|_{X}:X\to Y$ and $A^\top|_Y:Y\to X$ are two linear isomorphisms. 
Since $A^\top (\bar L)=\{z\in\R^n:\|z\|_1\le1\}$ agrees with the standard crosspolytope in $\R^n$, we have $\bar L=(A^\top)^{-1}(\Diamond_n)=(A^\top|_Y)^{-1}(\Diamond_n)+\mathrm{Ker}(A^\top)=L_\Diamond+Y^\bot=L_\Diamond+\mathrm{Ker}(A^\top)=L$ where $\Diamond_n:=\{x\in\R^n:\|x\|_1\le1\}$. It is not difficult to see that $L_\Diamond:=(A^\top|_Y)^{-1}(\Diamond_n)$ must be a crosspolytope in $Y$ which has the same polytope-structure with $\Diamond_n$. In summary, the  ``geometric interpretation'' (see Lemma \ref{lem:basic-contact} (i)) of the nontrivial Lagrange critical value gives  $\mathrm{Cri}(\|A^\top\cdot\|_1,\|\cdot\|_1)$ 
$=$ 
$\mathrm{Contact}(L,\Diamond_m)$. 

\textbf{Step 4}. $\mathrm{Contact}(L,\Diamond_m)=\mathrm{Contact}(\square,Z)$ for some zonotope $Z$

The duality results, Theorem \ref{th:critical(f,g)} and Lemma \ref{lem:basic-contact} (iii), imply
$$\mathrm{Contact}(L,\Diamond_m)=\mathrm{Contact}(\Diamond^\circ_m,L^\circ)=\mathrm{Contact}(\square_m,L^\circ).$$ 

It suffices to show $Z:=L^\polar$ is a zonotope. In fact, as shown in Step 3, $L$ is the Minkowski sum of $\mathrm{Ker}(A^\top)$ and the $n$-dimensional crosspolytope $L_\Diamond$ in $(\mathrm{Ker}(A^\top))^\bot$. Thus, $L^\polar$ equals the polar dual of $L_\Diamond$ restricted on $(\mathrm{Ker}(A^\top))^\bot$, which must be a zonotope. A more analytical explanation is as follows. 


Note that the Minkowski functional of $L^\circ$ agrees with the support function of $L$, which equals 
$$y\mapsto\sup_{\|A^\top x\|_1\le 1}\langle y,x\rangle=\begin{cases}
+\infty,&\text{ if }y\not\in \mathrm{Ker}(A^\top)^\bot=R(A)\\
\sup_{\|A^\top x\|_1\le 1}\langle Az,x\rangle,&\text{ if }y\in \mathrm{Ker}(A^\top)^\bot=R(A)
\end{cases}$$
and since $A^\top$ is a surjective, we have
$$\sup_{\|A^\top x\|_1\le 1}\langle Az,x\rangle=\sup_{\|A^\top x\|_1\le 1}\langle z,A^\top x\rangle=\|z\|_\infty=\|A^{-1}y\|_\infty.$$ Consequently, 
$L^\circ=\{y\in R(A):\|A^{-1}y\|_\infty\le 1\}=A\{z\in \R^n:\|z\|_\infty\le 1\}=A\square_n$ is a zonotope.

Finally, together with all the steps above, we have 
$$\mathrm{Contact}(K,\square_n)= \mathrm{Contact}(\square_m,A\square_n) $$
where $\square_n:=[-1,1]^n$, and $A\square_n$ indicates a zonotope. 
\end{proof}

\subsection{The 5th characterization of zonotopes: hypergraph 1- and $\infty$- Laplacians}
\label{sec:zonotope-1&infty}

A zonotope is the Minkowski summation of finitely many segments. It is interesting that such a simple object has 
many connections to hyperplane arrangements, vector configurations,  oriented matroids,  tilings and tropical geometry \cite{Ziegler95,BraunVindas-Melendez}. It is known that a convex polytope is a zonotope if and only if it is the orthogonal projection of a hypercube, if and only if each face is centrally symmetric, see e.g., Ziegler's book \cite{Ziegler95}. Thus, these three characterizations (Minkowski sum of segements, projection of hypercube, and centrally symmetric faces) are very straightforward ways to define zonotopes. The \textbf{fourth way} to characterize a zonotope is given by Witsenhausen \cite{Witsenhausen}, who showed that a convex polytope is a zonotope if and only if its support function $h$ satisfies the Hlawka inequality 
$$h(x)+h(y)+h(z)+h(x+y+z)\ge h(x+y)+h(y+z)+h(z+x).$$

In this section, we present new characterizations: a convex body is a zonotope if and only if the subdifferential of its support function equals the 1-Laplacian on some generalized hypergraph, if and only if the subdifferential of its Minkowski functional equals the $\infty$-Laplacian of a certain weighted hypergraph. 
This equivalent description reveals a connection between zonotopes and hypergraph 1-Laplacian and $\infty$-Laplacian.

A remarkable fact is that the first four known equivalence characterizations of zonotopes provide necessary and sufficient conditions for a \textbf{convex polytope} to be a zonotope, whereas we have established necessary and sufficient conditions for a {\bfseries\itshape  general convex body} to be a zonotope. Consequently, our result contains more information than the first four known equivalence characterisations of zonotopes.

For convenience, we use $[\vec a_1,\vec b_1]+\cdots+[\vec a_k,\vec b_k]$ to express a zonotope,  where $\vec a_i,\vec b_i\in\R^n\setminus \{\vec0\}$, $i=1,\cdots,k$.  
The \emph{graphical zonotope} $Z(G):=\sum_{\{i,j\}\in E}[\vec e_i,\vec e_j]$ of a graph $G=(V,E)$ encodes many properties of the graph $G$, where $\vec e_1,\cdots,\vec e_n$ are the coordinate vectors in $\R^n$. For example, the volume of $Z(G)$ equals the number of spanning trees of $G$; the number of lattice points in $Z(G)$ equals the number of forests in $G$; the number of vertices of $Z(G)$ is equal to the number of acyclic
orientations of $G$.

The highlight of this section is the initiate study of 
the relations among  zonotope, normed quantities and $p$-Laplacian eigenproblems.

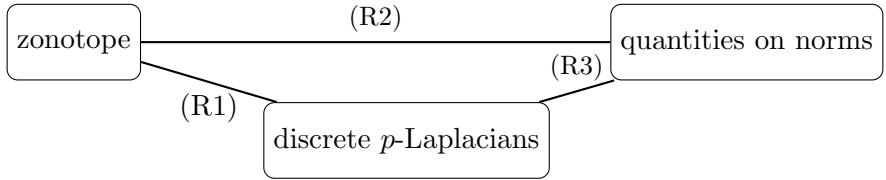
\tikzstyle{startstop} = [rectangle, rounded corners, minimum width=1cm, minimum height=1cm,text centered, draw=black, fill=red!0]
\tikzstyle{io1} = [rectangle, trapezium left angle=80, trapezium right angle=100, minimum width=1cm, minimum height=1cm, text centered, draw=black, fill=blue!0]
\tikzstyle{io2} = [trapezium,  rounded corners, trapezium left angle=100, trapezium right angle=100, minimum width=1cm, minimum height=1cm, text centered, draw=black, fill=yellow!0]
\tikzstyle{process} = [rectangle, minimum width=1cm, minimum height=1cm, text centered, draw=black, fill=orange!0]
\tikzstyle{decision} = [circle, minimum width=1cm, minimum height=1cm, text centered, draw=black, fill=green!0]
\tikzstyle{decision2} = [ellipse, rounded corners=10mm, minimum width=2cm, minimum height=2cm, text centered, draw=black, fill=green!0]
\tikzstyle{arrow} = [thick,-,>=stealth]
\begin{center}
\begin{tikzpicture}[node distance=3.9cm]
\node (submodular) [startstop] {  zonotope };
\node (optimal) [startstop, right of=submodular, xshift=5cm, yshift=0cm]  {  quantities on norms };
\node (minmax) [startstop, below of=optimal, xshift=-4.5cm, yshift=2.6cm]  { \color{black}  discrete $p$-Laplacians };
\draw [arrow](submodular) --node[anchor=south] { \small (R2) } (optimal);
\draw [arrow](optimal) --node[anchor=south] { \small (R3) } (minmax);
\draw [arrow](submodular) --node[anchor=north] { 
(R1)
} (minmax);
 \end{tikzpicture}
 \end{center}

Precisely, we find these three objects are pairwise related in natural ways: 
\begin{enumerate}[({R}1)]
\item The relation 
between zonotopes and discrete $p$-Laplacians include fruitful contents: 
\begin{enumerate}[({R1.}1)]
\item Up to an  isometric isomorphism,  zonotopes  are in one-to-one correspondence with  generalized hypergraphs via the 1-Laplacian and $\infty$-Laplacian, see Theorems 
\ref{result:hyper-1-Lap} and \ref{thm:PL-Hlawka-generalized-zonotope}. 
\item The contact data between a hypergraphic zonotope and a hypercube is equivalent to the nonzero eigenvalues of hypergraph 1-Laplacian as well as that of the dual hypergraph $\infty$-Laplacian, see Theorem \ref{thm:1-infty-zono}. 
\item We obtain two equivalent characterizations of zonotopes using hypergraph 1-Laplacian and $\infty$-Laplacian, which reveal two new necessary and sufficient conditions for a convex body to be a zonotope, see Theorem \ref{thm:1-infty-hyper}. 
\end{enumerate}

\item 
Zonotopes are closely related to some quantities on norms, such as $l^p$-minimal energy 
and Cheeger constants. In particular, we obtain a zonotope characterization for the Cheeger constant on connected graphs, see Theorem \ref{pro:zonotope-Cheeger-constnat} and Proposition \ref{prop:Cndp-minimal-energy}.   
\item  
The eigenvalues of graph $p$-Laplacian can be bounded by some quantities on norms with respect to the graph, such as the $l^p$-minimal energy, see Theorem \ref{thm:p-Lap-zonotope-C-n-d}. 
\end{enumerate}

We shall focus primarily on (R1). 
First, if we particularly concentrate on a simple graph, then the value of the graph 1-Laplacian is nothing but a face of the graphical zonotope.  

\begin{definition}[1-Laplacian for graphs]\label{def:weighted-graph1-Lap}
Given a finite graph $G=(V,E)$ with $V=\{1,\cdots,n\}$, 
the 1-Laplacian $\Delta_1$ is a set-valued map  
 defined by
$$(\Delta_1x)_i=\left\{\left.\sum_{j\in V:\,j\sim i}z_{ij}\right|z_{ij}\in \mathrm{Sgn}(x_i-x_j),z_{ij}=-z_{ji}\right\},\;\; i\in V$$
where $x\in\R^V$ is a given vector, and 
$$\mathrm{Sgn}(t):=\begin{cases}
 \{1\} & \text{if } t>0,\\
 [-1,1] & \text{if }t=0,\\
 \{-1\} & \text{if }t<0.
 \end{cases}
$$ 
\end{definition}
For more details on the basics of 1-Laplacian, see \cite[Section 7.2]{JMZ-book}.

\begin{prop}\label{pro:grahical-zonotope-1-Lap}
Given a graph $G=(V,E)$, its graphical zonotope\footnote{The graphical zonotope $Z_G$ used here is just a slight modification of $Z(G)$ in the fourth paragraph of the introduction in this subsection, via the simple relation $Z_G=2Z(G)-\sum_{i\in V}\deg_i\vec e_i$. And, we set $Z_G=Z_\varnothing=\vec0$ if $G$ is an empty graph. } defined as $Z_{G}:=\sum_{\{i,j\}\in E}[\vec e_i-\vec e_j,\vec e_j-\vec e_i]$ has the following relation with graph 1-Laplacian: 
$$\Delta_1\vec x=Z_{G[\vec x]}+C_{\vec x},\;\;\forall \vec x\in\R^V,$$
where $G[\vec x]$ is the subgraph of $G$ defined by $G[\vec x]=(V,E_x)$ and  $E_x=\{\{i,j\}\in E:x_i=x_j\}$, and $C_x=(\#N^+_x(1)-\#N^-_x(1),\cdots,\#N^+_x(n)-\#N^-_x(n))$ 
with  $N^\pm_x (i)=\{j\in V:\{i,j\}\in E,\pm(x_i-x_j)>0\}$, $\forall i\in V=\{1,\cdots,n\}$. 
\end{prop}

\begin{proof}
We use $\vec e_1,\cdots,\vec e_n$ to denote the standard coordinate vectors in $\R^n$.  
Note that for any $\vec x\in\R^V$,
\begin{align}
\Delta_1\vec x&=\left\{\left.\sum_{i=1}^n\sum_{j\in V:\{j,i\}\in E}z_{ij}\vec e_i\right|z_{ij}\in \mathrm{Sgn}(x_i-x_j),z_{ij}=-z_{ji}\right\}\notag
\\&=\left\{\left.\sum_{\{i,j\}\in E}z_{ij}(\vec e_i-\vec e_j)\right|z_{ij}\in \mathrm{Sgn}(x_i-x_j)\right\}\notag
\\&=\left\{\left.\sum_{\{i,j\}\in E:x_i>x_j}(\vec e_i-\vec e_j)+\sum_{\{i,j\}\in E:x_i=x_j}z_{ij}(\vec e_i-\vec e_j)\right|z_{ij}\in [-1,1]\right\}\notag
\\&=\sum_{\{i,j\}\in E:x_i>x_j}(\vec e_i-\vec e_j)+\sum_{\{i,j\}\in E:x_i=x_j}[\vec e_i-\vec e_j,\vec e_j-\vec e_i],\label{eq:1-Lap-ei-ej}
\end{align}
where the addition `$+$' in  \eqref{eq:1-Lap-ei-ej} is in the sense of Minkowski summation. Furthermore, 
$$\sum_{\{i,j\}\in E:x_i=x_j}[\vec e_i-\vec e_j,\vec e_j-\vec e_i]=\sum_{\{i,j\}\in E_x}[\vec e_i-\vec e_j,\vec e_j-\vec e_i]$$
and $$\sum_{\{i,j\}\in E:x_i>x_j}(\vec e_i-\vec e_j)=\sum_{i=1}^n\sum_{j\in V:\{j,i\}\in E}\mathrm{sign}(x_i-x_j)\vec e_i=\sum_{i=1}^n(\#N^+_x(i)-\#N^-_x(i))\vec e_i$$
where 
$$\mathrm{sign}(t):=\begin{cases}
 1 & \text{if } t>0,\\
 0 & \text{if }t=0,\\
 -1 & \text{if }t<0,
 \end{cases}$$
indicates the standard sign function. 
Therefore, we have $\Delta_1\vec x=Z_{G[x]}+C_x$ as desired in Proposition \ref{pro:grahical-zonotope-1-Lap}, which concludes the proof.
\end{proof}

It is known that for a vector configuration $\{\vec v_e\}_{e\in E}$, there is a natural bijection from the set of covectors
of the realizable oriented matroid \cite{Bjorner,GEBERT-ZIEGLER-zonotope94} 
$$\mathcal{L}(E):=\{(\mathrm{sign}(\vec x\cdot\vec v_e))_{e\in E}\in\{-1,0,1\}^E:\vec x\in \R^n\}$$
to the set of all non-empty faces of the zonotope $Z=\sum_{e\in E}[-\vec v_e, \vec v_e]$. 
In detail, we associate a
zonotope $Z_X$ (which is a face of $Z$) with every sign vector $X\in\mathcal{L}(E)\subset \{-1,0,1\}^E$  by $$Z_X=\sum_{e\in E:\,X_e=0}[-\vec v_e, \vec v_e]+\sum_{e\in E:\,X_e=1}\vec v_e-\sum_{e\in E:\,X_e=-1}\vec v_e.$$
Then the assignment $X\mapsto Z_X$ defines an order reversing bijection between  $\mathcal{L}(E)$ and all non-empty faces of $Z$. 

When we consider the label set $E$ as the edge set of a graph $G=(V,E)$ with $V=\{1,\cdots,n\}$, and take the vector configuration  $\{\vec v_e\}_{e\in E}$ as the column vectors of the incidence matrix of $G$, the output $Z$ is actually 
the graphical zonotope $Z_G$
, and $Z_X$ is  closely related to the graph 1-Laplacian.  
Based on these observations, and together with Proposition \ref{pro:grahical-zonotope-1-Lap}, we derive the following result which reveals a relationship among  realizable oriented matroids, graph 1-Laplacian,  and graphical zonotopes. 
 

\begin{prop}\label{pro:matroid-zonotope-1-Lap}
Taking $\{\vec v_e\}_{e\in E}$ as the column vectors of the incidence matrix of a  graph $G=(V,E)$, then the composition of the natural maps $\vec x\mapsto X:= (\mathrm{sign}(\vec x\cdot\vec v_e))_{e\in E}$ and $X\mapsto Z_X$ defines the  1-Laplacian on $G$. Precisely,  we can decompose the 1-Laplacian $\Delta_1$ into 
$$
\xymatrix@C=3em@R=0.9ex{ \Delta_1: & \R^n\ar[r] & \mathcal{L}(E)\ar[r] & \{\textrm{faces of }Z_G\} \\
 & \vec x \ar@{|->}[r] & X = (\mathrm{sign}(\vec x\cdot\vec v_e))_{e\in E}\ar@{|->}[r] & Z_X} 
$$
Moreover, $\dim \Delta_1\vec x=n-\#\{\text{connected components of }G[\vec x]\}$, where $G[\vec x]$ is defined in Proposition \ref{pro:grahical-zonotope-1-Lap}, and $\dim \Delta_1\vec x$ indicates the dimension of the zonotope $ \Delta_1\vec x$. 
 \end{prop}

\begin{proof}
Given a vector $\vec x\in\R^n$, it corresponds to a unique signature vector $X:=(\mathrm{sign}(\vec x\cdot\vec v_e))_{e\in E}$ in the realizable oriented matroid $\mathcal{L}(E)$, and to express the correspondence explicitly, we write $X$ as $X[\vec x]$.  
Since $\{\vec v_e\}_{e\in E}$ serves as an incidence matrix of the graph $G$, for an $e=\{i,j\}\in E$, we may assume without loss of generality that the $i$-th component of $\vec v_e$ is 1 and the $j$-th component of $\vec v_e$ is $-1$, and the other components are 0. 
Hence, $\vec x\cdot\vec v_e=x_i-x_j$ when $e=\{i,j\}$, and therefore $(X[\vec x])_e= \mathrm{sign}(\vec x\cdot\vec v_e)=\mathrm{sign}(x_i-x_j)$. Thus, keeping \eqref{eq:1-Lap-ei-ej} in mind, we obtain
\begin{align*}
Z_{X[\vec x]}&=\sum _{e:(X[\vec x])_e=0}[-\vec v_e, \vec v_e]+\sum_{e:(X[\vec x])_e=1}\vec v_e-\sum_{e:(X[\vec x])_e=-1}\vec v_e\\&=\sum_{\{i,j\}\in E:x_i=x_j}[\vec e_i-\vec e_j,\vec e_j-\vec e_i]+   \sum_{\{i,j\}\in E:x_i>x_j}(\vec e_i-\vec e_j) =\Delta_1\vec x .
\end{align*}


Suppose that $G[\vec x]=G_1\sqcup G_2\sqcup\cdots\sqcup G_k $, where $G_i$ is a connected component of $G[\vec x]$. Next, we shall prove that the dimension of the zonotope $Z_{X[\vec x]}$ equals $n-k$. First, note that $$\dim Z_{X[\vec x]}=\dim \sum_{l=1}^k\sum_{\{i,j\}\in E(G_l)}[\vec e_i-\vec e_j,\vec e_j-\vec e_i]=\sum_{l=1}^k\dim\Big(\sum_{\{i,j\}\in E(G_l)}[\vec e_i-\vec e_j,\vec e_j-\vec e_i]\Big).$$
Since $$\dim \Big(\sum_{\{i,j\}\in E(G_l)}[\vec e_i-\vec e_j,\vec e_j-\vec e_i]\Big)=|V(G_l)|-1$$
we have $\dim Z_{X[\vec x]}=\sum_{l=1}^k (|V(G_l)|-1)=n-k$, where $|V(G_l)|$ is the number of vertices in $G_l$.
\end{proof}


By Propositions \ref{pro:grahical-zonotope-1-Lap} and \ref{pro:matroid-zonotope-1-Lap}, we can give more geometric explanations of the eigenvalue problems for 1-Laplacian by using contact data of graphical zonotopes, see Theorem \ref{thm:1-infty-zono} for details. 
Before presenting the above results, we first provide a zonotope characterization of Cheeger’s constant. 
We note that the graph Cheeger constant is an important quantity that has been studied by many experts \cite{Bandeira}. 
It is known that the second smallest eigenvalue of graph 1-Laplacian equals the Cheeger constant \cite{Hein10}, and then by our result we have another geometric description for such eigenvalue. 

\begin{theorem}[Zonotope representation of Cheeger's constant]\label{pro:zonotope-Cheeger-constnat}
For a connected graph $G$, the Cheeger constant $h(G)$ equals the scaling constant 
$$r_*=\inf \{r>0:\text{the scaling zonotope }r\cdot Z_G\text{ intersects with }\partial \square\}$$ where $\square$ is the hypercube $[-1,1]^n$, and $Z_G$ is the graphical zonotope of $G$ (see Proposition \ref{pro:grahical-zonotope-1-Lap}).
\end{theorem}

\begin{proof}
Let $A$ be the incidence matrix of $G$. It is known that for a connected graph $G$, the Cheeger constant $h(G)$ equals the smallest positive eigenvalue of the 1-Laplacian eigenproblem $\partial \|Ax\|_1\cap \lambda \partial\|x\|_1\ne\varnothing$.  Lemma \ref{lem:basic-contact} (i) implies that the 1-Laplacian eigenproblem is equivalent to the 
contact problem of $K:=\{x\in\R^n:\|Ax\|_1\le 1\}$ and $L:=\{x\in\R^n:\|x\|_1\le 1\}$. By  Lemma \ref{lem:basic-contact} (iii), it is equivalent to $\mathrm{Contact}(L^\circ,K^\circ)$. 
By taking $f=\|A\cdot\|_1$ in Proposition \ref{pro:CV_1&Cv_p} (iii) and using Proposition \ref{pro:grahical-zonotope-1-Lap}, we have $Z_G=\Delta_1\vec 0=\partial|_{x=0} \|A x\|_1 =\{x\in\R^n:\|A x\|_1^\polar\le 1\}=K^\polar$. 
Combining this with  the fact that  $\square=L^\polar$, we have $$\mathrm{Contact}(L^\circ,K^\circ)=\mathrm{Contact}(\square,Z_G).$$ 
Consequently, $h(G)$ equals the minimum number in $\mathrm{Contact}(\square,Z_G)$, that is, 
$h(G)=r_*=\inf \{r>0:\text{the scaling zonotope }r\cdot Z_G\text{ intersects with }\partial \square\}$.
\end{proof}

\subsubsection{Zonotope representation of hypergraph 1-Laplacian and $\infty$-Laplacian}

The above results can be extended to more general hypergraph settings, which also includes weighted graphs as special case.

\begin{defn}[weighted hypergraph]\label{def:general-hypergraph}
A (vertex-edge) \emph{weighted hypergraph} $\Gamma$ is a matrix $w:=(w_{ih})_{i\in V,h\in H}$ in which the set $V:=\{1,\cdots,n\}$ refers to the vertex set, and the set $H$ indicates the hyperedge set, where $w_{ih}\in\R$. We denote it by a triple $\Gamma=(V,H,w)$.
\end{defn}
A weighted hypergraph is nothing more than a ``generalized hypergraph'' determined by ``incidence matrix''. 
The perspective of treating matrices as a generalized form of ‘‘hypergraphs’’ has been employed in the study of the $p$-Laplacian on oriented hypergraphs in Chapter 7 of \cite{JMZ-book}, as well as in the study of more general hypergraph $p$-Laplacians in \cite{Fazeny,Burger}. 
We also refer to \cite{Lanza,Huhtanen,Mazon,Weihs-an,Weihs-hi,Berkolaiko,Mugnolo,Backhausz,Vigano} for the recent fruitful progresses on analysis for hypergraphs, $p$-Laplacians and Cheeger-type constants in discrete settings.

\begin{definition}[$p$-Laplacian for vertex-edge weighted hypergraph]\label{def:p-Lap-hypergraph}
The $p$-Laplacian $\Delta_p:C(V)\to C(V)$ of $\Gamma=(V,H,w)$ is defined by $\Delta_p f=\gradientp \sum_{h\in H}\left|\sum_{j\in V}w_{jh}f(j)\right|^{p}$ in the sense of subdifferential, namely,
$$\Delta_pf(i)=\sum_{h\in H}w_{ih}\left|\sum_{j\in V}w_{jh}f(j)\right|^{p-2}\left(\sum_{j\in V}w_{jh}f(j)\right),\;\;i\in V,$$
$$\Delta_1f(i)=\left\{\sum_{h\in H}w_{ih}z_{h}:z_{h}\in \mathrm{Sgn}\left(\sum_{j\in V}w_{jh}f(j)\right)\right\},\;\;i\in V.$$
\end{definition}
For convenience, we set $w_{ih}=0$ for $i\notin h$. 

\begin{theorem}
\label{result:hyper-1-Lap}
The hypergraph 1-Laplacian is a zonotope-valued map in the following sense:
\begin{enumerate}[(i)]
\item Given the 1-Laplacian $\Delta_1$ on a hypergraph with $n$ vertices, for any subset  $S\subset\R^n$, $\bigcap_{x\in S}\Delta_1\vec x$ is a zonotope (or a single point or empty which can be viewed as degenerate zonotopes).  

\item
For any zonotope $Z$ in $\R^d$ with center $\vec0$, there exists a weighted  hypergraph\footnote{Here, we mean a hypergraph with vertex-edge weights, or equivalently,  hypergraph with  real  coefficients  (see Definition \ref{def:general-hypergraph} 
)} such that  $\Delta_1\vec0=Z$.

\item
For any zonotope $Z$ of dimension $d$, and for any $\vec x,\vec z\in\R^{d+1}$,  there exists a weighted hypergraph on $d+1$ vertices such that $\Delta_1\vec x$ is a zonotope centered at $\vec z$ and is  isometically isomorphic 
to $Z$. 
\end{enumerate}
\end{theorem}

\begin{proof}

For (i):  
The discrete 1-Laplacian maps each function to a zonotope in the following way: 
\begin{align*}
\Delta_1f&=\left\{\sum_{i\in V}\sum_{h\in H}w_{ih}z_{h}\vec e_i:z_{h}\in \mathrm{Sgn}\left(\sum_{j\in V}w_{jh}f(j)\right)\right\}
\\&=\left\{\sum_{h\in H}z_{h}\sum_{i\in V}w_{ih}\vec e_i:z_{h}\in \mathrm{Sgn}\left(\sum_{j\in V}w_{jh}f(j)\right)\right\}
\\&=\left\{\sum_{h\in H:f[h]> 0}\sum_{i\in V}w_{ih}\vec e_i-\sum_{h\in H:f[h]< 0}\sum_{i\in V}w_{ih}\vec e_i +\sum_{h\in H:f[h]=0}z\sum_{i\in V}w_{ih}\vec e_i:z\in [-1,1]\right\}
\\&=\sum_{h\in H}\mathrm{sign}(f[h])\sum_{i\in V}w_{ih}\vec e_i+\sum_{h\in H:f[h]=0}[-\sum_{i\in V}w_{ih}\vec e_i,\sum_{i\in V}w_{ih}\vec e_i]
\\&=\lim\limits_{p\to1^+}\Delta_p f+\sum_{h\in H:f[h]=0}[-\vec v_h,\vec v_h]
\end{align*}
is a zonotope with the center $$\vec a=\lim\limits_{p\to1^+}\Delta_p f=\sum_{h\in H}\mathrm{sign}(f[h])\sum_{i\in V}w_{ih}\vec e_i$$
where $\vec v_h:=\sum_{i\in V}w_{ih}\vec e_i$ and 
 $f[h]:=\sum_{j\in V}w_{jh}f(j)$, $\forall h\in H$. Here the addition `$+$' in the last two equalities are in the sense of Minkowski summation. 
Thus, $\Delta_1f=\sum_{h\in H:f[h]=0}[-\vec v_h,\vec v_h]+\sum_{h\in H:f[h]\ne0}\mathrm{sign}(f[h]) \vec v_h$ is a face 
of $\sum_{h\in H}[-\vec v_h,\vec v_h]$. 
Since the faces of 
$\sum_{h\in H}[-\vec v_h,\vec v_h]$ 
form a finite face-lattice, the intersection of finite faces must be also a face. 
Now for any nonempty set $S\subset\R^n$, $\bigcap_{f\in S}\Delta_1f$ is actually a finitely intersection, and thus is a face of $\sum_{h\in H}[-\vec v_h,\vec v_h]$, that is, a subzonotope (or a single point or empty). 

For (ii): 
For any zonotope $\sum_{j=1}^m[-\vec a_j,\vec a_j]$ in $\R^n$, we can take a 
vertex-edge weighted hypergraph $\Gamma$ on $n$ vertices with $m$ hyperedges and vertex-edge weight $(w_{ij})$ such that $\vec a_j=\sum_{i=1}^nw_{ij}\vec e_i$, $j=1,\cdots,m$. Then $\Delta_1\vec 0=\sum_{j=1}^m[-\vec a_j,\vec a_j]$.

For (iii): Continuing the discussion from (ii), first, we assume  $Z=\sum_{i=1}^m[-\vec v_i,\vec v_i]\subset\R^d$. For any $\vec x\in\R^{d+1}\setminus\{\vec 0\}$, consider the $d$-dimensional subspace $X=\vec x^\bot$. There exists a linearly   isometric isomorphism $\psi:\R^d\to X$. Denote by $\vec w_i=\psi(\vec v_i)$, $i=1,\cdots,m$. Then, construct the hypergraph whose vertex-edge incidence matrix is $(\vec w_1,\cdots,\vec w_m)$. It can be verified that $\Delta_1\vec x=\sum_{i=1}^m\mathrm{Sgn}(\vec w_i\cdot\vec x)\vec w_i=\sum_{i=1}^m\mathrm{Sgn}(0)\vec w_i=\sum_{i=1}^m[-\vec w_i,\vec w_i]=\psi(\sum_{i=1}^m[-\vec v_i,\vec v_i])=\psi(Z)$. Therefore, for any zonotope $Z$ of dimension $d$,  we have constructed a weighted hypergraph such that $\Delta_1\vec x$ is a zonotope centered at 0 which is isomorphic to $Z$.

Next, we shall prove that we can translate the center of $\Delta_1\vec x$ to $z$ for any given $z$. 
Taking additional vectors $\vec w_{m+1},\cdots,\vec w_{M}\in\R^{d+1}\setminus X$, then $\mathrm{Sgn}(\vec w_i\cdot\vec x)=\{1\}$ or $\{-1\}$, we have $\Delta_1\vec x=\sum_{i=1}^m[-\vec w_i,\vec w_i]+\vec a$ with $\vec a=\sum_{j=m+1}^M\varepsilon_j\vec w_j$ for $\varepsilon_j\in\{-1,1\}$. Therefore, if $\vec z\in\R^{d+1}\setminus X$, taking $\vec w_{m+1}=\vec z$; if $\vec z\in X$, taking $\vec w_{m+1},\vec w_{m+2}\in\R^{d+1}\setminus X$ such that $\vec w_{m+1}+\vec w_{m+2}=\vec z$. 
\end{proof}

\begin{remark}
Based on the above theorem, and the duality relation between 1-Laplacian and $\infty$-Laplacian, we have: up to an  isometric isomorphism,  zonotopes  are in one-to-one correspondence with  generalized hypergraphs via the 1-Laplacian and $\infty$-Laplacian.
\end{remark}

The Hlakwa inequality imposes further restrictions on the convexity of norms; the following result 
establishes a connection between zonotopes, the Hlakwa norm and the 1-Laplacian of hypergraphs.
\begin{theorem}\label{thm:PL-Hlawka-generalized-zonotope}
For a piecewise linear \textbf{semi-norm}  $\|\cdot\|$ on $\R^n$, the  following conditions are equivalent:
\begin{enumerate}[(a)]
\item $\|\cdot\|$ satisfies Hlakwa inequality;
\item $\partial \|\vec x\|$ is a  zonotope for any $\vec x$;
\item the polar dual of  the unit ball $\{\vec x\in\R^n:\|\vec x\|\le 1\}$ is a  zonotope;
\item  
there exists a real matrix $(w_{ij})_{m\times n}$ such that
$\|\vec x\|=\sum_{i=1}^m\left|\sum_{j=1}^nw_{ij}x_j\right|$ for any $\vec x\in\R^n$;
\item $\Delta_1:=\partial \|\cdot\|$ defines the 1-Laplacian on certain 
weighted hypergraph.
\end{enumerate}
\end{theorem}

\begin{proof}

(a) $\Leftrightarrow$ (c): 
Witsenhausen's theorem \cite{Witsenhausen} states that the support function of a centrally symmetric polytope $P$ satisfies Hlawka's inequality if and only if it is a zonotope. Now, suppose that  $\|\cdot\|_*$ is the Minkowski functional such that the unit ball $\{\vec x:\|\vec x\|_*\le1\}$ is the given origin-symmetric polytope $P$. 
Then the support function $\vec x\mapsto \sup_{y\in P}\vec x\cdot\vec y=\sup_{\| y\|_*\le1}\vec x\cdot\vec y$ agrees with the dual of $\|\cdot\|_*$, which we denote by $\|\cdot\|$. Then Witsenhausen's theorem is equivalent to the relation (a) $\Leftrightarrow$ (c).

(d) $\Rightarrow$ (e): A direct computation gives $\partial \|x\|=\Delta_1 x$ for the vertex-hyperedge incidence hypergraph $(w_{ij})$. 

(e) $\Rightarrow$ (b): 
By Theorem \ref{result:hyper-1-Lap} and its proof, we obtain that $\Delta_1 x=\partial \|x\|$ is a zonotope for any $x$,  which concludes (b).  


(b) $\Rightarrow$ (c): Note that  $\|\vec x\|_*\le 1$ if and only if $\vec x\cdot\vec y\le \|\vec y\|$ for any $\vec y$, which is equivalent to $\vec x\in \partial\|\vec 0\|$. That is to say, the polar dual $\{x\in\R^n:\|\vec x\|_*\le 1\}=\partial\|\vec 0\|$ is a zonotope. This derives (c).

Now, we prove (c) $\Rightarrow$  (d): we need the following useful statements.

\begin{enumerate}[{Claim} 1.]
    \item 
For a linear map $T:\R^m\to\R^n$ with $\mathrm{Range}(T)\ne\vec0$, we have 
$\{\vec y\in \mathrm{Range}(T):\|\vec y\|_{\inf}\le1\}=T\{\vec x\in\R^m:\|\vec x\|\le 1\}$, where 
$\|\vec y\|_{\inf}:=\inf\limits_{x\in T^{-1}(y)}\|\vec x\|$.  

Proof: In Step 1 of the proof of  Lemma \ref{lem:g_R(A)g_K(A)}, we have indeed shown the following claim:

$A\{x\in X:g(x)\le 1\}\subset \{y\in\mathrm{Range}(A):\bar g_{\RA}(y)\le1\}\subset A\{x\in X:g(x)< 1+\varepsilon\}$, $\forall \varepsilon>0$, where $A:X\to Y$ is a bounded linear operator, $g\in\cvp_c^p(X)$ and $\bar g_{\RA}(y)=\inf_{x\in A^{-1}y} g(x)$. Now, taking $X=\R^m$, $Y=\R^n$, $g=\|\cdot\|$ and $A=T$ in Lemma \ref{lem:g_R(A)g_K(A)}, then we have $\bar g_{\RA}(y)=\|y\|_{\inf}$, and 
$T\{x\in \R^m:\|x\|\le 1\}\subset \{y\in\mathrm{Range}(T):\|y\|_{\inf}\le1\}\subset T\{x\in \R^m:\|x\|< 1+\varepsilon\}$. By the local compactness of $\R^m$ and the compactness of the set $T\{x\in \R^m:\|x\|\le 1\}$, we have $T\{x\in \R^m:\|x\|\le 1\}=\bigcap_{\varepsilon>0}T\{x\in \R^m:\|x\|< 1+\varepsilon\}$ and thus $\{\vec y\in \mathrm{Range}(T):\|\vec y\|_{\inf}\le1\}=T\{\vec x\in\R^m:\|\vec x\|\le 1\}$ concluding Claim 1.

\item Let $T^*:\R^n\to\R^m$ be the dual operator of $T$. Then the dual norm of  $\|\cdot\|_{\inf}$  on $\mathrm{Range}(T)
$ equals $\|T^*\cdot\|^\polar$. Particularly, for  $\|\cdot\|=\|\cdot\|_\infty$, $(\|\cdot\|_{\inf})^\polar$ is an $l^1$-embedded norm. 

Proof: Taking $X=\R^m$, $Y=\R^n$, $g=\|\cdot\|$ and $A=T$ in Lemma \ref{lem:g_R(A)g_K(A)}, we derive from Step 1 in the proof of  Lemma \ref{lem:g_R(A)g_K(A)} that 
$(\bar g_{\RA})^\polar(y^*)=g^\polar(A^*y^*)$, that is, $(\|\cdot\|_{\inf})^\polar=\|T^*\cdot\|^\polar$, meaning that the dual norm of $\|T^*\cdot\|$  coincides with the dual norm of $\|\cdot\|_{\inf}$ on $\mathrm{Range}(T)
$. 

\end{enumerate}

Now we back to the proof of (c) $\Rightarrow$  (d). Let $\|\cdot\|^\polar$ denote the dual norm of $\|\cdot\|$. Then the dual polytope of  $\{\vec x:\|\vec x\|\le 1\}$ is $\{\vec x:\|\vec x\|^\polar\le 1\}$. Since $\{\vec x:\|\vec x\|^\polar\le 1\}$
 is a zonotope of dimension $n$ in $\R^n$, there exist $m\ge n$ and a linear map $T:\R^m\to\R^n$ such that $T(\{\vec x\in\R^m:\|\vec x\|_\infty\le 1\})=\{\vec y\in\R^n:\|\vec y\|^\polar\le 1\}$. It is clear that $\mathrm{Range}(T)=\R^n$, $\dim \mathrm{Ker}(T)=m-n$, $\dim \mathrm{Ker}(T^*)=0$,  and  $\|\vec y\|^\polar=\|\vec y\|_{\infty,\inf}$ by Claim 1. Therefore, applying Claim 2, 
 $\|\vec y\|= \|\vec y\|_{\infty,\inf}^\polar= \|T^*\vec y\|_1$. Since the matrix $T^*$  is of order $m\times n$, we can denote it by $T^*=(w_{ij})_{m\times n}$. Accordingly, $\|\vec x\|=\|T^*\vec x\|_1=\sum_{i=1}^m\left|\sum_{j=1}^nw_{ij}x_j\right|$.
\end{proof}

Inspired by the description before Proposition \ref{pro:matroid-zonotope-1-Lap}, we have the following concept:
\begin{defn}
We call the zonotope $Z$ generated by the oriented matroid $\mathcal{L}(E)$ related to the incidence matrix $\{\vec v_e\}_{e\in E}$ the \emph{hypergraphic zonotope} of the weighted hypergraph $\Gamma=(V,H,w)$.     
\end{defn}

\begin{theorem}\label{thm:1-infty-zono}
The nonzero eigenvalues of the 1-Laplacian on the hypergraph $\Gamma=(V,H,w)$, the nonzero eigenvalues of the $\infty$-Laplacian on the dual hypergraph $\Gamma^*=(H,V,w^\top)$, and the contact values between $Z$ and $\square$ coincide, i.e.,
$$\mathrm{Contact}(\square,Z)=\mathrm{spec}_{\ne0}(\Delta_1(\Gamma))=\mathrm{spec}_{\ne0}(\Delta_\infty(\Gamma^*))$$
where $Z$ is the hypergraphic zonotope 
of $\Gamma$
, and $\square$ is the unit ball of $\R^n$ in $\ell^\infty$ norm.
\end{theorem}
\begin{proof}
First, we refer to \cite[Chapters 7 and 9]{JMZ-book} for the basic knowledge of the eigenvalue problems for 1-Laplacian and $\infty$-Laplacian on hypergraphs. An explicit definition for these eigenvalue problems is as follows.  
The 1-Laplacian eigenproblem is to find $\lambda$ and $x$ such that $\partial \|Wx\|_1\cap \lambda \partial\|x\|_1\ne\varnothing$, and the $\infty$-Laplacian eigenproblem is to find $\lambda$ and $x$ such that $\partial \|Wx\|_\infty\cap \lambda \partial\|x\|_\infty\ne\varnothing$, where $W:=(w_{hi})_{h\in H,i\in V}$ is the edge-vertex incidence matrix of the hypergraph. 

Thus, the set of the nonzero eigenvalues of the 1-Laplacian on the hypergraph $\Gamma=(V,H,w)$ can be written as $\mathrm{Cri}_{LV}(\|W\cdot\|_1,\|\cdot\|_1)$, while the set of the nonzero eigenvalues of the $\infty$-Laplacian on the dual hypergraph $\Gamma^*=(H,V,w^\top)$ can be written as $\mathrm{Cri}_{LV}(\|W^\top\cdot\|_\infty,\|\cdot\|_\infty)$.

Note that $Wx=(\vec v_e\cdot \vec x:e\in E)$ and $\partial|_{x=0} \|Wx\|_1=W^\top (\partial \|Wx\|_1)|_{x=0}=W^\top\big(\sum_{e\in E}\mathrm{Sgn}(\vec v_e\cdot \vec x)\vec1_{\{e\}}\big)|_{x=0}=W^\top \sum_{e\in E}[-\vec1_e,\vec1_e]=W^\top\square_m$. Since $W^\top \vec 1_{\{e\}}=\vec v_e$, we derive that $\partial|_{x=0} \|Wx\|_1=\sum_{e\in E} [-\vec v_e,\vec v_e]$ which is the hypergraphic zonotope of $\Gamma$.

Taking $f=\|W\cdot\|_1$ in Proposition \ref{pro:CV_1&Cv_p} (iii), we have $\partial|_{x=0} \|Wx\|_1 =\{x\in\R^n:f^\polar(x)\le 1\}$. 
Let $Z=\{x\in\R^n:f^\polar(x)\le 1\}$. Then $Z$ is the hypergraphic zonotope of $\Gamma$, i.e., a zonotope $Z$ that is generated by the oriented matroid $\mathcal{L}(E)$ related to the incidence matrix $\{\vec v_e\}_{e\in E}$ of the hypergraph $\Gamma=(V,H,w)$.  

By Lemma \ref{lem:basic-contact} (i), 
$\mathrm{Contact}(\square,Z)=\mathrm{Cri}_{LV}(\|\cdot\|_\infty,f^\polar)=\mathrm{Cri}_{LV}(f,\|\cdot\|_1)=\mathrm{Cri}_{LV}(\|W\cdot\|_1,\|\cdot\|_1)$. 
By Theorem \ref{th:critical(fA,g)}, $\mathrm{Cri}_{LV}(\|W\cdot\|_1,\|\cdot\|_1) =\mathrm{Cri}_{LV}(\|W^\top\cdot\|_\infty,\|\cdot\|_\infty)$.

\end{proof}

We can apply Theorem \ref{thm:PL-Hlawka-generalized-zonotope}  to achieve the following result.
\begin{theorem}
\label{thm:1-infty-hyper}
A convex body is a zonotope, if and only if the subdifferential of its support function equals the 1-Laplacian of some weighted hypergraph, if and only if the subdifferential of its Minkowski functional equals the $\infty$-Laplacian of a certain weighted hypergraph.    
\end{theorem}

\begin{proof}
Let $P$ be a convex polytope containing the origin in its relative interior. 
Suppose $f$ is the Minkowski functional of $P$. Then $P=\{f\le 1\}$. Then, $f^\polar$ stands for the support function of $P$. 
By checking the proof of Theorem \ref{thm:PL-Hlawka-generalized-zonotope}, we find that the equivalence among (b), (c), (d), (e) does not require that $\|\cdot\|$ is a piecewise linear semi-norm. 
In fact, when we use a function $f\in \cvx_0^1(\R^n)$ instead of the semi-norm $\|\cdot\|$, the equivalence holds. 
Specifically, $P$ is a zonotope if and only if $\partial f^\polar$ is the 1-Laplacian on a certain weighted hypergraph. 
As explained in the proof of Theorem \ref{thm:1-infty-zono}, this is equivalent to the statement that $\partial f$ represents the $\infty$-Laplacian on a certain weighted hypergraph. 
\end{proof}

\begin{remark}
Here, we present the necessary and sufficient conditions for a general convex body to be a zonotope. The first four known equivalent characterizations of zonotopes, however, are restricted to the case of convex polytopes.
\end{remark}

\subsubsection{$l^p$-minimal energy}

It is known that zonotopes are closely related to vector configurations. In \cite{AmbrusNietert,FWY21}, the following constants on vector configurations are  introduced.
\begin{definition}[minimal energy \cite{AmbrusNietert,FWY21}]Given $n,d\in\mathbb{N}_+$, 
let
$$C_{n,d}:=\min\limits_{\vec v_1,\ldots,\vec v_n\in \mathbb{S}^{d-1}}\max\limits_{\varepsilon_1,\ldots,\varepsilon_n\in\{-1,1\}}\|\varepsilon_1\vec v_1+\ldots+\varepsilon_n\vec v_n\|_2$$
where  $\mathbb{S}^{d-1}$ is the unit sphere of dimension $(d-1)$ in $\R^d$ and  $\|\cdot\|_2$ is the standard Euclidean norm in $\R^d$. 
\end{definition}

The quantity $C_{n,d}$ is increasing with respect to $n$, and is decreasing with respect to $d$. We refer to \cite{AmbrusNietert,AmbrusMerino20} for the interesting  properties:  $\sqrt{n}\le C_{n,d}\le n$, $C_{d+1,d}\le  \sqrt{d+2}$, and $C_{n,d}=\sqrt{n}$ if and only if  $n\le d$. 
It is conjectured in \cite{AmbrusNietert} that $C_{d+1,d}=  \sqrt{d+2}$. Moreover, $C_{n,d}$ is also the best lower bound
of the diameter of the $d$-dimensional zonotope generated by $n$ unit segments, and the largest Laplacian eigenvalue of a graph on $n$ vertices has the lower bound $\frac1n \cdot C_{n,n-m_0}^2$, where $m_0$ is the multiplicity of the eigenvalue $0$. 

\begin{definition}[$l^p$-polarization (or Chebyshev) constant \cite{AmbrusNietert}] The constant 
$$\min\limits_{u_1,\cdots,u_n\in \mathbb{S}^{d-1}}\max\limits_{\|\vec x\|_2\le 1}\sum_{i=1}^n|\vec u_i\cdot\vec x|^p$$
is called the $l^p$-polarization (or Chebyshev) constant of $\mathbb{S}^{d-1}$.
\end{definition}
It is known that $l^2$-polarization  constant is $\frac nd$, 
when $\vec u_1,\cdots,\vec u_n$ form an isotropic set of unit vectors. While, the $l^1$-polarization  constant coincides with  $C_{n,d}$  (see  \cite{FWY21,AmbrusNietert} for two different proofs, or apply our duality result directly).  

We slightly generalize $C_{n,d}$ to the concept of $l^p$-minimal energy. 
\begin{defn}[$l^p$-minimal energy]
$$C_{n,d,p}:=\min\limits_{\|\vec v_i\|_p=1,\forall i}\max\limits_{\varepsilon_1,\ldots,\varepsilon_n\in\{-1,1\}}\|\varepsilon_1\vec v_1+\ldots+\varepsilon_n\vec v_n\|_p$$    
\end{defn}
\begin{prop}\label{prop:Cndp-minimal-energy}
The constant $C_{n,d,p}$ equals 
the best lower bound
of the $l^p$-diameter of the $d$-dimensional zonotope generated by $n$ unit
segments.
\end{prop}
 
And then we give an eigenvalue estimate by $l^p$-minimal energy:
\begin{theorem}\label{thm:p-Lap-zonotope-C-n-d}
For an unnormalized  $p$-Laplacian on 
an oriented $k$-uniform hypergraph\footnote{We refer to \cite[Sections 4.9 and 7.4]{JMZ-book} for the definition of oriented $k$-uniform hypergraph.} $\Gamma=(V,E,w)$,  
$$\lambda_{\max}(\Delta_p)\ge k^{\frac {p}{p^*}}C_{|E|,|V|-m_0,p^*}^p|E|^{1-p}$$

For a normalized  $p$-Laplacian on an oriented $k$-uniform hypergraph 
$$\lambda_{\max}(\Delta_p^N)\ge \frac{1}{n}C_{|V|,|V|-m_0,p}^p$$
where $p^*$ is the H\"older conjugate of $p$, and $m_0$ is the multiplicity of the eigenvalue $0$.

\end{theorem}
\begin{proof}
For the case of unnormalized $p$-Laplacian, 
for any $\vec x\ne\vec 0$, and for any $\varepsilon_e\in\{-1,1\}$, by H\"older's inequality,  
$$\frac{\sum_{e\in E}|\langle\vec w_e,\vec x\rangle|^p}{\|\vec x\|_p^p}= \frac{\sum_{e\in E} |\varepsilon_e|^p|\langle \vec w_e,\vec x\rangle|^p}{\|\vec x\|_p^p}\ge \frac{|\sum_{e\in E} \varepsilon_e\langle \vec w_e,\vec x\rangle|^p}{|E|^{p-1}\|\vec x\|_p^p}= \frac{|\langle \sum_{e\in E} \varepsilon_e\vec w_e,\vec x\rangle|^p}{|E|^{p-1}\|\vec x\|_p^p}$$
where $(\vec w_e)_{e\in E}$ is the hyperedge-vertex incidence matrix of $\Gamma$. 
Thus, by norm duality, we have
$$\max\limits_{x\ne 0}\frac{\sum_{e\in E}|\langle\vec w_e,\vec x\rangle|^p}{\|\vec x\|_p^p}\ge \max\limits_{x\ne 0}\frac{|\langle \sum_{e\in E} \varepsilon_e\vec w_e,\vec x\rangle|^p}{|E|^{p-1}\|\vec x\|_p^p}=\frac{\|\sum_{e\in E} \varepsilon_e\vec w_e\|_{p^*}^p}{|E|^{p-1}}.$$
Since the oriented hypergraph is $k$-uniform, i.e., each hyperedge corresponds to $k$ vertices, and thus $\|\vec w_e\|_{p^*}=k^{\frac{1}{p^*}}$. 
Then, we obtain
$$\lambda_{\max}(\Delta_p)=\max\limits_{x\ne 0}\frac{\sum_{e\in E}|\langle\vec w_e,\vec x\rangle|^p}{\|\vec x\|_p^p}\ge \max\limits_{\varepsilon_1,\ldots,\varepsilon_n\in\{-1,1\}}\frac{\|\sum_{e\in E} \varepsilon_e\vec w_e\|_{p^*}^p}{|E|^{p-1}}
\ge \frac{k^{\frac{p}{p^*}}C_{|E|,|V|-m_0,p^*}^p}{|E|^{p-1}}$$
where the last inequality uses the fact that $|V|-m_0=\dim\mathrm{span}(\vec w_e:e\in E)$. 

Suppose that $(\vec v_i)_{i\in V}$ indicates the vertex-hyperedge incidence matrix of $\Gamma$, that is, the ‌transposed matrix‌ of $(\vec w_e)_{e\in E}$.  
Applying Theorem \ref{th:critical(fA,g)} to $p$-Laplacian eigenvalues, for the largest eigenvalue of the normalized $p$-Laplacian $\Delta_p^N$, 
we have 
$$\lambda_{\max}(\Delta_p^N)= \max\limits_{x\ne 0}\frac{\sum_{e\in E}|\langle\vec w_e,\vec x\rangle|^p}{\sum_{i\in V}\deg_i|x_i|^p}=\max\limits_{y\ne 0}\left(\frac{\sum_{i\in V}\deg_i^{1-p^*}|\langle\vec v_i,\vec y\rangle|^{p^*}}{\|\vec y\|_{p^*}^{p^*}}\right)^{\frac{p}{p^*}}$$
 For any $\vec y\ne\vec 0$, and for any $\varepsilon_i\in\{-1,1\}$, by H\"older's inequality,  
$$\frac{\sum_{i\in V}\deg_i^{1-p^*}|\langle\vec v_i,\vec y\rangle|^{p^*}}{\|\vec y\|_{p^*}^{p^*}}= \frac{\sum_{i\in V} |\varepsilon_i|^{p^*}|\langle \deg_i^{-\frac1p}\vec v_i,\vec y\rangle|^{p^*}}{\|\vec y\|_{p^*}^{p^*}}\ge \frac{|\langle\sum_{i\in V} \varepsilon_i \deg_i^{-1/p}\vec v_i,\vec y\rangle|^{p^*}}{|V|^{p^*-1}\|\vec y\|_{p^*}^{p^*}}$$
Thus,
$$\max\limits_{y\ne 0}\frac{\sum_{i\in V}\deg_i^{1-p^*}|\langle\vec v_i,\vec y\rangle|^{p^*}}{\|\vec y\|_{p^*}^{p^*}}
\ge \max\limits_{y\ne 0} \frac{|\langle\sum_{i\in V} \varepsilon_i \deg_i^{-1/p}\vec v_i,\vec y\rangle|^{p^*}}{|V|^{p^*-1}\|\vec y\|_{p^*}^{p^*}}
=\frac{\|\sum_{i\in V} \varepsilon_i \deg_i^{-1/p}\vec v_i\|^{p^*}_p}{|V|^{p^*-1}}.$$
By $\|\deg_i^{-1/p}\vec v_i\|_p=1$, the maximum eigenvalue of the normalized $p$-Laplacian 
$$\lambda_{\max}(\Delta_p^N)=\max\limits_{y\ne 0}\left(\frac{\sum_{i\in V}\deg_i^{1-p^*}|\langle\vec v_i,\vec y\rangle|^{p^*}}{\|\vec y\|_{p^*}^{p^*}}\right)^{\frac{p}{p^*}} \ge \frac{\|\sum_{i\in V} \varepsilon_i \deg_i^{-1/p}\vec v_i\|^{p}_p}{|V|}\ge\frac{1}{n}C_{n,n-m_0,p}^p$$
where $n=|V|$, and we use the fact that $n-m_0=\mathrm{rank}(\vec w_e)_{e\in E}=\mathrm{rank}(\vec v_i)_{i\in V}=\dim\mathrm{span}(\vec v_i:e\in E)$. 
\end{proof}


\subsection{Application to nonlinear eigenproblems and bifurcation problems}
\label{sec:eigen&bifurcation}
\subsubsection{Nonlinear eigenproblems agree with fixed point problems}
\label{sec:eigenproblem}
Definition \ref{defn:critical(f,g)} leads to a nonlinear eigenproblem. In fact, for two locally Lipschitz continuous functions $f$ and $g$, if $\lambda\in\R$ and $x\in X\setminus0$ satisfy  
$$0\in\partial f(x)-\lambda\,\partial g(x),$$
then we say that $x$ is an \emph{eigenvector}, $\lambda$ is an \emph{eigenvalue}, and $(\lambda,x)$ is an \emph{eigenpair} of the function pair $(f,g)$. 
In this section, such $f$ and $g$ are further assumed to be positively one-homogeneous, convex, and positive-define (i.e., $f(x)>0$ whenever $x\ne0$).

\begin{defn}[\cite{Aubin}]
For a set-valued map $F \colon X \rightrightarrows X$, we say that $x\in X$ is a \emph{fixed point} of $F$ if $x\in F(x)$. The \emph{composition} $F\circ G$ of two set-valued maps $F,G\colon X \rightrightarrows X$ is simply defined as $F\circ G(x):=\bigcup_{z\in G(x)}F(z)$, and we also write it as $F(G(x))$, where $ x\in X$.     
\end{defn}

Next, we give the following result which indicates that the nonlinear eigenproblems agree with set-valued fixed point problems.
 
\begin{prop}
For any $x\ne0$, the following three statements are equivalent:
\begin{itemize}
\item[(1)] $x$ is an eigenvector of $(f,g)$
\item[(2)] $x/f(x)$ is a fixed point of $\partial f^\polar \circ\partial g$
\item[(3)] $x/g(x)$ is a fixed point of $\partial g^\polar \circ\partial f$
\end{itemize}
\end{prop}

\begin{proof}

(1) $\Rightarrow$ (2):  
Since 
$x$ is an eigenvector, 
there exists $\lambda>0$ such that $0\in\partial f(x)-\lambda\,\partial g(x)$, i.e., 
$$ \partial f(x)\bigcap\lambda\,\partial g(x) \ne\emptyset.$$
It means that there is $u\in \partial g(x)$ such that $\lambda u\in \partial f(x)$. Combining with the 0-homogeneity of $\partial g$ and $\partial f^\polar $,  by the claim in the proof of Theorem \ref{th:critical(f,g)}, we have $$\frac{x}{f(x)}\in \partial f^\polar (\lambda u)=\partial f^\polar ( u)\subset \partial f^\polar (\partial g(x))= \partial f^\polar \Big(\partial g\big(\frac{x}{f(x)}\big)\Big),$$ 
implying that $x/f(x)$ is a fixed point of $\partial f^\polar \circ\partial g$. 

(2) $\Rightarrow$ (1): Since $x$ is an eigenvector of $(f,g)$ if and only if $c x$ is an eigenvector of $(f,g)$ whenever $c>0$ is fixed, we assume without loss of generality that $f(x)=1$ and $x$ is a fixed point of $\partial f^\polar \circ\partial g$, i.e., 
$x\in \partial f^\polar (\partial g(x))$. Then there exists $u\in \partial g(x)$ such that $x\in \partial f^\polar (u)$. Thus, by the claim in the proof of Theorem \ref{th:critical(f,g)}, we have $u/f^\polar (u)\in \partial f(x)$. Taking $\lambda=1/f^\polar (u)$, we have $\lambda u\in \partial f(x)\bigcap\lambda\,\partial g(x)$. This implies that $x$ is an eigenvector of $(f,g)$. 

The proof of ``(1) $\Leftrightarrow$ (3)'' is similar, and thus we omit it.
\end{proof}

In a similar manner, we have:
\begin{prop}
$x$ is an eigenvector of $(f\circ A,g)$ if and only if $x/g(x)$ is a fixed point of $\partial g^\polar \circ A^*\circ\partial f\circ A$
\end{prop}
\begin{proof}
We show a proof sketch. If $x\in \partial g^\polar(A^*\partial f(Ax))$, then there exists $y\in A^*\partial f(Ax)$ such that $x\in \partial g^\polar(y)$ and $y\in A^*\partial f(Ax)=\partial (f\circ A)(x)$. Thus, $y/g^\polar (y)\in \partial g(x)$, and therefore,  $y\in \partial (f\circ A)(x)\bigcap g^\polar (y) \partial g(x)$  implying that $x$ is an eigenvector of $(f\circ A,g)$.

Conversely, if $(\lambda,x)$ is a nontrivial eigenpair of $(f\circ A,g)$, then there exist $u\in\partial g(x)$ and $v\in\partial f(Ax)$ such that $A^*v=\lambda u$. Accordingly, 
$$\frac{x}{g(x)}\in \partial g^\polar(u)=\partial g^\polar(\lambda u) =\partial g^\polar(A^*v)\subset \partial g^\polar(A^*\partial f(Ax))=\partial g^\polar(A^*v)\subset \partial g^\polar\Big(A^*\partial f\big(A\frac{x}{g(x)}\big)\Big)$$
yielding that $x/g(x)$ is a fixed point of $\partial g^\polar \circ A^*\circ\partial f\circ A$.
\end{proof}





\subsubsection{
Nonlinear power  iterations coincide with Dinkelbach schemes}
We connect several prominent approaches from the literature which approximate solutions to nonlinear eigenvalue problems. 

To compute the minimum or maximum 
of $f/g$, there are two standard ways: the nonlinear power iteration \cite{Pham84,Bhaskara11}, and the  Dinkelbach scheme.  
We will show that these two approaches are essentially equivalent if we add a natural and widely used relaxation step to the inner problem. 

In particular, we are aiming to solve:
\begin{equation}\label{eq:maxfAx/gx}
\max_{x\ne 0}\frac{f(Ax)}{g(x)}    
\end{equation}

The nonlinear power method for solving 
\eqref{eq:maxfAx/gx} is by utilizing the power operator  
$$Tx:=\partial g^\polar (A^* \partial f(Ax))$$ and doing iteration starting from an initial point  $x_0\ne0$:  
\begin{equation}\label{eq:power-iter}
\begin{cases}x_{k+1}\in Tx_k\\
r^{k+1}=\frac{f( Ax^{k+1})}{g( x^{k+1})}
\end{cases}
\end{equation}

The Dinkelbach (2-step) scheme is the iteration:
\begin{equation}
		\label{dinkel-iter0}
		\begin{cases}
x_{k+1}\in\mathop{\mathrm{argmax}}\limits_{x\ne0} \{f(Ax)-r^k g(x)\}, \\
		r^{k+1}=\frac{f( Ax^{k+1})}{g( x^{k+1})}.  
		\end{cases}
	\end{equation}
For any initial point $ x_0\ne0$, the solution produced by the algorithm above converges to the solution of \eqref{eq:maxfAx/gx} exactly. However, the inner problem $\max_{x} f(Ax)-rg(x)$ is usually not a convex function (it's actually a DC programming). In practice, we may use the following relaxation of \eqref{dinkel-iter0} as follows:
	\begin{equation}
		\label{dinkel-iter1}
		\begin{cases}
x_{k+1}\in\mathop{\mathrm{argmax}}\limits_{g(x)\le 1} \{\langle y^k,x\rangle-r^k g(x)\}, \\
y^{k+1}\in A^*\partial f(Ax), \\
		r^{k+1}=\frac{f( Ax^{k+1})}{g( x^{k+1})}. 
		\end{cases}
	\end{equation}
The above iteration \eqref{dinkel-iter1} is called the relaxed Dinkelbach scheme (or the Dinkelbach 3-step scheme). 

It is interesting that the nonlinear power method that involves the norm-like duality is equivalent to the relaxed Dinkelbach scheme in some sense.

\begin{prop}\label{pro:Dinkelbach-power}
The relaxed Dinkelbach scheme \eqref{dinkel-iter1}  is equivalent to the  power iteration \eqref{eq:power-iter} in the sense that from the same nontrivial initial point, in each step,  their outputs 
are the same.  
\end{prop}

\begin{proof}
Note that each step in \eqref{dinkel-iter1} can be rewritten as 
\begin{equation}
		\label{dinkel-iter1+}
		\begin{cases}
        
y\in A^*\partial f(Ax), \\
		c=\frac{f( Ax)}{g( x)}, \\
z\in\mathop{\mathrm{argmax}}\limits_{g(z')\le 1} \{\langle y,z'\rangle-c g(z')\}, 
		\end{cases}
	\end{equation}
while each step in \eqref{eq:power-iter} can be reformulated as 
\begin{equation}\label{eq:power-iter+}
\begin{cases}y\in A^*\partial f(Ax), \\
z\in \partial g^\polar (y).
\end{cases}
\end{equation}
We shall prove that for any nontrivial input $x\not\in\ker A$, the outputs $z$ produced by   \eqref{dinkel-iter1+} and \eqref{eq:power-iter+} are the same. Precisely, 
\begin{equation}\label{eq:power-subset-Din}
\partial g^\polar (y)\subset \mathop{\mathrm{argmax}}\limits_{g(z')\le 1} \{\langle y,z'\rangle-c g(z')\},     
\end{equation}
\begin{equation}\label{eq:power=Din/normalize}
\partial g^\polar (y)=\Big\{\frac{z}{g(z)}\,\Big|\; z\in \mathop{\mathrm{argmax}}\limits_{g(z')\le 1} \{\langle y,z'\rangle-c g(z')\}\setminus\{0\}\Big\}
\end{equation}
and 
\begin{equation}\label{eq:power=conditional-Din}
\partial g^\polar (y)= \mathop{\mathrm{argmax}}\limits_{g(z')\le 1} \{\langle y,z'\rangle-c g(z')\} \text{ whenever }f(Ax)/g(x)<f(Az)/g(z).\end{equation}
\begin{itemize}

\item Proof of \eqref{eq:power-subset-Din}: 
Let $z\in \partial g^\polar (y)$.  Then, by the claim in the proof of Theorem \ref{th:critical(f,g)}, we have $g(z)=1$ and $y/g^\polar (y)\in \partial g(z)$. Consequently, $0\in y-g^\polar (y)  \partial g(z)=\partial_z\big(\langle y,z\rangle-g^\polar (y)g(z)\big)$. This implies that $z$ is a critical point of the concave function $z'\mapsto \langle y,z'\rangle-g^\polar (y) g(z')$. Since the only critical value of a concave function must be the maximum, we have that $$z\in \mathop{\mathrm{argmax}}\limits_{g(z')= 1}\langle y,z'\rangle-g^\polar (y) g(z') \subset  \mathop{\mathrm{argmax}}\limits_{g(z')\le 1}\langle y,z'\rangle-g^\polar (y) g(z') \subset\mathop{\mathrm{argmax}}\limits_{z'}\langle y,z'\rangle-g^\polar (y) g(z').
$$
Since $\max\limits_{g(z)\le 1}\langle y,z\rangle-c g(z)\ge \langle y,x\rangle-c g(x)=f(Ax)-cg(x)=0$, there exist maximizers on the boundary  constraint $g=1$. Thus, combining the above facts, we obtain
\begin{align*}
z&\in 
\mathop{\mathrm{argmax}}\limits_{g(z')= 1}\langle y,z'\rangle-g^\polar (y) =\mathop{\mathrm{argmax}}\limits_{g(z')= 1}\langle y,z'\rangle= \mathop{\mathrm{argmax}}\limits_{g(z')= 1}\langle y,z'\rangle-c
\\&=\mathop{\mathrm{argmax}}\limits_{g(z')= 1}\langle y,z'\rangle-cg(z')\subset \mathop{\mathrm{argmax}}\limits_{g(z')\le 1}\langle y,z'\rangle-cg(z').    
\end{align*}
 \item Proof of \eqref{eq:power=Din/normalize}: 
Let $z\in\mathop{\mathrm{argmax}}\limits_{g(z')\le 1} \{\langle y,z'\rangle-c g(z')\}$ with $z\ne0$, then
$$\frac{z}{g(z)}\in \mathop{\mathrm{argmax}}\limits_{g(z')= 1}\langle y,z'\rangle-c g(z')=\mathop{\mathrm{argmax}}\limits_{g(z')= 1}\langle y,z'\rangle-c=\mathop{\mathrm{argmax}}\limits_{g(z')= 1}\langle y,z'\rangle.$$
Therefore, 
$$z\in \mathop{\mathrm{argmax}}\limits_{z'\ne0}\frac{\langle y,z'\rangle}{g(z')}$$
and thus
$$0\in\partial_z \frac{\langle y,z\rangle}{g(z)}\subset \frac{yg(z)-\langle y,z\rangle\partial g(z)}{g^2(z)}.$$
This derives $$ y\in \frac{\langle y,z\rangle}{g(z)}\partial g(z) $$
and hence, by the claim in the proof of Theorem \ref{th:critical(f,g)}, we have  $$\frac{z}{g(z)}\in \partial g^\polar (y).$$
Combining with \eqref{eq:power-subset-Din}, we complete the proof of \eqref{eq:power=Din/normalize}.
\item 
Proof of \eqref{eq:power=conditional-Din}: It remains to show that if $f(Ax)/g(x)<f(Az)/g(z)$, then $g(z)=1$ for any $z\in \mathop{\mathrm{argmax}}\limits_{g(z')\le 1} \{\langle y,z'\rangle-c g(z')\}$.

Note that $\langle y,x\rangle-c g(x)=\langle A^*y',x\rangle-c g(x)=\langle y',Ax\rangle-c g(x)=f(Ax)-cg(x)=0$, where $y'\in\partial f(Ax)$. 
It is not difficult to see that $\max\limits_{g(z')\le 1} \{\langle y,z'\rangle-c g(z')\}>0$ because otherwise $\langle y,z'\rangle-c g(z')\le0=  \langle y,x\rangle-c g(x)$, $\forall z'$, implying that 
$$\frac{f(Ax)}{g(x)}=c\ge \sup_{z'\ne0}\frac{\langle y,z'\rangle}{g(z')}\ge \frac{f(Az)}{g(z)}$$
which contradicts to the condition of \eqref{eq:power=conditional-Din}. Therefore, $\langle y,z\rangle-c g(z)>0$ meaning that $z\ne 0$, and moreover, $g(z)=1$ (because otherwise $g(z)<1$ and by the one-homogeneity, $\langle y,z\rangle-c g(z)<\langle y,\frac{z}{g(z)}\rangle-c g(\frac{z}{g(z)})$, a contradiction with $z\in\mathop{\mathrm{argmax}}\limits_{g(z')\le 1} \{\langle y,z'\rangle-c g(z')\}$).
\end{itemize}
\end{proof}





\subsubsection{bifurcation  problem}
Let $X$ and $Y$ be Banach spaces, let $\Omega$ be an open neighborhood of $0$ in $X$, and let $I$ be an open interval.  Consider a mapping $F\in C(I\times\Omega,Y)$ such that $F(\lambda,0)=0$, $\forall \lambda\in I$. A point $(\lambda_0,0)\in I\times\Omega$  is a \emph{bifurcation point} for the equation 
\begin{equation}\label{eq:bifurcation}
F(\lambda,x)= 0    
\end{equation} if every neighborhood of $(\lambda_0,0)$ contains at least one solution $(\lambda,x)$ of \eqref{eq:bifurcation} such that $x\ne0$. 
(see \cite{Mawhin}, or Definition 28.1 in \cite{Deimling}). 

Let $\Omega$ be an unbounded subset of $X$. Then $(\lambda_0;\infty)$ is said to be an asymptotic bifurcation point for the equation \eqref{eq:bifurcation} iff there exist zeros $(\lambda_n,x_n)$ of $F$ such that $\lambda_n\to \lambda_0$ and $\|x_n\|\to\infty$ as $n\to\infty$. 
(see Definition 28.2 in \cite{Deimling})

In this section, we consider the set-valued version of the bifurcation problem which is a slight generalization of the above case, and is widely used in the general bifurcation theory, see \cite{Mawhin}. 
Specifically, we consider the gradient-like operator $F$ defined as $F(\lambda,x):=\partial f(x)-\lambda\partial g(x)$, 
and we concentrate on the following bifurcation point problem: 
\begin{equation}\label{eq:bifurcation-gradient}
\partial f(x)-\lambda\partial g(x)\ni 0
\end{equation} 
and its dual bifurcation  problem
\begin{equation}\label{eq:polardual-bifurcation-gradient}
\partial g^\circ(x)-\lambda\partial f^\circ(x)\ni 0.
\end{equation} 

\begin{theorem}
Suppose that $f,g\in \cvx_0^p(X)$. 
Then, the (asymptotic) bifurcation points of \eqref{eq:bifurcation-gradient} and its dual \eqref{eq:polardual-bifurcation-gradient} coincide exactly. 
\end{theorem}

\begin{proof}
Following the eigenproblem interpretation of Lagrange critical value (see Definition \ref{defn:critical(f,g)} and Section \ref{sec:eigenproblem}), and by the definition of bifurcation problem, it is clear that 
$(\lambda_0,0)$ is a bifurcation point of \eqref{eq:bifurcation-gradient} if and only if it is a limit point of the eigenpairs of \eqref{eq:bifurcation-gradient}; $(\lambda_0;\infty)$ is an asymptotic bifurcation point of \eqref{eq:bifurcation-gradient} if and only if it is an asymptotic limit point of the eigenpairs of \eqref{eq:bifurcation-gradient}. 

Since $f$ and $g$ are $p$-homogeneous, we have that if $(\lambda,x)$ is an eigenpair of \eqref{eq:bifurcation-gradient}, then $(\lambda,tx)$ is also an eigenpair for any $t>0$. The proof is easy. In fact, by $t>0$ and the $(p-1)$-homogeneity of $\partial\f$ and $\partial\g$, we have
\begin{align*}
0 \in \partial \f( x) - \lambda \partial \g(  x) 
 \Longleftrightarrow 0 \in t^{p-1}\partial \f( x) - \lambda t^{p-1} \partial \g(x)  \Longleftrightarrow 0 \in \partial \f(t x) - \lambda   \partial \g(t x) .
\end{align*}
This means that $(\lambda,x)$ is an eigenpair of \eqref{eq:bifurcation-gradient} iff $(\lambda,tx)$ is an eigenpair of \eqref{eq:bifurcation-gradient}. 

By Theorem \ref{th:critical(f,g)}, if $(\lambda,x)$ is an eigenpair of  \eqref{eq:bifurcation-gradient}, then $(\lambda,u)$ is an eigenpair of \eqref{eq:polardual-bifurcation-gradient}, where $u\in\partial g(x)\cap \mathrm{cone}(\partial f(x))$. 


In addition, if  $(\lambda,0)$ is a  bifurcation point of \eqref{eq:bifurcation-gradient}, there exists a sequence of eigenpairs  $(\lambda_n,x_n)$ of \eqref{eq:bifurcation-gradient} such that $(\lambda_n,x_n)\to (\lambda,0)$. By Theorem \ref{th:critical(f,g)}, $(\lambda_n,u_n)$ is an eigenpair of   \eqref{eq:polardual-bifurcation-gradient}, where $u_n\in \partial g(x_n)\cap \mathrm{cone}(\partial f(x_n))$. Taking $t_n>0$ such that $\|t_nu_n\|<1/n$, then $(\lambda_n,t_nu_n)$ is an eigenpair of   \eqref{eq:polardual-bifurcation-gradient} and $(\lambda_n,t_nu_n)\to (\lambda,0)$, which implies that $(\lambda,0)$ is a  bifurcation point of \eqref{eq:polardual-bifurcation-gradient}.  
The rest of the proof is similar and hence we omit it.
\end{proof}


\subsection{
Fenchel duality for difference of convex functions
}
\label{sec:dc}





In \cite{Artstein09},  Artstein-Avidan and Milman show that any involution on $\cvx(X)$ which is order-reversing, must be, up to linear terms, the well known Fenchel dual. 
In this section, we consider the Fenchel dual, and extend the previous duality theory so that it applies to general pairs of convex functions in $\cvx(X)$, rather than just pairs of nonnegative homogeneous convex functions in $\cvx_0^p(X)$.  
We use $f^*$ to express the commonly used Fenchel duality, which is defined as
$$ f^*(x^*):=\sup_{x\in X}\big(\langle x^*,x\rangle-f(x)\big) $$
where 
$x^*\in X^*$. 

\subsubsection{Background on DC functions }
DC function appears in mathematical optimization and convex differential geometry \cite{dual-Tatsuya24} and is of great significance in geometry \cite{Csornyei}. 
In fact, one source of the duality theory for DC functions is precisely the Clarke duality, which has proved highly successful in the study of Hamiltonian systems \cite{AbbondandoloKang}. 
Below, we present a more detailed introduction.








\begin{itemize}
\item DC programming

From the study of Toland-Singer dual equality \cite{dual-Toland79,Singer79}, 
there have been efficient algorithms for DC programming. 
Toland considered the points $x$ and $y$ satisfying the equations 
\begin{align}
0\in \partial f(x)-\partial g(x), \label{eq:f-g-Tolandcri}  \\
0\in \partial g^*(y)-\partial f^*(y)\label{eq:g*-f*-Tolandcri} 
\end{align}
which are called the critical points of $f-g$ and $g^*-f^*$, respectively, 
and he showed that the problems \eqref{eq:f-g-Tolandcri} and \eqref{eq:g*-f*-Tolandcri} are equivalent via Legendre transformation. 
 Toland gave the definition of critical points from the equation point of view. In fact, from the variational perspective, 
 the extreme points of $f-g$ satisfy \eqref{eq:f-g-Tolandcri}. However, in general, $\partial (f(x)-g(x))\subsetneqq \partial f(x)-\partial g(x)$, which means that a solution satisfying \eqref{eq:f-g-Tolandcri} does not need to be a Clarke critical point of $f-g$. 
In this section, we complete a study of the criticality of $f-g$ and its dual, which does not depend on DC decomposition. 
This enriches the duality theory in DC programming. 

\item DC differential geometry 

A hypersurface is delta-convex  if it is locally represented by a graph of a DC function \cite{dual-Tatsuya24}. 
It is known that the singular set of a convex function $f:\R^n\to\R$ is a subset of a countable union of delta-convex hypersurfaces. If one uses the distance function from  any closed subset (of any complete Finsler manifold) instead of convex functions, then the singular set is in fact equal to a union of DC hypersurfaces up to an exceptional set of codimension two.  
These results  indicate that to study singular sets of  convex functions or distance functions, we also need a finer  analysis of DC functions.

\item ``discrete'' DC optimization:  (DS) difference of submodular functions

Submodular function is an important topic,  while in many situations, we may have to deal with the difference of two submodular functions which, however, does not need to be submodular. Fortunately, 
by Choquet integral, submodular functions can be \emph{equivalently} transformed to convex functions. 
Furthermore, any set-function on the power set of a finite ambient set can be decomposed as the difference of two submodular functions, and thus, by Choquet extension, can be transformed to a DC function. %
 It would be good to use duality method in convex analysis to further investigate the topics around set-functions on finite sets. 

\end{itemize}

\subsubsection{Critical duality equivalence on DC functions}\label{sec}

It is well known that the mostly used criticality of a DC function generally depends on the specific DC decomposition of the objective function and is therefore not an intrinsic property, since a given DC function admits infinitely many decompositions (see \cite{LeThi18}). 
It would be good to find some properties which do not depend on DC decompositions. 
In this section, we establish a critical duality theory for DC functions that is independent of DC decompositions, which indeed answers a question left open from the works of Toland on DC functions and the works of Le Thi and Pham Dinh on DC programming. 

In fact, we shall focus on Morse criticality and min-max critical values -- two stronger notions that do not depend on any particular DC decomposition -- and moreover, satisfy duality properties under reasonable conditions. 

Below, we state our main duality result for DC functions in this section.

\begin{theorem}
\label{th:sublevel-equi}
Given $A\in\mathcal{B}(X,Y)$ 
with closed range, $f\in \cvx(Y)$ and $g\in \cvx(X)$.  Assume that there exists a closed subspace $X_1\subset X$ such that $\ker A$ and $X_1$ are complemented, and there exists a closed subspace $Y_1'\subset Y^*$ such that $\ker A^*$ and $Y_1'$ are complemented. 
Suppose that $f^*$ and $g$ are $C^1$-smooth (or, $g^*$ and $f$ are $C^1$-smooth), $\partial g(X_1)\subset R(A^*)$ and $\partial f^*(Y_1')\subset R(A)$. 

Then, for any $c\in\R$, the sub-level set  $\{f\circ A-g\le c\}\cap X_1
$ is homotopy equivalent to $\{g^*\circ A^*-f^*\le c\}\cap Y_1'
$. 
Moreover, the Morse critical points of $f\circ A-g$ restricted on $X_1$ are one-to-one correspondence with the Morse critical points of $g^*\circ A^*-f^*$ restricted on $Y_1'$; the Rothe critical groups are also invariant under such one-to-one correspondence. 
\end{theorem}

In Theorem \ref{thm:homo-f-g} below, we 
study the min-max critical values \cite{Matousek} via Lyusternik–Schnirelmann theory. 
In the preceding Section \ref{sec:ratio}, we have shown that the refined min-max critical information for RC functions is also preserved under polarity dual.

\begin{theorem}\label{thm:homo-f-g}Under the same condition with Theorem \ref{th:sublevel-equi}, 
the  Lyusternik–Schnirelmann min-max critical values of $f\circ A-g$ and $g^*\circ A^*-f^*$ coincide exactly, i.e., $c_k(f\circ A-g)=c_k(g^*\circ A^*- f^*)$, $\forall k$, where  
$$c_k(f\circ A-g):=\inf_{\substack{
S\subset \{f\circ A-g\le c\}\cap X_1 \\ \mathrm{ind}(S)\ge k  
}}\sup_{x\in S}\left(f(Ax)-g(x)\right),
$$
$$
c_k(g^*\circ A^*- f^*):=\inf_{ 
\substack {S\subset \{g^\polar\circ A^* /f^\polar\le c\}\cap Y_1' \\ \mathrm{ind}(S)\ge k
}}\sup_{x\in S}\left(g^*(A^*x)-f^*(x)\right).$$
\end{theorem}

Before going into the proof, we first take a look at the significance of the above theorem.

\begin{itemize}
\item 
If $X$ and $Y$ are Hilbert spaces, then it is standard to take $X_1=(\ker A)^\bot$ and $Y_1'=(\ker A^*)^\bot$. 
When the operator $A$ is further assumed to be a Fredholm operator, then $R(A)$ is closed, and Theorem \ref{th:sublevel-equi} applies to certain convex functions $f$ and $g$. 
\item If $X$ and $Y$ are general reflexive spaces, and  $A$ is an isomorphism from $X$ to $Y$, then $A^*$ is also an isomorphism from $Y^*$ to $X^*$, and $X_1=X$ and $Y_1'=Y^*$. In this case, we have  for any $c\in\R$, the sub-level set  $\{f\circ A-g\le c\}$ is homotopy equivalent to $\{g^*\circ A^*-f^*\le c\}
$; and the Morse critical points of $f\circ A-g$ are one-to-one correspondence to the Morse critical points of $g^*\circ A^*-f^*$. 

In particular, taking $Y=X$ and $A=\mathrm{id}$, and suppose that $f,g\in \cvx(X)$, and $f^*$ and $g$ are $C^1$-smooth (or, $g^*$ and $f$ are $C^1$-smooth), 
then we have the following statements: 

\begin{enumerate}[(i)]
\item 
For any $c\in\R$, the sub-level set $\{f-g\le c\}$ is homotopy equivalent to $\{g^*-f^*\le c\}$.
\item 

For any  Morse critical point $x$ of $f-g$, $y:=\gradientp g(x)$ is a Morse critical point  of $g^*-f^*$; 
conversely, for any  Morse critical point $y$ of $g^*-f^*$, $x:=\gradientp f^*(y)$ is a  Morse critical point  of $f-g$.
\item 

The Rothe critical groups of $f-g$ at $x$ and the Rothe critical groups of $g^*-f^*$ at $y$ are isomorphic, where $x$ and $y$ are Morse critical points in couple, as described in (ii) above.
\item 
The handlebody decompositions associated with $f-g$ and its dual $g^*-f^*$ are isomorphic.

\item 

For any Lusternik-Schnirelman min-max critical value  
$$c_k(f-g):=\inf_{\mathrm{ind}(S)\ge k}\sup_{x\in S}\big(f(x)-g(x)\big)$$
defined via admissible index, 
we have $c_k(f-g)=c_k(g^*-f^*)$ for any $k$.
\end{enumerate}

\end{itemize}

We remark that (i)-(iv) are direct consequences of  Theorem \ref{th:sublevel-equi}; while (v) is derived from Theorem \ref{thm:homo-f-g}. 
It should be noted  that the 
sublevel sets of $f-g$ and $g^*-f^*$ are homotopy equivalent level by level, which strictly deepen the existing results \cite{dual-Toland79}.



\begin{proof}[Proof of Theorem \ref{th:sublevel-equi}]
We first establish some technical claims. 

Claim 1: For any $y\in Y^*$ with  $A^*y\in\partial g(x)$, we have 
$$g^*(A^*y)-f^*(y)\le f(Ax)-g(x)$$
and the equality holds if and only if 
$y$ is a critical point of $g^*\circ A^*-f^*$ in the sense of \eqref{eq:g*-f*-Tolandcri}.

Proof of Claim 1: Since $A^*y\in\partial g(x)$, we have $g^*(A^*y)=\langle A^*y,x\rangle-g(x)$ and $x\in \partial g^*(A^*y)$. Then, 
\begin{align}
g^*(A^*y)-f^*(y)&=\langle A^*y,x\rangle-g(x) -\sup_{z\in Y}\big(\langle y,z\rangle-f(z)\big)\notag
\\&\le \langle A^*y,x\rangle-g(x) - \big(\langle y,Ax\rangle-f(Ax)\big)=f(Ax)-g(x).\label{eq:g*-f*=f-g}
\end{align}
It is known that $f^*(y)=\langle y,Ax\rangle-f(Ax)$ 
if and only if $Ax\in \partial f^*(y)$ if and only if  $y\in\partial f(Ax)$. Therefore, the `$\le$' in \eqref{eq:g*-f*=f-g} reduces to equality `$=$'  
iff $f(Ax)+f^*(y)=\langle y,Ax\rangle$ iff $Ax\in \partial f^*(y)$ iff $y\in\partial f(Ax)$. 
Since we have $Ax\in A \partial g^*(A^*y)=\partial(g^*\circ A^*)(y)$, the equality case holds iff $Ax\in \partial(g^*\circ A^*)(y)\cap \partial f^*(y)$, i.e., $y$ is a critical point of $(g^*\circ A^*,f^*)$ in the sense of \eqref{eq:g*-f*-Tolandcri}. 

\vspace{0.2cm}


Claim 2: For any $x\in X$ and $0\le t\le1$,  $y\in Y^*$ with $A^*y\in\partial g(x)$, and $z\in X$ with $J(Az)\in \partial f^* (y)$, and define $H(x,t):=tz+(1-t)x$. 
We have
$$f(AH(x,t))-g(H(x,t))\le f(Ax)-g(x).$$
If the inequality in Claim 1 is strict, then 
$f(AH(x,t))-g(H(x,t))< f(Ax)-g(x)$ for any $t\in(0,1]$.

Proof of Claim 2: According to Claim 1, we have
$$f(Az)-g(z)\le g^*(A^*y)-f^*(y)\le f(Ax)-g(x).$$
The assumptions 
$A^*y\in \partial g(x)$
and $J(Az)\in\partial f^*(y)$ imply $\langle y,Az\rangle=f^*(y)+f(Az)$ and $\langle A^*y,x\rangle=g(x)+g^*(A^*y)$. 
For any $t\in[0,1]$, 
\begin{align}
g(H(x,t))-g(x)&=g(tz+(1-t)x)-g(x)\notag
\\&\ge \langle A^*y,tz+(1-t)x-x\rangle \notag
\\&=t \langle A^*y,z-x\rangle = t \big(\langle A^*y,z\rangle -\langle A^*y,x\rangle\big)\notag
\\&= t \big(\langle y,Az\rangle -\langle A^*y,x\rangle\big)\notag
\\&=t\big(f^*(y)+f(Az)-g(x)-g^*(A^*y)\big)\notag
\\&=t\big(f^*(y)-g^*(A^*y)+f(Az)-g(x)\big)\notag
\\&\ge t\big(g(x)-f(Ax)+f(Az)-g(x)\big) \label{eq:by-claim-1}
\\&=t(f(Az)-f(Ax)) \notag
\\&=tf(Az)+(1-t)f(Ax)-f(Ax) \notag
\\&\ge f(tAz+(1-t)Ax)-f(Ax)\notag
\\&=f(AH(x,t))-f(Ax)\notag
\end{align}
where \eqref{eq:by-claim-1} is due to $f^*(y)-g^*(A^*y)\ge g(x)-f(Ax)$ (see Claim 1). 
The proof of Claim 2 is then completed.

\vspace{0.2cm}

We are now in a position to prove Theorem \ref{th:sublevel-equi}. 
Since $g$ and $f^*$ are $C^1$-smooth, the subgradients $A^*y=\partial g(x)$ and $Az=\partial f^*(y)$ are uniquely defined, and they are continuous with respect to $x$. 
Since $R(A^*|_{Y_1'})=R(A^*)$ is closed,  
$A^*|_{Y_1'}:Y_1'\to R(A^*)$ has bounded inverse, which we denote it by $(A^*|_{Y_1'})^{-1}$. Similarly, $A|_{X_1}:X_1\to R(A)$ has bounded inverse, and we denote it as $(A|_{X_1})^{-1}$.   
Let $\acg:=(A^*|_{Y_1'})^{-1}\gradientp g$ and $\afc:=(A|_{X_1})^{-1}\gradientp f^*$. 
Then $\dom\acg=\{x\in X:\partial g(x)\in R(A^*)\}$ and $\acg$ is continuous on its domain. By condition, $\dom\acg
\supset X_1$, and hence $\acg$ is continuous at every point in $X_1$. 

Similarly, $\afc$  is continuous on its domain $\{y^*\in Y^*:\partial f^*(y^*)\in R(A)\}$. By condition, $\dom\afc\supset Y_1'$. 

In addition, $\dom (\afc\circ\acg)=\{x\in X: \partial g(x)\in R(A^*),\partial f^*((A^*|_{Y_1'})^{-1}\partial g(x))\in R(A)\}\supset X_1$,  $\dom (\acg\circ\afc)\supset Y_1'$, $\acg(X_1)\subset Y_1'$ and $\afc(Y_1')\subset X_1$ when we identify $X$ with $X^{**}$, and identify $Y$ with $Y^{**}$. 

Then, we can define $H(x,t):=th(x)+(1-t)x$ with $h(x)=z=\afc(\acg(x))$. Clearly, $H:X_1\times[0,1]\to X_1$ is a continuous map.

By Claim 1, $\acg\{f\circ A-g\le c\}\subset \{g^*\circ A^*-g^*\le c\}$, and in a similar manner, $ \afc\{g^*\circ A^*-g^*\le c\} \subset\{f\circ A-g\le c\}$. By Claim 2, $H(\{f\circ A-g\le c\},t)\subset \{f\circ A-g\le c\}$.

Since $H(\cdot,0)=\mathrm{id}$ and $H(\cdot,1)=\afc\circ\acg$, we have $\afc|_{\{g^*\circ A^*-f^*\le c\}}\circ\acg|_{\{f\circ A-g\le c\}}\simeq \mathrm{id}|_{\{f\circ A-g\le c\}}$. 

Similarly, $\acg|_{\{f\circ A-g\le c\}}\circ\afc|_{\{g^*\circ A^*-f^*\le c\}}\simeq \mathrm{id}|_{\{g^*\circ A^*-f^*\le c\}}$. This derives that $\acg|_{\{f\circ A-g\le c\}}$ induces a homotopy equivalence between $\{f\circ A-g\le c\}$ and $\{g^*\circ A^*-f^*\le c\}$.

Let $x$ be a Morse critical point of $f\circ A-g$ restricted on $X_1$. 
We shall prove that, $\acg(x)$ is a  Morse critical point of $g^*\circ A^*-f^*$ restricted on $Y_1'$. 
Suppose the contrary, that $y=\acg(x)$ is a critical point but a  Morse regular point  of the dual DC function $g^*\circ A^*-f^*$. 
Then $x$ is a critical point of the primal DC function $f\circ A-g$, i.e., $0\in \partial (f\circ A)(x)-\partial g(x)$, 

Assume that $\eta$ is a continuous local flow around $y$, i.e., there exists a neighborhood $U_y$ such that for any $y'\in U_y$, $t\in(0,1]$,
$$(g^*\circ A^*-f^*)(\eta(y',t))<(g^*\circ A^*-f^*)(y') $$
and $\eta(y',0)=y'$. Note that $(f\circ A-g)(\psi(\eta(y',t)))\le (g^*\circ A^*-f^*)(\eta(y',t))$.

We shall construct a continuous flow around $x$. Define $\widetilde{\eta}(x',t)=\afc(\eta(\acg(x'),t))$. 
For any $x'\in U_x$ and $t\in(0,1]$,
\begin{align*}
f\circ A(\widetilde{\eta}(x',t))-g(\widetilde{\eta}(x',t))&=
(f\circ A-g)(\psi(\eta(y',t)))
\\&\le (g^*\circ A^*-f^*)(\eta(y',t))
\\&< (g^*\circ A^*-f^*)(y')
\\&=(g^*\circ A^*-f^*)(\acg(x'))
\\& \le f(Ax)-g(x)
\end{align*}
where $y':=\acg(x')$. 
However, $\widetilde{\eta}(\cdot,\cdot)$ is not a flow as $\widetilde{\eta}(x',0)=\afc(\eta(\acg(x'),0))=\afc( \acg(x') )\ne x'$ in general. 
While we can define $\delta(x')=\|x'-\afc( \acg(x'))\|$ which satisfies that $x'$ is a critical point of the DC function $f\circ A-g$ iff $\delta(x')=0$. Moreover, $\delta(\cdot)$ is continuous. Let  $\eta^\#:U_x\times[0,1]\to X$ be defined by
$$\eta^\#(x',t)=\begin{cases}
H(x',t/\delta(x')),\text{ if }1\ge \delta(x')>t\ge0\\
\widetilde{\eta}(x',t-\delta(x')),\text{ if }1\ge t\ge\delta(x')\ge0
\end{cases}$$
Then $\eta^\#(x',0)=x'$, $\eta^\#$ is continuous, and 
$$(f\circ A-g)(\eta^\#(x',t))<(f\circ A-g)(x'),\;\forall t\in (0,1]. $$
Therefore, we have verified that $\eta^\#$ is a decreasing flow around $x$. 
This implies that $x$ is actually a Morse regular point of $f\circ A-g$, a contradiction to the assumption. 
We then complete the proof of the one-to-one correspondence of Morse critical points via Fenchel duality. 

We are in a position to show the isomorphisms of Rothe critical groups. 

In fact, if $\al$ is an isolated Morse critical point of $f\circ A-g$, then by the homology excision property, $C_*(f\circ A-g,\al)\cong H_*(\{f\circ A-g\le f(A\al)-g(\al)\},\{f\circ A-g\le f(A\al)-g(\al)\}\setminus\{\al\})$. 
Since $\acg(\al)=\bet$, $f(A\al)-g(\al)=g^* (A^*\tau)-f^* (\tau)$, and 
$\acg:\{f\circ A-g\le f(A\al)-g(\al)\}\to \{g^*\circ A^* -f^* \le g^* (A^*\tau)-f^* (\tau)\}$ is a homotopy equivalence, we have that $\acg$ induces a homomorphism $$\acg_*: C_*(f\circ A-g,\al)\to C_*(g^*\circ A^* -f^* ,\tau).$$
Similarly, $\afc $ induces a homomorphism 
$$\afc_*:  C_*(g^*\circ A^* -f^* ,\tau)\to C_*(f\circ A-g,\al).$$
Note that $\acg_*\circ \afc_*=(\acg \circ \afc)_*:C_*(f\circ A-g,\al)\to C_*(f\circ A-g,\al)$ is an isomorphism because $\afc \circ \acg\simeq \mathrm{id}$ by Theorem \ref{th:ratio-sublevel-equiA}. Thus, we obtain that both $\acg_*$ and $\afc_*$ are isomorphisms between $C_*(f\circ A-g,\al)$ and $ C_*(g^*\circ A^* -f^* ,\tau)$. 
\end{proof}

\begin{proof}[Proof of Theorem \ref{thm:homo-f-g}]
By definition, it can be verified that $$c_k(f\circ A-g)=\inf_{\mathrm{ind}\left(\{f\circ A-g\le c\}\cap X_1\right)\ge k }c.$$
To show $c_k(f\circ A-g)=c_k(g^*\circ A^* -f^* )$, 
It suffices to prove 
\begin{equation}\label{eq:index=dual-fA-g}
\mathrm{ind} \left(\{f\circ A-g\le c\}\cap X_1\right)=\mathrm{ind}\left(\{g^*\circ A^* -f^* \le c\} \cap Y_1'\right)   
\end{equation}

We first prove the case that $f^* $ and $g$ are further assumed to be $C^1$-smooth (or, $f$ and $g^* $ are $C^1$-smooth). 

If the admissible  index is a homotopy invariant (see (I3) in Definition \ref{def:admissible-index}), then 
by 
Theorem \ref{th:sublevel-equi}, 
we have $\{f\circ A-g\le c\}\cap X_1\simeq\{g^* \circ A^*-f^* \le c\}\cap Y_1'$ for any $c\in\R$, and then we immediately obtain 
the equality \eqref{eq:index=dual-fA-g}.

If the admissible  index satisfies the nondecreasing property under continuous map (i.e., (I4) in Definition \ref{def:admissible-index}), we can also obtain the equality \eqref{eq:index=dual-fA-g}.  
First, suppose that $g$ and $f^* $ are $C^1$-smooth, then by the continuity of $\acg$ and the nondecreasing property of $\mathrm{ind}$, 
we have 
\begin{align*}
 \mathrm{ind}\big(\{f\circ A-g\le c\}\cap X_1\big)
&\le \mathrm{ind}\big(\acg(\{f\circ A-g\le c\}\cap X_1)\big)
\\&\le \mathrm{ind}\big(\{g^* \circ A^*-f^* \le c\}\cap Y_1'\big).   
\end{align*}
The converse similarly holds. 

Now, we remove the $C^1$-smooth assumption, and use standard approximation method, the equality $c_k(f\circ A-g)=c_k(g^*\circ A^* -f^* )$ can be also verified.  
\end{proof}

{\small 
\addcontentsline{toc}{section}{Bibliography}     
\bibliographystyle{plain}    
\bibliography{ref}  
}

\appendix

\section{Appendix
}

\subsection{Supplement for the function space $\cvx_0^p(X)$}

The following 
results are basic properties on polarity dual and direct consequences  of Definition \ref{def:CV_p}.

\begin{enumerate}[(\textbf{{P}}1)]
\item \label{item:p-homo-polar} If $f\in\cvx_0(X)$ is further assumed to be $p$-homogeneous with $p>1$, then 
$$f^\polar(x^*)=\inf\left\{c\in\R_+:\langle  x^*, x\rangle\le \Big(c\f( x)\frac{p^p}{(p-1)^{p-1}}\Big)^{\frac1p},\forall x\in X \right\}   
.$$ 

Proof: 
Suppose $f(x)>0$ and $\langle  x^*, x\rangle>0$, 
replacing $x$ by $t_0x$ with $t_0=(\langle  x^*, x\rangle/pcf^p(x))^{\frac{1}{p-1}}>0$, it can be verified by elementary computation that $t_0\langle  x^*, x\rangle\le ct_0^pf(x)+1$ if and only if 
$$\langle  x^*, x\rangle^p\le c\f( x)\frac{p^p}{(p-1)^{p-1}}$$

Note that $t_0$ is the minimizer of the function $(0,\infty)\ni t\mapsto ct^pf(x)+1-t\langle  x^*, x\rangle$. 
Thus, the above inequality holds if and only if $t\langle  x^*, x\rangle\le ct^pf(x)+1$, $\forall t>0$. 

\item\label{item:decreasing-along-ker} For any $f\in \cvx_0^p(X)$, we have $f(x)\ge f(  x+  z)$ for any $  z\in  \ker f$ and $  x\in \dom f$. 

Proof: 
Assume the contrary holds:  $f(  x+  z)>f(  x)$ for some $  x\in \dom f$ and  $  z\in \ker f $. Fix such $  x$ and $  z$, and let $\delta=f(  x+  z)-f(  x)>0$. By the convexity  of $f$, for any $t\ge0$, $\frac{1}{1+t}f(  x+(1+t)  z)+\frac{t}{1+t}f(  x)\ge f(  x+  z)$,  which is equivalent to 
\begin{equation}\label{eq:by-convexity-f}
f(  x+(1+t)  z)\ge f(  x+  z)+t(f(  x+  z)-f(  x))=f(  x+  z)+t\delta.
\end{equation}
The $p$-homogeneity and convexity of $f$ imply that $\dom f$ is a convex cone, and the conditions 1 and 2 in Definition \ref{def:CV_p} implies that $ \ker f$ is also a convex cone. Thus, $  z\in \ker f $ implies  $(1+t)  z\in \ker f $.  Then, it follows from $(1+t)  z\in \ker f$ and the convexity and $p$-homogeneity of $f$ that 
 $$\frac12f(  x)=\frac{f(  x)+f((1+t)  z)}{2}\ge f\left(\frac{  x+(1+t)  z}{2}\right)=\frac1{2^p} f(  x+(1+t)  z)$$ which yields $2^{p-1}f(  x)\ge f(  x+(1+t)  z)$. Together with \eqref{eq:by-convexity-f}, we obtain $2^{p-1}f(  x)\ge f(  x+  z)+t\delta$ for any $t>0$, which contradicts $x\in\dom f$.

\item Given $f\in \cvx_0^p(X)$ with the additional property that $\ker f$ is a linear subspace of $X$, we have $f(x)= f(  x+  z)$ for any $  z\in  \ker f$ and $  x\in \dom f$. 

Proof: We have by the above \ref{item:decreasing-along-ker} and the convexity of $f$ that $f(x)\ge \frac12(f(x+z)+f(x-z))\ge f(x)$ which implies $f(x)=f(x+z)$, $\forall z\in\ker f$. 
\item \label{item:1-to-p-homo} If $f\in \cvx_0^1(X)$, then $f^p\in \cvx_0^p(X)$ for any $p\ge1$.

Proof: For any $x,y\in X$, $t\in[0,1]$, we have by H\"older's inequality that
\begin{align*}
f^p(tx+(1-t)y)&\le (tf(x)+(1-t)f(y))^p\le (t+1-t)^{\frac pq}(tf^p(x)+(1-t)f^p(y))
\\&=tf^p(x)+(1-t)f^p(y)    
\end{align*}

\item \label{item:p/root/of/p-homo}
If $f\in \cvx_0^p(X)$ with $p>1$, then $f^{\frac1p}\in \cvx_0^1(X)$.

Proof: It suffices to verify the convexity of $f^{\frac1p}$. 
For any $x,y\in \dom f\setminus\ker f$, $0<f(x),f(y)<\infty$. Then, for any $0\le t\le 1$,  $tf^{\frac1p}(x)+(1-t)f^{\frac1p}(y)>0$. Then
\begin{align*}
&\frac{f(tx+(1-t)y)}{\big(tf^{\frac1p}(x)+(1-t)f^{\frac1p}(y)\big)^p}
\\=\;&f\left(\frac{tf^{\frac1p}(x)}{tf^{\frac1p}(x)+(1-t)f^{\frac1p}(y)}\frac{x}{f^{\frac1p}(x)}+\frac{(1-t)f^{\frac1p}(y)}{tf^{\frac1p}(x)+(1-t)f^{\frac1p}(y)}\frac{y}{f^{\frac1p}(y)}\right)
\\ \le \;&\frac{tf^{\frac1p}(x)}{tf^{\frac1p}(x)+(1-t)f^{\frac1p}(y)} f\Big(\frac{x}{f^{\frac1p}(x)}\Big)+\frac{(1-t)f^{\frac1p}(y)}{tf^{\frac1p}(x)+(1-t)f^{\frac1p}(y)}f\Big(\frac{y}{f^{\frac1p}(y)}\Big)=1
\end{align*}
which implies $f^{\frac1p}(tx+(1-t)y)\le tf^{\frac1p}(x)+(1-t)f^{\frac1p}(y)$. 
If either $x$ or $y$ lies in $X\setminus\dom f$, then the convex inequality clearly holds. 
If $x,y\in \ker f$, then since $\ker f$ is a convex cone, the convex inequality clearly holds. 
If $x\in\dom f\setminus\ker f$ and $z\in \ker f$, then $f((1-t)x+tz)\le f((1-t)x)=(1-t)^pf(x)$, and hence 
$f^{\frac1p}((1-t)x+tz)\le (1-t)f^{\frac1p}(x)\le (1-t)f^{\frac1p}(x)+tf^{\frac1p}(z)$. 
\item \label{item:CV1-polar}
For any $f\in \cvx_0^1(X)$,  $f^\polar\in \cvx_0^1(X^*)$. 

Proof: $f^\polar$ is one-homogeneous because $f^\polar(tx^*)=\sup_{f(x)\le 1}\langle tx^*,x\rangle
= t\sup_{f(x)\le  1}\langle x^*,x\rangle=tf^\polar(x^*)$. 

$f^\polar$ is nonnegative because $f^\polar(x^*)=\sup_{f(x)\le 1}\langle x^*,x\rangle
\ge \langle x^*,0\rangle=0$. 


$f^\polar$ is convex because 
$f^\polar(tx^*_1+(1-t)x^*_2)=\sup_{f(x)\le 1}\langle tx^*_1+(1-t)x^*_2,x\rangle\le t\sup_{f(x)\le 1}\langle x^*_1,x\rangle+(1-t)\sup_{f(x)\le 1}\langle x^*_2,x\rangle=tf^\polar(x_1^*)+(1-t)f^\polar(x_2^*)$

$f^\polar$ is lower semi-continuous since for any convergent sequence $x^k\to x^*$,  
$$\liminf_k f^\polar(x^k)=\lim_{n\to+\infty}\inf_{k\ge n}\sup_{f(x)\le1}\langle x^k,x\rangle\ge \sup_{f(x)\le1}\lim_{n\to+\infty}\inf_{k\ge n}\langle x^k,x\rangle=\sup_{f(x)\le1}\langle x^*,x\rangle=f^\polar(x^*)$$
\item \label{item:Euler} 
If $f\in\cvx_0^p(X)$, then for any $x^*\in\partial f(x)$, $\langle x^*,x\rangle=pf(x)$. 

Proof: $t\langle x^*,x\rangle=\langle x^*,(1+t)x-x\rangle\le f((1+t)x)-f(x)=((1+t)^p-1) f(x)$. Hence, 
$$ \langle x^*,x\rangle\le \lim_{t\to 0^+}\frac{(1+t)^p-1}{t}f(x)=pf(x)=\lim_{t\to 0^-}\frac{(1+t)^p-1}{t}f(x)\le \langle x^*,x\rangle$$
\item  \label{item:p-homo-subdifferen-polar-rela}
If $f\in \cvx_0^1(X)$, then $\forall x\not\in \ker f$, $\partial f^p(x)=pf^{p-1}(x)\partial f(x)$ and $(f^p)^\polar(x^*)=\frac{(p-1)^{p-1}}{p^{p}}(f^\polar(x^*))^p$

Proof: Suppose without loss of generality that $p>1$. 
Since $x\not\in \ker f$, for any $x^*\in \partial f(x)$, $\langle pf^{p-1}(x)x^*,y-x\rangle=pf^{p-1}(x)\langle x^*,y-x\rangle\le pf^{p-1}(x)(f(y)-f(x))\le f^p(y)-f^p(x)$, which implies that $pf^{p-1}(x)x^*\in \partial f^p(x)$. Conversely, for any $x^*\in \partial f^p(x)$, we have by Euler's identity that $\langle x^*,x\rangle=pf^p(x)$. By definition, for any $y\in X$, $\langle x^*,y-x\rangle\le f^p(y)-f^p(x)$, that is, $\langle x^*,y\rangle\le f^p(y)+(p-1)f^p(x)$. 
If $f(y)>0$, then we use $f(x)y/f(y)$ instead of $y$ to obtain $(f(x)/f(y))\langle x^*,y\rangle\le (f(x)/f(y))^pf^p(y)+(p-1)f^p(x)=pf^p(x)$. Thus, $\langle x^*,y\rangle\le pf^{p-1}(x)f(y)$, i.e., $\langle \frac{x^*}{pf^{p-1}(x)},y\rangle\le f(y)$, which holds for any $y\in X$. 
Therefore, $\langle \frac{x^*}{pf^{p-1}(x)},y-x\rangle\le f(y)-f(x)$, $\forall y$, which implies $\frac{x^*}{pf^{p-1}(x)}\in\partial f(x)$. 

The equality
$$(f^p)^\polar(x^*)=\frac{(p-1)^{p-1}}{p^{p}}(f^\polar(x^*))^p$$ 
is derived from \ref{item:p-homo-polar}.
\item \label{item:sublevel-ker-dom-cone} If $f\in\cvx_0(X)$ is further assumed to be $p$-homogeneous with $p\ge1$, then for any real number $c\ge0$, $\{f\le c\}$ is a closed convex set, $\ker f$ is a closed convex cone, and $\dom f$ is a convex cone. 

Proof: It follows from the convexity of $f$ that its sublevel sets are convex, i.e., $\{f\le c\}$ is convex for any real $c$. 
By the lower semicontinuity of $f$, $\{f\le c\}$ is closed, $\forall c\in\R$. 
Since $f\ge 0$, we have $\ker f=\{f\le 0\}$ is a closed convex set. 
By the homogeneity of $f$, its zero set $\ker f$ is a cone. 
Thus, $\ker f$ is a closed convex cone. 
Since for any $t,s>0$, $x,y\in\dom f$, $f(tx+sy)\le (t+s)^{p-1}(tf(x)+sf(y))<+\infty$, we have $tx+sy\in\dom f$, and therefore, $\dom f$ is a convex cone.
\end{enumerate}

As a complement‌ to \ref{item:p-homo-subdifferen-polar-rela}, we have for $p>1$, $\partial f^p(x)=(\dom f)^\polar$ for any $x\in\ker f$. 

A proof is as follows. Since $x\in\ker f$, we have $\forall x^*\in \partial f^p(x)$, $\langle x^*,y\rangle\le f^p(x+y)-f^p(x)=f^p(x+y)\le f^p(y)$ (by \ref{item:decreasing-along-ker}), for any $y\in X$. Replacing $y$ by $ty$ with $t>0$, then $\langle x^*,y\rangle\le t^{p-1}f^p(y)$, $\forall t>0$, implying that $\langle x^*,y\rangle\le0$. Thus, $x^*\in(\dom f)^\polar$. Conversely, for any $x^*\in(\dom f)^\polar$, for any $y\in \dom f$, $\forall x\in\ker f$, $\langle x^*,y\rangle\le 0\le f^p(x+y)=f^p(x+y)-f^p(x)$. This implies that $x^*\in \partial f^p(x)$. 

It follows from \ref {item:1-to-p-homo} and \ref{item:p/root/of/p-homo} that $f\in \cvx_0^1(X)$ if and only if $f^p\in \cvx_0^p(X)$.

For any $f\in\cvx_0^1(X)$, we have the following statements: 
\begin{enumerate}[(\textbf{{S}}1)]

\item \label{item:S1} 
$\ker f^\polar=(\dom f)^\polar$ 

Proof: It suffices to prove 
\text{$f^\polar(x^*)=0$ if and only if $x^*\in(\dom f)^\polar$}: In fact, $0=f^\polar(x^*)=\sup_{f(x)\le 1}\langle x^*,x\rangle$ $\Longleftrightarrow$ $\langle x^*,x\rangle\le 0$ for any $x$ with $f(x)\le 1$ $\Longleftrightarrow$  $\langle x^*,x\rangle\le 0$ for any $x\in\dom f$ $\Longleftrightarrow$ $x^*\in(\dom f)^\polar$. 

\item  \label{item:S2}  $\dom f^\polar\subset (\ker f)^\polar$

Proof: It suffices to prove that \text{for any $  x^*\not\in(\ker f)^\polar$, $f^\polar(x^*)=+\infty$}: Since $  x^*\not\in(\ker f)^\polar$, there exists $z_0\in\ker f$ such that $\langle x^*,z_0\rangle>0$. Thus,  $f^\polar(x^*)=\sup_{f(x)\le1}\langle x^*,x\rangle\ge \sup_{z\in\ker f}\langle x^*,z\rangle\ge \sup_{t>0}\langle x^*,tz_0\rangle=+\infty$.  

\item \label{item:S3} $f\ge f^{\polar\polar}$

Proof: For any $x^*\in X^*$ with $f^\polar(x^*)\le 1$, if $x\in X$ satisfies  $f(x)=+\infty$, then $f^{\polar\polar}(x)\le f(x)$; if $0<f(x)<+\infty$, then $\langle x^*,x/f(x)\rangle\le 1$ which implies $\langle x^*,x\rangle\le f(x)$, and thus $f^{\polar\polar}(x)=\sup_{f^\polar(x^*)\le 1}\langle x^*,x\rangle\le f(x)$; if $f(x)=0$, then $\langle x^*,x\rangle\le0=f(x)$ (because otherwise $f^\polar(x^*)\ge \sup_{t>0}\langle x^*,tx\rangle=+\infty$ a contradiction). In any case, we obtain $f\ge f^{\polar\polar}$. 
\item \label{item:S4} $\overline{\dom f^{\polar\polar}}=\overline{\dom f}=(\ker f^\polar)^\polar$

Proof:  Taking polar transform on both sides of \ref{item:S2}, we obtain $(\dom f^\polar)^\polar\supset \ker f$. 
Replacing $f$ and $f^\polar$ by $f^\polar$ and $f^{\polar\polar}$, respectively, we have $(\dom f^{\polar\polar})^\polar\supset \ker f^\polar$. 
Thus, $\overline{\dom f^{\polar\polar}}=(\dom f^{\polar\polar})^{\polar\polar}\subset (\ker f^\polar)^\polar$. 
By \ref{item:S3}, we have $\dom f\subset\dom f^{\polar\polar}$ and hence  $\overline{\dom f}\subset\overline{\dom f^{\polar\polar}}\subset(\ker f^\polar)^\polar$.  
Taking polar transform on both sides of \ref{item:S1}, we have 
$\overline{\dom f}=(\ker f^\polar)^\polar$ and then $\overline{\dom f^{\polar\polar}}=\overline{\dom f}=(\ker f^\polar)^\polar$.   
\item \label{item:S5}$f=f^{\polar\polar}$

Proof: 
Next, we prove that $\{f^{\polar\polar}\le c\}=\{f\le c\}$ for any $c\in (0,+\infty)$. In fact, $\{f\le c\}\subset \{f^{\polar\polar}\le c\}$ are closed convex sets. If there exists a positive number $c$ such that $\{f\le c\}\subsetneq \{f^{\polar\polar}\le c\}$, then there exists $x_0\in \{f^{\polar\polar}\le c\}\setminus \{f\le c\}$. By Hahn-Banach separation theorem, there exists $x^*$ such that $\langle x^*,x_0\rangle>\sup_{x\in \{f\le c\}}\langle x^*,x\rangle=cf^\polar(x^*)$. 
If $f^\polar(x^*)>0$, then by scaling, we can assume without loss of generality that $f^\polar(x^*)= 1$. 
Then $c\ge f^{\polar\polar}(x_0)=\sup_{f^\polar(y^*)\le 1}\langle y^*,x_0\rangle\ge \langle x^*,x_0\rangle>cf^\polar(x^*)=c$, a contradiction. 
If $f^\polar(x^*)=0$, then taking polarity transform on  \ref{item:S4}, we have $(\dom f^{\polar\polar})^\polar=(\dom f)^\polar=\ker f^\polar$, and therefore   $x^*\in \ker f^\polar=(\dom f)^\polar=(\dom f^{\polar\polar})^\polar$. 
Since $x_0\in \dom f^{\polar\polar}$, we have $\langle x^*,x_0\rangle\le0$, a contradiction. 

Consequently, $\dom f^{\polar\polar}=\cup_{c>0}\{f^{\polar\polar}\le c\}=\cup_{c>0}\{f \le c\}=\dom f$. 
If there exists $x_0$ such that $f(x_0)>f^{\polar\polar}(x_0)$, then $0<f(x_0)<+\infty$. Taking $c=f(x_0)$ and $\varepsilon_0>0$ such that $(1+\varepsilon_0)f^{\polar\polar}(x_0)<c$, we have $(1+\varepsilon_0)x_0\in\{f^{\polar\polar}\le c\}\setminus \{f\le c\}$, a contradiction to the discussion above. Therefore, $f=f^{\polar\polar}$.
\item\label{item:S6}$\overline{\dom f^\polar}= (\ker f)^\polar$ and $\overline{\dom f}= (\ker f^\polar)^\polar$

Proof: Replacing $f$ by $f^\polar$ in \ref{item:S1} and using \ref{item:S5}, we have 
$(\dom f^\polar)^\polar=\ker f^{\polar\polar}=\ker f$. Then, taking polarity on both sides, we obtain $(\ker f)^\polar=(\dom f^\polar)^{\polar\polar}=\overline{\dom f^\polar}$. 
The second statement $\overline{\dom f}= (\ker f^\polar)^\polar$ is a direct consequence of the first one. 
\item If $\ker f\ne \dom f$, then $\ker f^\polar\ne\dom f^\polar$.

Proof: 
Suppose the contrary, that $\ker f^\polar=\dom f^\polar$. Then, by \ref{item:S6}, $\overline{\dom f}= (\ker f^\polar)^\polar=(\dom f^\polar)^\polar=(\overline{\dom f^\polar})^\polar= (\ker f)^{\polar\polar}=\ker f$. Since $\ker f\subset \dom f$ and $\ker f$ is closed, we have $\ker f= \dom f$, a contradiction.

\item If $\ker f$ is a linear subspace, so does $\overline{\dom f^\polar}$. If $\overline{\dom f^\polar}$ is a linear subspace, so does $\ker f^\polar$. 

Proof: According to  \ref{item:S6}, it suffices to show that for a linear subspace $\Pi\subset X$, the polar dual $\Pi^\polar$ is also a linear subspace. 
This is indeed a known result, but for completeness, we provide a short verification: $\forall x^*\in\Pi^\polar$, $\langle x^*,x\rangle\le0$ and hence $\langle x^*,-x\rangle\ge0$, $\forall x\in\Pi$. 
Since $\Pi$ is a linear subspace, $-x\in \Pi$, we also have $\langle x^*,-x\rangle\le0$. This implies $\langle x^*,x\rangle=0$, and therefore, $x^*\in \Pi^\bot$. 
Here, $\Pi^\bot$ represents the annihilator of $\Pi$ which is actually a closed linear subspace of $X^*$.

\item \label{item:partialf-ker} $\partial f(x)\subset (\ker f)^\polar$

Proof: In fact, for any $x^*\in\partial f(x)$, $\langle x^*,x'-x\rangle\le f(x')-f(x)$, $\forall x'\in X$. We may take $x'=x+z$ for $z\in\ker f$. Then, by \ref{item:decreasing-along-ker}, $\langle x^*,z\rangle\le f(x+z)-f(x)\le0$, for any $z\in\ker f$, which implies $\langle x^*,z\rangle\le0$, $\forall z\in\ker f$, and thus $x^*\in (\ker f)^\polar$.

\item \label{item:partialf(0)} $\partial f(0)=\{x^*\in X^*:f^\polar(x^*)\le 1\}$

Proof: 
\begin{align*}
\partial f(0)&=\{x^*\in X^*:\langle x^*,x-0\rangle\le f(x)-f(0)\}
=\{x^*\in X^*:\langle x^*,x\rangle\le f(x)\}\\&=\{x^*\in X^*:\sup_{f(x)\le 1}\langle x^*,x\rangle\le 1\}=\{x^*\in X^*:f^\polar(x^*)\le 1\}    
\end{align*}
\end{enumerate}

\subsection{Proofs of Propositions \ref{pro:CV_p-basic}, \ref{pro:CV_1&Cv_p}, \ref{prop:f-continuous-equiva}, \ref{prop:character:CV_c^1}, \ref{pro:quotient-subdifferential},  Corollary \ref{cor:function-spaces} and Lemma \ref{lem:normal-cone}}




\begin{proof}[Proof of Proposition \ref{pro:CV_p-basic}]We show the proof step by step.
\begin{enumerate}[(i)]
\item See \ref{item:sublevel-ker-dom-cone}. 
\item This directly follows from \ref{item:decreasing-along-ker}.
\item This is \ref{item:p/root/of/p-homo}
\item 
For any $f\in \cvx_0^p(X)$, $f^\circ \in \cvx_0^p(X^*)$. Furthermore, $\ker f^\polar=(\dom f)^\polar$ and $\overline{\dom f^\polar}=(\ker f)^\polar$. 
These are derived from \ref{item:CV1-polar}, \ref{item:S1} and \ref{item:S6} 
\item They are consequences of \ref{item:S5}  and \ref{item:partialf-ker}
\end{enumerate}
The above verification also implicitly‌ uses the relation $\partial f^p(x)=pf^{p-1}(x)\partial f(x)$ and $(f^p)^\polar(x^*)=\frac{(p-1)^{p-1}}{p^{p}}(f^\polar(x^*))^p$ shown in \ref{item:p-homo-subdifferen-polar-rela}.
\end{proof}

\begin{proof}[Proof of Proposition \ref{pro:CV_1&Cv_p}]
(i) see \ref{item:p-homo-subdifferen-polar-rela}. 

(ii) We shall use the equivalent definition 
$f^\polar(x^*)=\inf\left\{c\in\R_+:\langle  x^*, x\rangle\le c\f( x),\forall x\in X \right\} $. Let $c>0$ be such that for any $x$ with $f(x)\le 1$, $\langle  x^*, x\rangle\le c$, i.e.,  $\sup_{f(x)\le 1}\langle  x^*, x\rangle\le c$. We shall prove that $\langle  x^*, x\rangle\le cf(x)$, $\forall x\in X$. 
Suppose the contrary that $\langle  x^*, x\rangle> cf(x)$ for some $x\in X$. Then $x\in\dom f$ and $f(x)<1$. 
If $x\in \ker f$, then $x^*\not\in(\ker f)^\polar$ and by \ref{item:S6}, $\sup_{f(x)\le 1}\langle  x^*, x\rangle=+\infty$. Since $f(tx)=0$ and $\sup_{t>0}\langle  x^*, tx\rangle=+\infty$, we have  $f^\polar(x^*)=+\infty$. If $x\not\in\ker f$, then $\langle  x^*, x/f(x)\rangle>c$ which contradicts to $\sup_{f(x)\le 1}\langle  x^*, x\rangle\le c$. 
Therefore, $\sup_{f(x)\le 1}\langle  x^*, x\rangle\ge f^\polar(x^*)$. 
Conversely, for any $c>0$ such that $\langle  x^*, x\rangle\le cf(x)$, $\forall x\in X$, we have $\sup_{f(x)\le1}\langle  x^*, x\rangle\le c$. Hence, $\sup_{f(x)\le1}\langle  x^*, x\rangle\le f^\polar(x^*)$. 

(iii) see \ref{item:partialf(0)}.
\end{proof}

\begin{proof}[Proof of Proposition \ref{prop:f-continuous-equiva}]
Since $f$ is convex and one-homogeneous, $f$ has subadditivity, that is, $f(x+y)\le f(x)+f(y)$, $\forall x,y\in X$. 
Hence, for any $x_0\in X$, 
$f(x)-f(x_0)\le f(x-x_0)$ and $f(x_0)-f(x)\le f(x_0-x)$, $\forall x\in X$. Thus, $|f(x)-f(x_0)|\le \max\{f(x-x_0), f(x_0-x)\}$. If $f$ is continuous at $0$,  then $\limsup_{x\to x_0}|f(x)-f(x_0)|\le \limsup_{x\to x_0}\max\{f(x-x_0), f(x_0-x)\}=0$, which means that $f$ is continuous at $x_0$. 

Conversely, if $f$ is continuous at $x_0$, for any $y\in X$, $f(y)\le f(y+x_0)+f(-x_0)$. There exists $\delta>0$ such that $|f(y+x_0)-f(x_0)|< 1$ whenever $\|y\|<\delta$. Hence, $f(y)\le f(y+x_0)+f(-x_0)\le f(x_0)+1+f(-x_0)$ for any $y$ with $\|y\|<\delta$. By the one-homogeneity of $f$, for any $\varepsilon>0$, we obtain $|f(x)|=\varepsilon|f(x/\varepsilon)|\le (f(x_0)+1+f(-x_0))\varepsilon$ whenever $\|x\|<\varepsilon\delta$. Therefore, $f$ is continuous at $0$.
\end{proof}

\begin{proof}[Proof of Proposition \ref{prop:character:CV_c^1}]
It suffices to work with the case $p=1$. Denote by $$\widetilde{\cvp}_c^1(X)=\big\{f\in \cvx_{0}^1(X):\exists C_2>C_1>0\text{ s.t. }C_1\|x\|\le f(x)\le C_2\|x\|,\forall x\in X\big\}. $$ 

If $f\in\widetilde{\cvp}_c^1(X)$, then $f$ is continuous at $0$. Note that, $f^\polar(x^*)=\sup_{f(x)\le 1}\langle x^*,x\rangle\le \sup_{C_1\|x\|\le 1}\langle x^*,x\rangle=\frac{1}{C_1}\sup_{\|x\|\le 1}\langle x^*,x\rangle=\|x^*\|^\polar /C_1$ which implies that $f^\polar$ is continuous at $0$. 
This implies that $\widetilde{\cvp}_c^1(X)\subset\cvp_c^1(X)$. 

Conversely, given $f\in \cvp_c^1(X)$, that is, $f$ and $f^\polar$ are continuous at $0$, then there exists $\delta>0$ such that $|f(x)|< 1$ whenever $\|x\|\le \delta$, implying that $|f(x)|=\frac{\|x\|}{\delta}f(\frac{\delta }{\|x\|}x)<\frac{\|x\|}{\delta}$, $\forall x\in X$. That is, by taking $C_2=1/\delta$, we have $|f(x)|\le C_2\|x\|$, $\forall x\in X$.

Now, we shall prove that there exists $C_1>0$ such that $C_1\|x\|\le f(x)$. Suppose the contrary, that there exist $x_n$ with $\frac1n \|x_n\|>f(x_n)$,  $n=1,2,\cdots$. Then, $x_n\ne 0$ and $f(nx_n/\|x_n\|)<1$, $n=1,2,\cdots$. By the Hahn-Banach theorem, there exists $x_n^*\ne0$ such that $\langle x^*_n,nx_n/\|x_n\|\rangle=\|x^*_n\|\|nx_n/\|x_n\|\|=n\|x^*_n\|$. By scaling, we can assume that $\|x^*_n\|=1/n$. Then $\lim_{n\to+\infty}x^*_n=0$ and  
$f^\polar(x^*_n)=  \sup_{f(x)\le 1}\langle x^*_n,x\rangle\ge \langle x^*_n,nx_n/\|x_n\|\rangle=n\|x^*_n\|=1$, which contradicts the continuity of $f^\polar$ at $0$.  

In consequence, we have proved $\cvp_c^1(X)\subset \widetilde{\cvp}_c^1(X)$.
\end{proof}

\begin{proof}[Proof of Corollary \ref{cor:function-spaces}]
The equivalence $f\in \cvx_0^p(X)\Longleftrightarrow f^\polar\in \cvx_0^p(X^*)$ follows directly from \ref{item:1-to-p-homo}, \ref{item:p/root/of/p-homo}, \ref{item:CV1-polar}, \ref{item:p-homo-subdifferen-polar-rela} and \ref{item:S5}. 

Together the above properties with \ref{item:S6}, we derive $f\in \cvx_{0,+}^p(X)\Longleftrightarrow f^\polar\in \cvx_{0,+}^p(X^*)$. 

The equivalence $f\in \cvp_c^p(X)\Longleftrightarrow f^\polar\in\cvp_c^p(X^*)$ is deduced by Proposition \ref{prop:character:CV_c^1}. 
\end{proof}

\begin{proof}[Proof of Proposition \ref{pro:quotient-subdifferential}]
By the definition of $\bar f_\Pi$, $f_\Pi([x])\le f(x)$, $\forall x\in X$. We show a proof via the following two claims.

Claim 1: If $f(x)>\bar f_\Pi([x])$, then $\partial f(x)\cap \Pi^\bot=\varnothing$.

Proof: Suppose the contrary that there exists $x$ satisfying $f(x)>\bar f_\Pi([x])$ and $\partial f(x)\cap \Pi^\bot\ne\varnothing$. Then, there exists $x'\in[x]$ such that $f(x')<f(x)$. Thus, for any $x^*\in\partial f(x)\cap \Pi^\bot$, $0>f(x')-f(x)\ge \langle x^*,x'-x\rangle=0$, a contradiction. 

Claim 2: If $f(x)=\bar f_\Pi([x])$, then $\partial f(x)\cap \Pi^\bot= \pi^*(\partial\bar f_\Pi([x]))$.

Proof: 
We first prove $\partial f(x)\cap \Pi^\bot\subset\pi^*(\partial\bar f_\Pi([x]))$. Suppose $\partial f(x)\cap \Pi^\bot\ne\varnothing$, otherwise the statement is trivial.  
For any $x^*\in\partial f(x)\cap \Pi^\bot$, we shall prove $(\pi^*)^{-1}x^*\in \partial \bar f_\Pi([x])$. In fact, for any $[y]\in X/\Pi$, there exists $y'\in [y]$ such that $\bar f_\Pi([y])>f(y')-\varepsilon$. Thus, $\bar f_\Pi([y])-\bar f_\Pi([x])=\bar f_\Pi([y])-f(x)>f(y')-f(x)-\varepsilon\ge \langle x^*,y'-x\rangle-\varepsilon=\langle x^*,y-x\rangle-\varepsilon=\langle (\pi^*)^{-1}x^*,[y]-[x]\rangle-\varepsilon$. Taking $\varepsilon\to0^+$, we have $\bar f_\Pi([y])-\bar f_\Pi([x])\ge \langle (\pi^*)^{-1}x^*,[y]-[x]\rangle$, which implies that $(\pi^*)^{-1}x^*\in \partial \bar f_\Pi([x])$.

On the other direction, we prove $\pi^*(\partial\bar f_\Pi([x]))\subset\partial f(x)\cap \Pi^\bot$. Suppose $\pi^*(\partial\bar f_\Pi([x]))\ne\varnothing$.  
For any $\bar x^*\in \partial \bar f_\Pi([x])$, we shall prove that $\pi^*(\bar x^*)\in \partial f(x)\cap \Pi^\bot$. In fact, for any $y\in X$, $f(y)-f(x)\ge \bar f_\Pi([y])-\bar f_\Pi([x])\ge \langle \bar x^*,[y]-[x]\rangle=\langle \bar x^*,\pi(y-x)\rangle=\langle \pi^*(\bar x^*),y-x\rangle$, yielding that $\pi^*(\bar x^*)\in \partial f(x)$. For any $z\in\Pi$, $[x+z]=[x]$, and thus $0=\bar f_\Pi([x+z])-\bar f_\Pi([x])\ge \langle \bar x^*,[x+z]-[x]\rangle=\langle \bar x^*,\pi(x+z-x)\rangle=\langle \pi^*(\bar x^*),z\rangle$. Changing $z$ to $-z\in\Pi$, we still have $0\ge \langle \pi^*(\bar x^*),-z\rangle$. Hence, $\langle\pi^*(\bar x^*),z\rangle=0$ for any $z\in \Pi$, meaning that $\pi^*(\bar x^*)\in \Pi^\bot$.
\end{proof}

\begin{proof}[Proof of Lemma \ref{lem:normal-cone}]

For any $x\in\partial K$ and for any $x^*\in\partial f_K(x)$, we have by definition that $f_K(y)-f_K(x)\ge \langle x^*,y-x\rangle$, where $\partial K$ indicates the relative boundary of $K$. 
Since $K=\{x\in X:f_K(x)\le 1\}$ and $\partial K=\{x\in X: f_K(x)=1\}$, we have for any $y\in K$, 
$\langle x^*,y-x\rangle\le f_K(y)-f_K(x)\le 1-1=0$, implying that $x^*\in \mathrm{NC}_x(K)$. Therefore, $\partial f_K(x)\subset \mathrm{NC}_x(K)$. Clearly, $\mathrm{NC}_x(K)$ is a closed cone, and hence, $\mathrm{cl}\,\mathrm{cone}(\partial f_K(x))\subset \mathrm{NC}_x(K)$.

Now, we focus on the inverse. 
Given $x\in \partial K$ and $x^*\in\mathrm{NC}_x(K)$, we shall discuss two cases.

Case 1. $\langle x^*,x\rangle=0$. 

In this case, since $\langle x^*,y\rangle=\langle x^*,y-x\rangle\le 0$ for any $y\in K$, and $0$ lies in the relative interior of $K$, we have $\langle x^*,x'\rangle\le 0$ for any $x'\in\mathrm{span}(K)$. Thus, $\langle x^*,x'\rangle= 0$ for any $x'\in\mathrm{span}(K)$, meaning that $x^*\in \mathrm{span}(K)^\bot$. 

Case 2. $\langle x^*,x\rangle\ne 0$. 

In this case, since $\langle x^*,-x\rangle=\langle x^*,0-x\rangle\le 0$, we have  $\langle x^*,x\rangle> 0$. In addition, for any $x'\in K$, $\langle x^*,x'-x\rangle\le 0$, implying that $\langle x^*,x'\rangle\le \langle x^*,x\rangle$ whenever $x'\in K$. 
Hence, $f_K^\polar(x^*)=\sup_{f_k(x')\le 1}\langle x^*,x'\rangle=\sup_{x'\in K}\langle x^*,x'\rangle=\langle x^*,x\rangle>0$. 

For any $y\in X$, we consider three subcases:

Case 2.1. $f_K(y)=0$

In this subcase, we shall first prove $\langle x^*,y\rangle\le 0$, otherwise, $\langle x^*,y\rangle> 0$ implies that $\langle x^*,ty\rangle>\langle x^*,x\rangle$ for sufficiently large $t>0$. Note that, $f_K(ty)=f_K(y)=0$ meaning that $ty\in K$, then, $\langle x^*,ty-x\rangle>0$ which contradicts $x^*\in\mathrm{NC}_x(K)$. Therefore, 
 $\langle x^*/f_K^\polar(x^*),y-x\rangle\le \langle x^*/f_K^\polar(x^*),-x\rangle=-1=f_K(y)-f_K(x)$. 

Case 2.2. $f_K(y)=+\infty$

In this subcase, we have  $\langle x^*/f_K^\polar(x^*),y-x\rangle<+\infty=f_K(y)-f_K(x)$.

Case 2.3. $0<f_K(y)<+\infty$

In this subcase, $y/f_K(y)\in \partial K$, thus $\langle x^*,y/f_K(y)\rangle\le f_K^\polar(x^*)$ which implies $\langle x^*/f_K^\polar(x^*),y\rangle\le f_K(y)$. 
Since $\langle x^*/f_K^\polar(x^*),x\rangle=\langle x^*,x\rangle/f_K^\polar(x^*)=1=f_K(x)$, we finally obtain 
$\langle x^*/f_K^\polar(x^*),y-x\rangle\le f_K(y)-f_K(x)$.

Combining Cases 2.1-2.3, we have $x^*/f_K^\polar(x^*)\in \partial f_K(x)$, and hence $x^*\in\mathrm{cone}(\partial f_K(x)):=\{tv^*:t>0,v^*\in\partial f_K(x)\}$. 

Next, we prove that $\mathrm{span}(K)^\bot\subset \mathrm{cl}\,\mathrm{cone}(\partial f_K(x))$, i.e., $\forall x^*\in \mathrm{span}(K)^\bot$, $\exists t_n>0$, $v_n^*\in \partial f_K(x)$ such that $t_nv_n^*\to x^*$. In fact, we can fix $v^*\in \partial f_K(x)$, and take $v_n^*:=v^*+nx^*$ and $t_n:=1/n$. 
Since $\dom f_K\subset \mathrm{span}(K)$,  $x^*\in \mathrm{span}(K)^\bot$ and $\langle v^*,y-x\rangle\le f_K(y)-f_K(x)$ $\forall y$, we have $\langle v^*+nx^*,y-x\rangle=\langle v^*,y-x\rangle\le f_K(y)-f_K(x)$ $\forall y\in \mathrm{span}(K)$.
For any $y\not \in \mathrm{span}(K)$, we have $f_K(y)=+\infty$, which implies $\langle v^*_n,y-x\rangle:=\langle v^*+nx^*,y-x\rangle\le f_K(y)-f_K(x)$ $\forall y\in X$. This yields $v_n^*\in \partial f_K(x)$. Clearly, $t_nv_n^*=x^*+v^*/n\to x^*$, $n\to+\infty$. 

In addition, we shall prove that $\mathrm{span}(K)^\bot\cap\mathrm{cone}(\partial f_K(x))=\emptyset$. Suppose the contrary, that $x^*\in \mathrm{span}(K)^\bot\cap\mathrm{cone}(\partial f_K(x))$, we may further assume that $x^*\in \mathrm{span}(K)^\bot\cap\partial f_K(x)$ because we can always use $tx^*$ instead of $x^*$ to ensure this. 
Then $\langle x^*,x'\rangle=0$, $\forall x'\in \mathrm{span}(K)$, and in particular, taking $x'=-x$, we have $0=\langle x^*,0-x\rangle\le f_K(0)-f_K(x)=-1$, a contradiction. 
\end{proof}

\subsection{Examples and Counterexamples on homotopy equivalence}

In Section \ref{sec:ratio}, we propose the critical equivalence theory for RC functions under duality based on polarity, while in Section \ref{sec:dc}, we extend our framework to DC functions by establishing the critical equivalence under Fenchel conjugate. 
In this appendix, we shall provide examples to illustrate that we cannot obtain the same critical duality theory if we choose ``inappropriate'' dualities.

The following example shows that for RC functions, if we use the Fenchel dual rather than the polarity dual, then the critical duality equivalence will not hold. 
\begin{example}\label{exam:f/g-cannot-Fenchel}
Given $p>2$, $0<a_1<a_2<\cdots<a_n$, let $f(x)=(\sum_{i=1}^na_ix_i^2)^{\frac p2}$ and $g(x)=(\sum_{i=1}^n x_i^2)^{\frac p2}$, $\forall x:=(x_1,\cdots,x_n)\in\R^n$.  
Then $$\frac{f(x)}{g(x)}=\Big(\frac{\sum_{i=1}^na_ix_i^2}{\sum_{i=1}^n x_i^2}\Big)^{\frac p2}\;\text{ and }\;\frac{g^*(x)}{f^*(x)}=\Big(\frac{\sum_{i=1}^nx_i^2}{\sum_{i=1}^n x_i^2/a_i}\Big)^{\frac {p^*}{2}}$$ 

It can be checked that the critical values of $f/g$ are $a_1^{\frac p2},\cdots,a_n^{\frac p2}$, 
and every critical point corresponding to $a_j^{\frac p2}$ must be nondegenerate and has the Morse index $j-1$, where $j=1,\cdots,n$. 
Similarly, the critical values of $g^*/f^*$ are $a_1^{\frac {p^*}2},\cdots,a_n^{\frac {p^*}2}$, and every critical point corresponding to $a_j^{\frac {p^*}2}$ must be nondegenerate and has the Morse index $j-1$. 
It can be further checked that for any $c\in (a_j^{\frac p2},a_{j+1}^{\frac p2})$, $\{f/g\le c\}$ must be homotopy equivalent to $\mathbb{S}^{j-1}$, the unit sphere of dimension $j-1$. 
Analogously, $\{g^*/f^*\le c\}$  is homotopy equivalent to $\mathbb{S}^{j-1}$, whenever $c\in (a_j^{\frac {p^*}2},a_{j+1}^{\frac {p^*}2})$. 
Therefore, as $a_j^{\frac p2}\ne a_j^{\frac {p^*}2}$, there exists $c\in\R$ such that $\{g^*/f^*\le c\}$ is not homotopy equivalent to $\{f/g\le c\}$. 


For instance, take $p=n=3$, $a_1=2^4$, $a_2=3^4$ and $a_3=5^4$. Then $p^*=\frac32$, and 
$$1<a_1^{\frac34}=2^3<a_2^{\frac34}=3^3<a_1^{\frac32}=2^6<a_3^{\frac34}=5^3<a_2^{\frac32}=3^6<a_3^{\frac32}=5^6.$$
According to the discussion above, for any $c\in (2^6,5^3)=(a_1^{\frac32},a_3^{\frac34})$, $\{f/g\le c\}\simeq \mathbb{S}^{1-1}=\mathbb{S}^0$ while $\{g^*/f^*\le c\}\simeq \mathbb{S}^{2-1}=\mathbb{S}^1$, which means that $\{f/g\le c\}\not\simeq \{g^*/f^*\le c\}$.  
\end{example}

The next example shows that for DC functions, if we use the polarity dual rather than the Fenchel dual, then the critical dual equivalence will not hold. 

\begin{example}\label{exam:f-g-cannot-polar}
Let $f,g\in \cvx_0(\R^2)$ be defined as $f(x)=x_1^2+x_2^2$, and $g(x)=\sqrt{x_1^2+x_2^2}$. Then $g^\polar(x)=\sqrt{x_1^2+x_2^2}$ and $f^\polar(x)=\frac14(x_1^2+x_2^2)$. Consider $f-g$ and $g^\polar-f^\polar$. 
It is clear that, $f(x)-g(x)=x_1^2+x_2^2-\sqrt{x_1^2+x_2^2}\ge -\frac14$, while $g^\polar(x)-f^\polar(x)=\sqrt{x_1^2+x_2^2}-\frac14(x_1^2+x_2^2)\le1$. Therefore, the homotopy type of $\{f-g\le c\}$ and $\{g^\polar-f^\polar\le c\}$ are different for varying $c\in\R$. 

If we instead use Fenchel duality rather than polar duality, we can check $\{f-g\le c\}\simeq\{g^*-f^*\le c\}$ for any $c\in\R$. In fact, 
$g^*(x)-f^*(x)=\iota_{B_1}-\frac14(x_1^2+x_2^2)=+\infty$ if $x_1^2+x_2^2>1$ and $=-\frac14(x_1^2+x_2^2)$ if $x_1^2+x_2^2\le 1$.

For any $c\ge 0$, $\{f-g\le c\}$ must be  a closed disc, and $\{g^*-f^*\le c\}$ is always the unit closed disc. 

For any $-\frac14< c<0$, both $\{f-g\le c\}$ and $\{g^*-f^*\le c\}$ are ring zones. 

For $c=-\frac14$, both $\{f-g\le c\}$ and $\{g^*-f^*\le c\}$ are circles of dimension one.

For $c<-\frac14$, both $\{f-g\le c\}$ and $\{g^*-f^*\le c\}$ are empty sets.

Then, we complete the detailed verification on a concrete example for Theorem \ref{th:sublevel-equi}.
\end{example}

{\small
\noindent Dong Zhang 
\vskip 2pt
\noindent School of Mathematical Sciences, Peking University,  100871 Beijing, China.\vskip 2pt
\noindent Email: \textbf{dongzhang@math.pku.edu.cn}

\end{document}